\documentclass[oneside,british]{amsart}
\usepackage[T1]{fontenc}
\usepackage[latin9]{inputenc}
\usepackage{amstext}
\usepackage{amsthm}
\usepackage{amssymb}

\makeatletter
\numberwithin{equation}{section}
\numberwithin{figure}{section}
\theoremstyle{plain}
\newtheorem{thm}{\protect\theoremname}[section]
\theoremstyle{definition}
\newtheorem{defn}[thm]{\protect\definitionname}
\theoremstyle{plain}
\newtheorem{cor}[thm]{\protect\corollaryname}
\theoremstyle{definition}
\newtheorem{example}[thm]{\protect\examplename}
\theoremstyle{plain}
\newtheorem{lem}[thm]{\protect\lemmaname}
\theoremstyle{plain}
\newtheorem{prop}[thm]{\protect\propositionname}
\theoremstyle{plain}
\newtheorem*{thm*}{\protect\theoremname}

\usepackage[pdftex,pdfpagelabels,bookmarks,hyperindex,hyperfigures]{hyperref}

\makeatother

\usepackage{babel}
\providecommand{\corollaryname}{Corollary}
\providecommand{\definitionname}{Definition}
\providecommand{\examplename}{Example}
\providecommand{\lemmaname}{Lemma}
\providecommand{\propositionname}{Proposition}
\providecommand{\theoremname}{Theorem}

\begin{document}

\title{On the Rajchman property for self-similar measures on $\mathbb{R}^{d}$}

\author{Ariel Rapaport}


\begin{abstract}
We establish a complete algebraic characterization of self-similar
iterated function systems $\Phi$ on $\mathbb{R}^{d}$, for which
there exists a positive probability vector $p$ so that the Fourier
transform of the self-similar measure corresponding to $\Phi$ and
$p$ does not tend to $0$ at infinity.
\end{abstract}

\maketitle

\section{\label{sec:Introduction-and-the}Introduction and the main result}

\subsection{Introduction}

Let $d\ge1$ be an integer. Given a Borel probability measure $\nu$
on $\mathbb{R}^{d}$ its Fourier transform is denoted by $\widehat{\nu}$.
That is,
\[
\widehat{\nu}(\xi):=\int e^{i\left\langle \xi,x\right\rangle }\:d\nu(x)\text{ for }\xi\in\mathbb{R}^{d},
\]
where $\left\langle \cdot,\cdot\right\rangle $ is the standard inner
product of $\mathbb{R}^{d}$. It is said that $\nu$ is a Rajchman
measure if $|\widehat{\nu}(\xi)|\rightarrow0$ as $|\xi|\rightarrow\infty$.
The Riemann--Lebesgue lemma says that $\nu$ is Rajchman whenever
it is absolutely continuous with respect to the Lebesgue measure.
For singular measures determining which ones are Rajchman is a subtle
question with a long history (see \cite{Ly}). In this paper we study
the Rajchman property in the context of self-similar measures on $\mathbb{R}^{d}$.

Denote the orthogonal group of $\mathbb{R}^{d}$ by $O(d)$. A similarity
of $\mathbb{R}^{d}$ is a map $\varphi:\mathbb{R}^{d}\rightarrow\mathbb{R}^{d}$
of the form $\varphi(x)=rUx+a$, where $r>0$, $U\in O(d)$ and $a\in\mathbb{R}^{d}$.
When $0<r<1$, the map $\varphi$ is said to be a contracting similarity.
A finite collection $\Phi=\{\varphi_{i}\}_{i=1}^{\ell}$ of contracting
similarities is called a self-similar iterated function system (IFS)
on $\mathbb{R}^{d}$. It is well known (see \cite{Hut}) that there
exists a unique nonempty compact $K\subset\mathbb{R}^{d}$ which satisfies
the relation
\[
K=\cup_{i=1}^{\ell}\varphi_{i}(K)\:.
\]
It is called the self-similar set, or attractor, corresponding to
$\Phi$.

Following \cite{Ho} we make the following definition.
\begin{defn}
We say that $\Phi$ is affinely irreducible if there does not exist
a proper affine subspace $\mathbb{V}$ of $\mathbb{R}^{d}$ so that
$\varphi_{i}(\mathbb{V})=\mathbb{V}$ for all $1\le i\le\ell$.
\end{defn}

It is easy to see that $\Phi$ is not affinely irreducible if and
only if its attractor $K$ is contained in a proper affine subspace
$\mathbb{V}$ of $\mathbb{R}^{d}$. Observe that when $d=1$, the
IFS $\Phi$ is affinely irreducible if and only if the maps in $\Phi$
do not all have the same fixed point.

It is also well known (again, see \cite{Hut}) that given a probability
vector $p=(p_{i})_{i=1}^{\ell}$ there exists a unique Borel probability
measure $\mu$ on $\mathbb{R}^{d}$ which satisfies the relation,
\[
\mu=\sum_{i=1}^{\ell}p_{i}\cdot\varphi_{i}\mu,
\]
where $\varphi_{i}\mu:=\mu\circ\varphi_{i}^{-1}$ it the pushforward
of $\mu$ via $\varphi_{i}$. The measure $\mu$ is called the self-similar
measure corresponding to $\Phi$ and $p$, and it is supported on
the attractor $K$. If $p_{i}>0$ for all $1\le i\le\ell$ we say
that $p$ is positive and write $p>0$. When $p>0$ the support of
$\mu$ is equal to $K$. If $\Phi$ is not affinely irreducible then
$\mu(\mathbb{V})=1$ for some proper affine subspace $\mathbb{V}\subset\mathbb{R}^{d}$,
in which case it is easy to see that $\mu$ is not Rajchman. For this
reason, we shall always assume that our function systems are affinely
irreducible.

Before stating our main theorem we mention some relevant previous
results, mainly regarding the Fourier decay of self-similar measures
on $\mathbb{R}$. We start with the basic case of Bernoulli convolutions.
Given $\lambda\in(0,1)$, write $\nu_{\lambda}$ for the distribution
of the random sum $\sum_{n\ge0}\pm\lambda^{n}$, where the $\pm$
are independent unbiased random variables. This measure is called
the Bernoulli convolution with parameter $\lambda$. It can also be
realised as the self-similar measure corresponding to the IFS $\{t\rightarrow\lambda t\pm1\}$
and the probability vector $(\frac{1}{2},\frac{1}{2})$.

Erd\H{o}s \cite{Er1} proved that $\nu_{\lambda}$ is not Rajchman
whenever $\lambda^{-1}$ is a Pisot number different from $2$. Recall
that a Pisot number, also called a Pisot--Vijayaraghavan number or
a P.V. number, is an algebraic integer greater than one whose algebraic
(Galois) conjugates are all less than one in modulus. Note that $\nu_{1/2}$
is absolutely continuous and in particular Rajchman. Later Salem \cite{Sa}
showed that if $\lambda^{-1}$ is not a Pisot number then $\nu_{\lambda}$
is Rajchman, thus providing a characterization of Rajchman Bernoulli
convolution measures. Erd\H{o}s \cite{Er2} proved that $\widehat{\nu_{\lambda}}$
has power decay for a.e. $\lambda\in(0,1)$. That is, there exist
$s>0$ and $C>1$ so that $|\widehat{\nu_{\lambda}}(\xi)|\le C|\xi|^{-s}$
for $\xi\in\mathbb{R}$. Kahane \cite{Kah} later observed that this
actually holds for all $\lambda\in(0,1)$ outside a set of zero Hausdorff
dimension.

We turn to discuss the Fourier decay of general orientation preserving
self-similar measures on the real line, in which case a lot of recent
progress has been made. Let $\Phi=\{\varphi_{i}(t)=r_{i}t+a_{i}\}_{i=1}^{\ell}$
be a self-similar IFS on $\mathbb{R}$, with $r_{i}>0$ for $1\le i\le\ell$.
Set
\[
\Delta:=\{(p_{i})_{i=1}^{\ell}\in(0,1]^{\ell}\::\:p_{1}+...p_{\ell}=1\},
\]
and for $p\in\Delta$ write $\mu_{p}$ for the self-similar measure
corresponding to $\Phi$ and $p$. Let $\mathbf{H}\subset\mathbb{R}_{>0}$
be the group generated by the contractions $\{r_{i}\}_{i=1}^{\ell}$,
where $\mathbb{R}_{>0}$ is the multiplicative group of positive real
numbers. It is desirable to characterize the systems $\Phi$ for which
there exists $p\in\Delta$ so that $\mu_{p}$ is non-Rajchman. The
following result, due to Li and Sahlsten, reduces this problem to
the case in which $\mathbf{H}$ is cyclic.
\begin{thm}
[Li--Sahlsten, \cite{LS}]\label{thm:li and sahl}Suppose that $\Phi$
is affinely irreducible and that $\mathbf{H}$ is not cyclic. Then
$\mu_{p}$ is Rajchman for every $p\in\Delta$.
\end{thm}

A related result has recently been obtained by Algom, Rodriguez Hertz
and Wang \cite[Corollary 1.2]{AHW}, which verifies the Rajchman property
for self-conformal measures under mild assumptions. In \cite{SS},
Sahlsten and Stevens have established power Fourier decay for self-conformal
measures under certain conditions.

The proof of Theorem \ref{thm:li and sahl} is based on the classical
renewal theorem for transient random walks on $\mathbb{R}$. This
approach was initiated by Li \cite{Li}, who established the Rajchman
property for the Furstenberg measure on $\mathbb{RP}^{1}$ under mild
assumptions. Renewal theory also plays a major role in the proof of
the main result of this paper.

The situation in which $\mathbf{H}$ is cyclic has been considered
by Brémont (see also the paper by Varj{\'u} and Yu \cite{VY}). We
continue to consider the orientation preserving system $\Phi=\{\varphi_{i}(t)=r_{i}t+a_{i}\}_{i=1}^{\ell}$
on $\mathbb{R}$.
\begin{thm}
[Brémont, \cite{Br}]\label{thm:bermont}Suppose that $\Phi$ is affinely
irreducible and that $\mathbf{H}$ is cyclic. Let $r\in(0,1)$ be
with $\mathbf{H}=\{r^{n}\}_{n\in\mathbb{Z}}$. Then there exists $p\in\Delta$
so that $\mu_{p}$ is non-Rajchman if and only if, $r^{-1}$ is a
Pisot number and $\Phi$ can be conjugated by a suitable similarity
to a form such that $a_{i}\in\mathbb{Q}(r)$ for $1\le i\le\ell$.
\end{thm}

Brémont also proved that when $\mathbf{H}=\{r^{n}\}_{n\in\mathbb{Z}}$
for a Pisot number $r$ and $a_{i}\in\mathbb{Q}(r)$ for $1\le i\le\ell$,
then in fact $\mu_{p}$ is non-Rajchman for every $p\in\Delta$ outside
a finite union of proper submanifolds of $\Delta$. Moreover, in this
case he also showed that $\mu_{p}$ is absolutely continuous whenever
it is Rajchman.

Theorems \ref{thm:li and sahl} and \ref{thm:bermont} provide a complete
algebraic characterization of the systems $\Phi$ on $\mathbb{R}$
for which there exists $p\in\Delta$ so that $\mu_{p}$ is non-Rajchman.
The purpose of this paper is to extend this characterization to arbitrary
self-similar IFSs on $\mathbb{R}^{d}$.

We point out that Solomyak \cite{So} has recently shown that there
exists $\mathcal{E}\subset(0,1)^{\ell}$ of zero Hausdorff dimension
so that when $\Phi$ is affinely irreducible and $(r_{i})_{i=1}^{\ell}\notin\mathcal{E}$,
it holds that $\widehat{\mu_{p}}$ has power decay for all $p\in\Delta$.
For explicit parameters, and under additional diophantine assumptions,
logarithmic decay rate has recently been obtained in \cite{LS}, \cite{VY}
and \cite{AHW}. In the present paper we are only interested in the
complete characterization of self-similar IFSs generating non-Rajchman
measures, and do not consider the Fourier rate of decay.

Finally, we mention that in the context of self-affine measures on
$\mathbb{R}^{d}$, the Rajchman property has recently been considered
by Li and Sahlsten \cite{LS2}. Assuming the group generated by the
linear parts of the affine maps is proximal and totally irreducible
and that the attractor is not a singleton, they have established that
all self-affine measures, corresponding to positive probability vectors,
are Rajchman. When $d=2,3$, or under additional assumptions on the
group generated by the linear parts, they have also obtained power
Fourier decay. The proximality assumption makes the situation studied
in that paper very different compared to the self-similar setup studied
here.

\subsection{\label{subsec:The-main-result.}The main result}

Following \cite{Ca} and \cite[Section 9.2]{BDGPS} we make the following
definition, which is necessary in order to state our main result.
\begin{defn}
Given $k\ge1$, a finite collection $\{\theta_{1},...,\theta_{k}\}$
of distinct algebraic integers is said to be a P.V. $k$-tuple if
the following conditions are satisfied.
\begin{enumerate}
\item $|\theta_{j}|>1$ for $1\le j\le k$;
\item there exists a monic polynomial $P\in\mathbb{Z}[X]$ so that $P(\theta_{j})=0$
for $1\le j\le k$, and $|z|<1$ for $z\in\mathbb{C}\setminus\{\theta_{1},...,\theta_{k}\}$
with $P(z)=0$.
\end{enumerate}
\end{defn}

We make some remarks regarding this definition. In what follows, let
$\{\theta_{1},...,\theta_{k}\}$ be a P.V. $k$-tuple.
\begin{itemize}
\item For each $1\le j_{0}\le k$ there exists $1\le j_{1}\le k$ so that
$\theta_{j_{1}}=\overline{\theta_{j_{0}}}$. Additionally, writing
$J$ for the set of $1\le j\le k$ so that $\theta_{j}$ is conjugate
to $\theta_{j_{0}}$ over $\mathbb{Q}$, it holds that $\{\theta_{j}\}_{j\in J}$
is also a P.V. tuple.
\item Note that a positive real number $\theta$ is a Pisot number precisely
when $\{\theta\}$ is a P.V. $1$-tuple. A nonreal complex number
$\theta$ so that $\{\theta,\overline{\theta}\}$ is a P.V. $2$-tuple
is commonly called a complex Pisot number.
\item We shall be interested in P.V. tuples whose elements have the same
modulus. Obviously every P.V. $1$-tuple and every P.V. $2$-tuple
of the form $\{\theta,\overline{\theta}\}$ has this property. Assuming
$|\theta_{1}|=...=|\theta_{k}|$, for every $m\ge1$ the collection
\[
\{z\in\mathbb{C}\::\:z^{m}=\theta_{j}\text{ for some }1\le j\le k\}
\]
is a P.V. $mk$-tuple with this property. Further examples can be
obtained by considering the products of real or complex Pisot numbers
with certain primitive roots of unity. For instance, as pointed out
in \cite{Bo}, if $\theta$ and $\overline{\theta}$ are the complex
Pisot numbers whose minimal polynomial is $X^{3}+X^{2}-1$ and $u$
and $\overline{u}$ are the primitive $6\text{th}$ roots of unity,
then $\{\theta u,\overline{\theta}u,\theta\overline{u},\overline{\theta u}\}$
is a P.V. $4$-tuple whose elements are all conjugates over $\mathbb{Q}$
and have the same modulus.
\item Suppose that $|\theta_{1}|=...=|\theta_{k}|$ and that $\theta_{1},...,\theta_{k}$
are conjugates over $\mathbb{Q}$. It is natural to ask whether we
can say more about the structure of the P.V. $k$-tuple $\{\theta_{1},...,\theta_{k}\}$
under these additional assumptions. From a result of Boyd \cite{Bo}
and Ferguson \cite{Fer} it follows that if $\theta_{j}$ is real
for some $1\le j\le k$ then $\{\theta_{1},...,\theta_{k}\}=\{e^{2\pi ij/k}\theta_{1}\}_{j=1}^{k}$.
Considering this result and the previous remark, one might think that,
under the additional assumptions, for every $1\le j\le k$ at least
one of the numbers $\theta_{j}/\theta_{1}$ and $\theta_{j}/\overline{\theta_{1}}$
is a root of unity. In Example \ref{exa:interesting} below we show
that this is not the case.
\end{itemize}
\smallskip{}
$\quad$As noted above, we shall always assume that our function systems
are affinely irreducible. On the other hand, the situation in which
there exists a nontrivial linearly invariant subspace should be taken
into account. Let $\Phi=\{\varphi_{i}(x)=r_{i}Ux+a_{i}\}_{i=1}^{\ell}$
be a self-similar IFS on $\mathbb{R}^{d}$. Given a linear subspace
$\mathbb{V}$ of $\mathbb{R}^{d}$ we write $\pi_{\mathbb{V}}$ for
the orthogonal projection onto $\mathbb{V}$. Observe that if $d':=\dim\mathbb{V}>0$,
$U_{i}(\mathbb{V})=\mathbb{V}$ for $1\le i\le\ell$, and $S:\mathbb{V}\rightarrow\mathbb{R}^{d'}$
is an isometry (which is necessarily an affine map), then $\{S\circ\pi_{\mathbb{V}}\circ\varphi_{i}\circ S^{-1}\}_{i=1}^{\ell}$
is a self-similar IFS on $\mathbb{R}^{d'}$. Moreover, in this situation
for every $1\le i\le\ell$ there exists $U_{i}'\in O(d')$ and $a_{i}'\in\mathbb{R}^{d'}$
so that
\[
S\circ\pi_{\mathbb{V}}\circ\varphi_{i}\circ S^{-1}(x)=r_{i}U_{i}'x+a_{i}'\text{ for }x\in\mathbb{R}^{d'}\:.
\]

We are now ready to state the main result of this paper. In what follows
we consider $\mathbb{R}^{d}$ as a subset of $\mathbb{C}^{d}$. We
denote the standard inner product of $\mathbb{C}^{d}$ by $\left\langle \cdot,\cdot\right\rangle $,
that is $\left\langle z,w\right\rangle =\sum_{j=1}^{d}z_{j}\overline{w_{j}}$
for $z,w\in\mathbb{C}^{d}$. Given a linear operator $A$ on $\mathbb{R}^{d}$
we consider it also as a linear operator on $\mathbb{C}^{d}$ in the
natural way, that is by setting $A(x+iy):=Ax+iAy$ for $x,y\in\mathbb{R}^{d}$.
\begin{thm}
\label{thm:main}Let $\Phi=\{\varphi_{i}(x)=r_{i}U_{i}x+a_{i}\}_{i=1}^{\ell}$
be an affinely irreducible self-similar IFS on $\mathbb{R}^{d}$,
with $0<r_{i}<1$, $U_{i}\in O(d)$ and $a_{i}\in\mathbb{R}^{d}$
for $1\le i\le\ell$. Then there exists a probability vector $p=(p_{i})_{i=1}^{\ell}>0$
such that the self-similar measure corresponding to $\Phi$ and $p$
is non-Rajchman if and only if there exists a linear subspace $\mathbb{V}\subset\mathbb{R}^{d}$,
with $d':=\dim\mathbb{V}>0$ and $U_{i}(\mathbb{V})=\mathbb{V}$ for
$1\le i\le\ell$, and an isometry $S:\mathbb{V}\rightarrow\mathbb{R}^{d'}$
so that the following conditions are satisfied.
\begin{enumerate}
\item \label{enu:main thm first cond}For $1\le i\le\ell$ let $U_{i}'\in O(d')$
and $a_{i}'\in\mathbb{R}^{d'}$ be with $S\circ\pi_{\mathbb{V}}\circ\varphi_{i}\circ S^{-1}(x)=r_{i}U_{i}'x+a_{i}'$.
Let $\mathbf{H}\subset GL_{d'}(\mathbb{R})$ be the group generated
by $\{r_{i}U_{i}'\}_{i=1}^{\ell}$, and set $\mathbf{N}:=\mathbf{H}\cap O(d')$.
Then $\mathbf{N}$ is finite, $\mathbf{N}\triangleleft\mathbf{H}$
and $\mathbf{H}/\mathbf{N}$ is cyclic.
\item \label{enu:main thm second cond}For every contracting $A\in\mathbf{H}$
with $\{A^{n}\mathbf{N}\}_{n\in\mathbb{Z}}=\mathbf{H}/\mathbf{N}$,
there exist $k\ge1$, $\theta_{1},...,\theta_{k}\in\mathbb{C}$ and
$\zeta_{1},...,\zeta_{k}\in\mathbb{C}^{d'}\setminus\{0\}$, so that
\begin{enumerate}
\item $\{\theta_{1},...,\theta_{k}\}$ is a P.V. $k$-tuple;
\item $A^{-1}\zeta_{j}=\theta_{j}\zeta_{j}$ for $1\le j\le k$;
\item for every $1\le i\le\ell$ and $V\in\mathbf{N}$ there exists $P_{i,V}\in\mathbb{Q}[X]$
so that $\left\langle Va_{i}',\zeta_{j}\right\rangle =P_{i,V}(\theta_{j})$
for $1\le j\le k$.
\end{enumerate}
\end{enumerate}
\end{thm}

We make some remarks regarding the theorem.
\begin{itemize}
\item When $d=1$ and the system $\Phi$ is orientation preserving, Theorem
\ref{thm:main} is easily seen to be equivalent to Theorems \ref{thm:li and sahl}
and \ref{thm:bermont}.
\item As the proof will show, if condition (\ref{enu:main thm first cond})
holds and condition (\ref{enu:main thm second cond}) is satisfied
for some contracting $A\in\mathbf{H}$ with $\{A^{n}\mathbf{N}\}_{n\in\mathbb{Z}}=\mathbf{H}/\mathbf{N}$,
then there exists a probability vector $p=(p_{i})_{i=1}^{\ell}>0$
so that the corresponding self-similar measure is non-Rajchman.
\item In condition (\ref{enu:main thm first cond}), since $\mathbf{N}$
is the kernel of the homomorphism sending $rU\in\mathbf{H}$ with
$r>0$ and $U\in O(d')$ to $r$, it is obvious that $\mathbf{N}\triangleleft\mathbf{H}$.
The statements regarding the finiteness of $\mathbf{N}$ and $\mathbf{H}/\mathbf{N}$
being cyclic are the interesting part of this condition.
\item In condition (\ref{enu:main thm second cond}), note that by restricting
to a suitable nonempty subset of $\{\theta_{1},...,\theta_{k}\}$
we may assume that $\theta_{1},...,\theta_{k}$ are conjugates over
$\mathbb{Q}$.
\item As we show in Example \ref{exa:dep on A} below, the parameters $k$
and $\theta_{1},...,\theta_{k}$ in condition (\ref{enu:main thm second cond})
may depend on the choice of $A$. On the other hand, it is not hard
to show that if conditions (\ref{enu:main thm first cond}) and (\ref{enu:main thm second cond})
are satisfied and $A_{1},A_{2}\in\mathbf{H}$ are contractions with
$\{A_{i}^{n}\mathbf{N}\}_{n\in\mathbb{Z}}=\mathbf{H}/\mathbf{N}$
for $i=1,2$, then for every eigenvalue $\theta$ of $A_{1}$ there
exists a root of unity $u$ so that $u\theta$ is an eigenvalue of
$A_{2}$.
\item In condition (\ref{enu:main thm second cond}), since $A^{-1}$ is
a member of $\mathbf{H}$ all of its eigenvalues have the same modulus.
In particular $|\theta_{1}|=...=|\theta_{k}|$.
\item Theorem \ref{thm:main} provides many explicit examples of affinely
irreducible self-similar function systems for which there exists a
positive probability vector so that the corresponding self-similar
measure is non-Rajchman. In fact, for every $k\ge1$ and $\theta_{1},...,\theta_{k}\in\mathbb{C}$
such that $\{\theta_{1},...,\theta_{k}\}$ is a P.V. $k$-tuple and
$|\theta_{1}|=...=|\theta_{k}|$, we can construct a corresponding
self-similar IFS on $\mathbb{R}^{k}$ with these properties (see Example
\ref{exa:many examples} below).
\item For a system $\Phi$ satisfying conditions (\ref{enu:main thm first cond})
and (\ref{enu:main thm second cond}), it could be interesting to
study the exceptional set of positive probability vectors for which
the corresponding self-similar measure is Rajchman. As noted after
the statement of Theorem \ref{thm:bermont}, this has been carried
out by Brémont \cite{Br} in the case of orientation preserving systems
on $\mathbb{R}$.
\end{itemize}
\smallskip{}
$\quad$Theorem \ref{thm:main} can be used to verify the Rajchman
property in many situations. For instance we have the following simple
corollary.
\begin{cor}
Let $\Phi=\{\varphi_{i}(x)=r_{i}U_{i}x+a_{i}\}_{i=1}^{\ell}$ be an
affinely irreducible self-similar IFS on $\mathbb{R}^{d}$. Suppose
that there exists a probability vector $p=(p_{i})_{i=1}^{\ell}>0$
such that the self-similar measure corresponding to $\Phi$ and $p$
is non-Rajchman. Then there exists an algebraic integer $\theta>1$
so that for every $1\le i\le\ell$ there exists a rational integer
$n_{i}\ge1$ with $r_{i}=\theta^{-n_{i}}$.
\end{cor}

\begin{proof}
By Theorem \ref{thm:main} there exist $\mathbb{V}$ and $S$ as in
the statement of the theorem. Let $A$ and $\theta_{1}$ be as in
condition (\ref{enu:main thm second cond}), and note that $A=rU$
for some $0<r<1$ and $U\in O(d')$. Since $\theta_{1}^{-1}$ is an
eigenvalue of $A$ we have $\Vert A\Vert=|\theta_{1}|^{-1}$, where
$\Vert\cdot\Vert$ is the operator norm. From $\{A^{n}\mathbf{N}\}_{n\in\mathbb{Z}}=\mathbf{H}/\mathbf{N}$
and $\Vert A\Vert<1$, it follows that for $1\le i\le\ell$ there
exists $n_{i}\ge1$ and $V_{i}\in\mathbf{N}$ so that $r_{i}U_{i}'=A^{n_{i}}V_{i}$.
Thus,
\[
r_{i}=\Vert r_{i}U_{i}'\Vert=\Vert A^{n_{i}}V_{i}\Vert=\Vert A\Vert^{n_{i}}=|\theta_{1}|^{-n_{i}}\:.
\]
Since $\theta_{1}$ is a member of a P.V. $k$-tuple it is an algebraic
integer. Since the modulus of an algebraic integer is still an algebraic
integer, this completes the proof of the corollary.
\end{proof}

\subsection{Examples}
\begin{example}
\label{exa:many examples}Let $k\ge1$ and $\theta_{1},...,\theta_{k}\in\mathbb{C}$
be such that $\{\theta_{1},...,\theta_{k}\}$ is a P.V. $k$-tuple
and $|\theta_{1}|=...=|\theta_{k}|$. In this example we show that
it is possible to construct an affinely irreducible self-similar IFS
$\Phi$ on $\mathbb{R}^{k}$ so that conditions (\ref{enu:main thm first cond})
and (\ref{enu:main thm second cond}) in Theorem \ref{thm:main} are
satisfied with the parameters $k$ and $\theta_{1},...,\theta_{k}$,
where we take $\mathbb{V}=\mathbb{R}^{k}$ and $S=Id$.

Since $\{\theta_{1},...,\theta_{k}\}$ is a P.V. $k$-tuple, for every
$1\le j_{1}\le k$ there exists $1\le j_{2}\le k$ so that $\theta_{j_{2}}=\overline{\theta_{j_{1}}}$.
Thus, there exists $A\in GL_{k}(\mathbb{R})$ so that $\theta_{1},...,\theta_{k}$
are the eigenvalues of $A^{-1}$ and $A=rU$ for some $0<r<1$ and
$U\in O(k)$. Let $\zeta_{1},...,\zeta_{k}\in\mathbb{C}^{k}$ be such
that $\{\zeta_{1},...,\zeta_{k}\}$ is an orthonormal basis for $\mathbb{C}^{k}$,
$A^{-1}\zeta_{j}=\theta_{j}\zeta_{j}$ for $1\le j\le k$, and $\zeta_{j_{2}}=\overline{\zeta_{j_{1}}}$
for $1\le j_{1},j_{2}\le k$ with $\theta_{j_{2}}=\overline{\theta_{j_{1}}}$.
Set $\xi:=\sum_{j=1}^{k}\zeta_{k}$, so that $\xi\in\mathbb{R}^{k}$.
Let $\Phi:=\{\varphi_{i}\}_{i=0}^{k}$ be the self-similar IFS on
$\mathbb{R}^{k}$ with $\varphi_{0}(x)=Ax$ and $\varphi_{i}(x)=Ax+A^{1-i}\xi$
for $1\le i\le k$.

Let us show that $\Phi$ is affinely irreducible. Denote the attractor
of $\Phi$ by $K$. Let $y_{0}$ be the zero vector of $\mathbb{R}^{k}$,
and for $1\le i\le k$ write $y_{i}:=(I-A)^{-1}A^{1-i}\xi$. For $0\le i\le k$
we have $\varphi_{i}(y_{i})=y_{i}$, and so $y_{0},...,y_{k}\in K$.
The matrix $(\left\langle A^{1-i}\xi,\zeta_{j}\right\rangle )_{i,j=1}^{k}$
is equal to the Vandermonde matrix $\{\theta_{j}^{i-1}\}_{i,j=1}^{k}$,
and so its determinant is nonzero. It follows that $\{A^{1-i}\xi\}_{i=1}^{k}$
are linearly independent, and so $\{y_{i}\}_{i=1}^{k}$ are also linearly
independent. This shows that the affine span of $K$ is equal to $\mathbb{R}^{k}$,
which implies that $\Phi$ is affinely irreducible.

It is obvious that condition (\ref{enu:main thm first cond}) in Theorem
\ref{thm:main} is satisfied with $\mathbf{N}=\{Id\}$. Moreover,
for $1\le i,j\le k$ we have $\left\langle \varphi_{i}(0),\zeta_{j}\right\rangle =\theta_{j}^{i-1}$.
From this, and since $\left\langle \varphi_{0}(0),\zeta_{j}\right\rangle =0$
for $1\le j\le k$, it follows that condition (\ref{enu:main thm second cond})
is also satisfied. From Theorem \ref{thm:main} we now get that there
exists a probability vector $p=(p_{i})_{i=1}^{k}>0$ so that the self-similar
measure corresponding to $\Phi$ and $p$ is non-Rajchman.
\end{example}

\begin{example}
\label{exa:dep on A}The purpose of this example is to show that the
parameters $k$ and $\theta_{1},...,\theta_{k}$, appearing in condition
(\ref{enu:main thm second cond}) in Theorem \ref{thm:main}, may
depend on the choice of $A$. Set $r_{1}=r_{2}=1/2$, let $U_{1}$
be the identity map of $\mathbb{R}^{2}$, let $U_{2}\in O(2)$ be
a planar rotation of angle $\pi/2$ (i.e. $U_{2}(x_{1},x_{2})=(-x_{2},x_{1})$),
set $a_{1}=(1,0)$ and $a_{2}=0$, and set $\varphi_{i}(x)=r_{i}U_{i}x+a_{i}$
for $i=1,2$ and $x\in\mathbb{R}^{2}$. It is easy to verify that
the IFS $\Phi:=\{\varphi_{1},\varphi_{2}\}$ is affinely irreducible.

Let $\mathbf{H}\subset GL_{2}(\mathbb{R})$ be the group generated
by $r_{1}U_{1}$ and $r_{2}U_{2}$, and set $\mathbf{N}:=\mathbf{H}\cap O(2)$.
We have $\mathbf{N}=\{U_{2}^{l}\}_{l=1}^{4}$, and for $A_{1}:=r_{1}U_{1}$
and $A_{2}:=r_{2}U_{2}$ it holds that $\{A_{1}^{n}\mathbf{N}\}_{n\in\mathbb{Z}}=\{A_{2}^{n}\mathbf{N}\}_{n\in\mathbb{Z}}=\mathbf{H}/\mathbf{N}$.
Thus, condition (\ref{enu:main thm first cond}) of Theorem \ref{thm:main}
is satisfied (with $\mathbb{V}=\mathbb{R}^{2}$ and $S=Id$). It is
also easy to verify that if we take $k=1$, $\theta_{1}=2$ and $\zeta_{1}=(1,0)$
then condition (\ref{enu:main thm second cond}) holds with respect
to $A_{1}$, and if we take $k=2$, $\theta_{1}=2i$, $\theta_{2}=-2i$,
$\zeta_{1}=(1,i)$ and $\zeta_{2}=(1,-i)$ then condition (\ref{enu:main thm second cond})
holds with respect to $A_{2}$. Moreover, since $2i$ and $-2i$ are
conjugates over $\mathbb{Q}$, with $k=1$ condition (\ref{enu:main thm second cond})
cannot hold with respect to $A_{2}$. This shows that the parameters
$k$ and $\theta_{1},...,\theta_{k}$ depend on the choice of $A$.
\end{example}

\begin{example}
\label{exa:interesting}The purpose of this example is to construct
a P.V $k$-tuple $\{\theta_{1},...,\theta_{k}\}$ such that $k\ge3$,
$\theta_{1},...,\theta_{k}$ are all conjugates over $\mathbb{Q}$,
$|\theta_{1}|=...=|\theta_{k}|$, and for every $1\le j_{1}<j_{2}\le k$
the number $\theta_{j_{1}}\theta_{j_{2}}^{-1}$ is not a root of unity.

A polynomial $P\in\mathbb{Z}[X]$ of degree $n$ is said to be reciprocal
if $P(X)=X^{n}P(X^{-1})$. In this case the roots of $P$ fall into
reciprocal pairs, that is $z^{-1}$ is a root of $P$ whenever $z\in\mathbb{C}$
is a root of $P$. We say that $P\in\mathbb{Z}[X]$ is a Salem polynomial
if it is the minimal polynomial of a Salem number. This means that
$P$ is irreducible, monic, reciprocal, it has degree at least $4$,
there exists $s>1$ with $P(s)=0$, and $|z|=1$ for every $z\in\mathbb{C}\setminus\{s,s^{-1}\}$
with $P(z)=0$. The number $s$ is called a Salem number.

Let $m\ge4$ be even, and let $P$ be a Salem polynomial of degree
$2m$. For example, we can take $P(X)$ to be $X^{8}-X^{5}-X^{4}-X^{3}+1$.
Let $z_{1}>1$ be the Salem number corresponding to $P$, and let
$z_{2},...,z_{m}$ be the roots of $P$ located on the upper half
of the unit circle in $\mathbb{C}$. Set $I:=\{1,...,m\}$, and for
$J\subset I$ write $\theta_{J}:=\Pi_{j\in J}z_{j}\cdot\Pi_{j\in I\setminus J}z_{j}^{-1}$.
The Galois group of $P$ is analysed in \cite[Theorem 1.1]{CM}. From
that result it follows that $\{\theta_{J}\}_{J\subset I}$ is a complete
set of algebraic conjugates over $\mathbb{Q}$.

Let $\mathbb{F}\subset\mathbb{C}$ be the splitting field of $P$
over $\mathbb{Q}$. Note that given $j_{0}\in I$, $J_{1},J_{2}\subset I$
with $j_{0}\in J_{1}\setminus J_{2}$, and an automorphism $\sigma:\mathbb{F}\rightarrow\mathbb{F}$
with $\sigma(z_{j_{0}})=z_{1}$, we have $|\sigma(\theta_{J_{1}})|=z_{1}>1$
and $|\sigma(\theta_{J_{2}})|=z_{1}^{-1}<1$. Since $\sigma(u)$ is
a root of unity whenever $u\in\mathbb{F}$ is a root of unity, it
follows that $\theta_{J_{1}}\ne\theta_{J_{2}}e^{2\pi iq}$ for all
distinct $J_{1},J_{2}\subset I$ and $q\in\mathbb{Q}$. This shows
that $\{\theta_{J}\::\:1\in J\subset I\}$ is a P.V. $2^{m-1}$-tuple,
and that it satisfies the required properties.
\end{example}

\subsection{\label{subsec:About-the-proof}About the proof}

Let $\Phi=\{\varphi_{i}(x)=r_{i}U_{i}x+a_{i}\}_{i=1}^{\ell}$ be an
affinely irreducible self-similar IFS on $\mathbb{R}^{d}$. Most of
the proof of Theorem \ref{thm:main} deals with the direction in which
$\Phi$ is assumed to generate a non-Rajchman measure. We present
a general outline of the argument for this direction. Everything will
be repeated in a rigorous manner in later parts of the paper.

Let $p=(p_{i})_{i=1}^{\ell}$ be a positive probability vector, let
$\mu$ be the self-similar measure corresponding to $\Phi$ and $p$,
and suppose that $\mu$ is non-Rajchman. Write $G\subset\mathbb{R}\times O(d)$
for the closed subgroup generated by $\{(\log r_{i}^{-1},U_{i})\}_{i=1}^{\ell}$.
For $(t,U)=g\in G$ set $\psi g=t$. Since $\psi$ is a proper continuous
map, $\psi(G)$ is a closed subgroup of $\mathbb{R}$. Let $\gamma:\psi(G)\rightarrow G$
be a continuous homomorphism with $\psi\circ\gamma=Id$. We define
a right action of $G$ on $\mathbb{R}^{d}$ by setting $x.(t,U):=2^{-t}U^{-1}x$
for $(t,U)\in G$ and $x\in\mathbb{R}^{d}$.

Let $X_{1},X_{2},...$ be i.i.d. $G$-valued random elements with
$\mathbb{P}\{X_{1}=(\log r_{i}^{-1},U_{i})\}=p_{i}$ for $1\le i\le\ell$,
and for $n\ge1$ set $Y_{n}:=X_{1}\cdot...\cdot X_{n}$. For $t>0$
denote by $\tau_{t}$ the stopping time which is equal to the smallest
$n\ge1$ for which $\psi Y_{n}\ge t$. Using a result obtained in
\cite{BDGHU}, which extends the classical renewal theorem, we show
that as $t\rightarrow\infty$ the random elements $\gamma_{-t}Y_{\tau_{t}}$
converge in distribution to a probability measure $\nu$ on $G$ which
is absolutely continuous with respect to the Haar measure of $G$.
This key fact will be used several times during the paper. In particular,
we use it to prove the following lemma.
\begin{lem}
\label{lem:ATP ini up bd}For every $\epsilon>0$ there exists $T>1$
such that the following holds. Let $t\ge T$ be with $t\in\psi(G)$
and let $\xi\in\mathbb{R}^{d}$ be with $|\xi|\le\epsilon^{-1}$,
then
\[
|\widehat{\mu}(\xi.\gamma_{-t})|^{2}\le\epsilon+\int\int\left|\int e^{i\left\langle \xi.g,x-y\right\rangle }d\nu(g)\right|\:d\mu(x)\:d\mu(y)\:.
\]
\end{lem}

This lemma is inspired by the argument in \cite{LS} used in the proof
of Theorem \ref{thm:li and sahl}. The lemma is only useful when the
group $G$ is nondiscrete.

As we show below, from the affine irreducibility of $\Phi$ it follows
that $\mu(\mathbb{V})=0$ for every proper affine subspace $\mathbb{V}$
of $\mathbb{R}^{d}$. Using this fact we prove the following lemma.
\begin{lem}
\label{lem:ATP ub int sig of fur of curve mass}For every $\epsilon>0$
there exists $S>1$ so that the following holds. Let $s\ge S$ and
let $c:[0,1]\rightarrow\mathbb{R}^{d}$ be a smooth curve with $|c'(t)|\ge\epsilon$
and $|c''(t)|\le\epsilon^{-1}$ for all $0\le t\le1$, then
\[
\int\int\left|\int_{0}^{1}e^{is\left\langle c(t),x-y\right\rangle }\:dt\right|\:d\mu(x)\:d\mu(y)<\epsilon\:.
\]
\end{lem}

Now it is not difficult to show that $\psi(G)\ne\mathbb{R}$. Indeed,
assuming this is not the case we can represent $\nu$ as an average
of smooth $1$-dimensional probability measures, each of which is
supported on a single coset of the subgroup $\gamma(\mathbb{R})$.
By using this decomposition together with Lemmata \ref{lem:ATP ini up bd}
and \ref{lem:ATP ub int sig of fur of curve mass}, we show that $\mu$
must be Rajchman which contradicts our assumption.

Next we want to make a reduction from the case in which $G$ is nondiscrete
with $\psi(G)\ne\mathbb{R}$, to the case in which $G$ is discrete.
For this we need to choose appropriately the subspace $\mathbb{V}\subset\mathbb{R}^{d}$
appearing in the statement of Theorem \ref{thm:main}. Denote by $G_{0}$
the connected component of $G$ containing the identity element. We
choose $\mathbb{V}$ to be the linear subspace consisting of all $x\in\mathbb{R}^{d}$
so that $x.g=x$ for all $g\in G_{0}$. From $G_{0}\triangleleft G$
it follows that $x.g\in\mathbb{V}$ for $x\in\mathbb{V}$ and $g\in G$,
which implies that $U_{i}(\mathbb{V})=\mathbb{V}$ for $1\le i\le\ell$.

Note that from $\psi(G)\ne\mathbb{R}$ it follows that the connected
Lie group $G_{0}$ is contained in the compact group $\{0\}\times O(d)$.
By using this fact, by representing $\nu$ as an average of certain
smooth $1$-dimensional measures, and by applying Lemmata \ref{lem:ATP ini up bd}
and \ref{lem:ATP ub int sig of fur of curve mass} once more, we prove
the following proposition. It will enable us to perform the aforementioned
reduction.
\begin{prop}
\label{prop:ATP reduc}For every $\epsilon>0$ there exists $R>1$
so that $|\widehat{\mu}(\xi)|<\epsilon$ for every $\xi\in\mathbb{R}^{d}$
with $|\pi_{\mathbb{V}^{\perp}}\xi|\ge\max\{R,\epsilon|\pi_{\mathbb{V}}\xi|\}$.
\end{prop}

Next we consider the case in which $G$ is discrete. Cleary $\psi(G)\ne\mathbb{R}$
in this case, and so $\psi(G)=\beta\mathbb{Z}$ for some $\beta>0$.
Let $U\in O(d)$ be with $(\beta,U)\in G$, and set $A=2^{-\beta}U$.
Under the additional technical assumption $a_{1}=0$, we show that
condition (\ref{enu:main thm second cond}) in the statement of Theorem
\ref{thm:main} holds for the matrix $A$. The proof is a nontrivial
extension of the argument used in \cite{VY} for the direction of
Theorem \ref{thm:bermont} in which the IFS is assumed to generate
a non-Rajchman measure. One of the main ingredients of that argument
is a classical theorem of Pisot. This theorem says that if $\theta>1$
and $0\ne\lambda\in\mathbb{R}$ satisfy $\sum_{n\ge0}\Vert\lambda\theta^{n}\Vert^{2}<\infty$,
where $\Vert\cdot\Vert$ is the distance to the nearest integer, then
$\theta$ is a Pisot number and $\lambda\in\mathbb{Q}(\theta)$. In
our proof we shall need to use a generalisation of this result for
P.V. $k$-tuples, which is basically contained in Pisot's original
paper \cite{Pi}.

Observe that from proposition \ref{prop:ATP reduc}, and since $\mu$
is not Rajchman, it follows that $d':=\dim\mathbb{V}>0$. Let $S:\mathbb{V}\rightarrow\mathbb{R}^{d'}$
be an isometry. By using Proposition \ref{prop:ATP reduc}, the fact
that $\mu$ is not Rajchman, and the self-similarity of $\mu$, we
can show that $S\pi_{\mathbb{V}}\mu$ is also not Rajchman. The measure
$S\pi_{\mathbb{V}}\mu$ is the self-similar measure corresponding
to the self-similar IFS $\Phi':=\{S\circ\pi_{\mathbb{V}}\circ\varphi_{i}\circ S^{-1}\}_{i=1}^{\ell}$
on $\mathbb{R}^{d'}$ and the probability vector $p$. Let $\mathbf{H}$
be the closed group generated by the linear parts of $\Phi'$, and
set $\mathbf{N}:=\mathbf{H}\cap O(d')$. By using our choice of $\mathbb{V}$,
it is not hard to show that $\mathbf{H}$ is discrete, $\mathbf{N}$
is finite and $\mathbf{H}/\mathbf{N}$ is cyclic. Moreover, we can
choose the isometry $S$ so that the technical assumption $a_{1}=0$
is satisfied for the IFS $\Phi'$. At this point we complete the proof
by applying on $\Phi'$ our result for the case in which $G$ is discrete.

\subsection*{Organisation of the paper}

In Section \ref{sec:Preliminaries} we develop notations and establish
some basic properties of the group $G$. Assuming irreducibility,
we also prove that self-similar measures vanish on proper affine subspaces.
In Section \ref{sec:Renewal-theory-and first hit} we state the version
of the renewal theorem for $G$, and derive the statement regarding
the limit distribution of $\gamma_{-t}Y_{\tau_{t}}$. Section \ref{sec:The-nondiscrete-case}
deals with the parts of the argument in which $G$ is assumed to be
nondiscrete. In Section \ref{sec:The-discrete-case} we consider the
case in which $G$ is discrete. In particular, in this section we
construct non-Rajchman self-similar measures when $G$ is discrete
and the IFS satisfies assumptions similar to condition (\ref{enu:main thm second cond})
in Theorem \ref{thm:main}. In Section \ref{sec:Proof-of-the main}
we connect all the pieces, and complete the proof of Theorem \ref{thm:main}.

\section{\label{sec:Preliminaries}Preliminaries}

\subsection{\label{subsec:General-notations}General notations}

For an integer $m$ we write $\mathbb{Z}_{\ge m}:=\{m,m+1,...\}$.
We use the notations $\mathbb{Z}_{>m}$, $\mathbb{Z}_{\le m}$ and
$\mathbb{Z}_{<m}$ in a similar way.

Let $d\in\mathbb{Z}_{\ge1}$ be fixed. We denote the standard inner
product of $\mathbb{R}^{d}$ or $\mathbb{C}^{d}$ by $\left\langle \cdot,\cdot\right\rangle $,
that is
\[
\left\langle z,w\right\rangle =\sum_{j=1}^{d}z_{j}\overline{w_{j}}\text{ for }z,w\in\mathbb{C}^{d}\:.
\]
For a linear subspace $\mathbb{V}$ of $\mathbb{R}^{d}$ or $\mathbb{C}^{d}$,
the orthogonal projection onto $\mathbb{V}$ is denoted by $\pi_{\mathbb{V}}$.
We write $\mathbb{V}^{\perp}$ for the orthogonal complement of $\mathbb{V}$.
The orthogonal group of $\mathbb{R}^{d}$ is denoted by $O(d)$. Given
a Borel probability measure $\sigma$ on $\mathbb{R}^{d}$ its Fourier
transform $\widehat{\sigma}$ is defined by
\[
\widehat{\sigma}(\xi):=\int e^{i\left\langle \xi,x\right\rangle }\:d\sigma(x)\text{ for }\xi\in\mathbb{R}^{d}\:.
\]

For a locally compact Hausdorff space $X$, we write $C_{c}(X)$ for
the space of continuous functions $f:X\rightarrow\mathbb{R}$ with
compact support. We denote by $\mathcal{M}(X)$ the collection of
all compactly supported Borel probability measures on $X$. If $Y$
is another topological space, $\sigma$ is a Borel measure on $X$,
and $F:X\rightarrow Y$ is Borel measurable, then we write $F\sigma$
for the pushforward of $\sigma$ via $F$. That is, $F\sigma:=\sigma\circ F^{-1}$.

Throughout the paper $\Phi=\{\varphi_{i}(x)=r_{i}U_{i}x+a_{i}\}_{i=1}^{\ell}$
is an affinely irreducible self-similar IFS on $\mathbb{R}^{d}$,
so that $0<r_{i}<1$, $U_{i}\in O(d)$ and $a_{i}\in\mathbb{R}^{d}$
for $1\le i\le\ell$. We consider $\Phi$ as fixed, and so usually
the dependence of various parameters on $\Phi$ will not be indicated.
We denote by $K$ the attractor of $\Phi$, that is $K$ is the unique
nonempty compact subset of $\mathbb{R}^{d}$ with
\[
K=\cup_{i=1}^{\ell}\varphi_{i}(K)\:.
\]
Given a probability vector $p=(p_{i})_{i=1}^{\ell}$ there exists
a unique $\mu\in\mathcal{M}(K)$ which satisfies the relation
\begin{equation}
\mu=\sum_{i=1}^{\ell}p_{i}\cdot\varphi_{i}\mu\:.\label{eq:SS relation}
\end{equation}
It is called the self-similar measure corresponding to $\Phi$ and
$p$. We usually assume that $p_{i}>0$ for $1\le i\le\ell$, in which
case we say that $p$ is positive and write $p>0$.

We sometimes write $\Lambda$ for the index set $\{1,...,\ell\}$,
and denote the set of finite words over $\Lambda$ by $\Lambda^{*}$.
Following \cite{BP}, we say that a finite set of words $\mathcal{W}\subset\Lambda^{*}$
is a minimal cut-set for $\Lambda^{*}$ if every infinite sequence
in $\Lambda^{\mathbb{N}}$ has a unique prefix in $\mathcal{W}$.
Given a group $Y$, indexed elements $\{y_{i}\}_{i=1}^{\ell}\subset Y$,
and a word $i_{1}...i_{n}=w\in\Lambda^{*}$, we often write $y_{w}$
in place of $y_{i_{1}}\cdot...\cdot y_{i_{n}}$. For the empty word
$\emptyset$ we write $y_{\emptyset}$ in place of $1_{Y}$, where
$1_{Y}$ is the identity of $Y$. Note that if $\mathcal{W}$ is a
minimal cut-set for $\Lambda^{*}$, then by the self-similarity relation
(\ref{eq:SS relation})
\[
\mu=\sum_{w\in\mathcal{W}}p_{w}\cdot\varphi_{w}\mu\:.
\]

For $1\le i\le\ell$ set
\[
g_{i}:=(\log r_{i}^{-1},U_{i})\in\mathbb{R}\times O(d),
\]
where throughout the paper the base of the $\log$ function is always
$2$. Let $G$ be the smallest closed subgroup of $\mathbb{R}\times O(d)$
containing the elements $\{g_{i}\}_{i=1}^{\ell}$. We always equip
$G$ with the subspace topology inherited from $\mathbb{R}\times O(d)$.
Since $G$ is a closed subgroup of the Lie group $\mathbb{R}\times O(d)$,
it is itself a Lie group. We denote by $G_{0}$ the connected component
of $G$ containing the identity element. We write,
\[
x.(t,U):=2^{-t}U^{-1}x\text{ for }(t,U)\in G\text{ and }x\in\mathbb{R}^{d},
\]
which defined a right action of $G$ on $\mathbb{R}^{d}$.

Let $\psi:G\rightarrow\mathbb{R}$ be the projection onto the first
coordinate, that is $\psi(t,U)=t$ for $(t,U)\in G$, and write $N$
for the kernel of $\psi$. Since the homomorphism $\psi$ is continuous
and proper, it is also a closed map. In particular $\psi(G)$ is a
closed subgroup of $\mathbb{R}$. Given $T\in\mathbb{R}$ we write
$G_{\le T}$ for the set $\psi^{-1}(-\infty,T]$, and use the notations
$G_{<T}$, $G_{\ge T}$ and $G_{>T}$ in a similar way.

Let $\mathbf{m}_{\mathbb{R}}$ be the Lebesgue measure of $\mathbb{R}$.
Let $\mathbf{m}_{\psi(G)}$ be the Haar measure of $\psi(G)$, normalized
so that $\mathbf{m}_{\psi(G)}=\mathbf{m}_{\mathbb{R}}$ if $\psi(G)=\mathbb{R}$
and $\mathbf{m}_{\psi(G)}\{t\}=\beta$ for all $t\in\psi(G)$ if $\psi(G)=\beta\mathbb{Z}$
with $\beta>0$. We show in Corollary \ref{cor:G unimodular and m_G=00003D}
below that $G$ is unimodular. It is easy to see that if $\mathbf{m}$
is a Haar measure for $G$ then $\psi\mathbf{m}$ is a Haar measure
for $\psi(G)$. We denote by $\mathbf{m}_{G}$ the Haar measure of
$G$, normalized so that $\psi\mathbf{m}_{G}=\mathbf{m}_{\psi(G)}$.
Our choice of normalization for $\mathbf{m}_{G}$ can be explained
by the version of the renewal theorem for $G$ stated below (see Section
\ref{subsec:A-version-of ren thm G}).

\subsection{\label{subsec:Basic-properties-of G}Basic properties of $G$}
\begin{lem}
\label{lem:proper homo from R}There exists a continuous and proper
homomorphism $\gamma:\psi(G)\rightarrow G$ such that $\psi\circ\gamma=Id$.
If $\psi(G)=\mathbb{R}$, then $\gamma$ is smooth. If $\psi(G)=\beta\mathbb{Z}$
for $\beta>0$, then for any $g\in G$ with $\psi(g)=\beta$ it is
possible to define $\gamma$ so that $\gamma(\beta)=g$.
\end{lem}

\begin{proof}
If $\psi(G)=\beta\mathbb{Z}$ for $\beta>0$ then the lemma is trivial.
Given $g\in G$ with $\psi(g)=\beta$, we simply set $\gamma(\beta n)=g^{n}$
for $n\in\mathbb{Z}$. Clearly $\gamma$ satisfies the required properties.

Suppose next that $\psi(G)=\mathbb{R}$. Let $\mathfrak{g}$ and $\mathfrak{o}(d)$
be the Lie algebras of $G$ and $O(d)$ respectively. We may identify
$\mathfrak{g}$ as a Lie subalgebra of $\mathbb{R}\times\mathfrak{o}(d)$.
Since $\psi(G)=\mathbb{R}$, there exists $X\in\mathfrak{g}$ so that
its projection onto the first coordinate of $\mathbb{R}\times\mathfrak{o}(d)$
is equal to $1$. Set $\gamma(t)=\exp(tX)$ for every $t\in\mathbb{R}$,
where $\exp:\mathfrak{g}\rightarrow G$ is the exponential map of
$G$. It is easy to check that $\gamma$ satisfies the required properties.
\end{proof}
We consider the homomorphism $\gamma$ from the previous lemma as
fixed. In later sections we shall often write $\gamma_{t}$ in place
of $\gamma(t)$. Recall that we write $N$ for the kernel of $\psi$,
and let $H$ denote the image of $\gamma$. Since $\gamma$ is continuous
and proper, $H$ is a closed subgroup of $G$. Since $N\triangleleft G$,
the subgroup $H$ acts on $N$ by conjugation. For $h\in H$ and $n\in N$
we write $n^{h}$ in place of $hnh^{-1}$. Let $N\rtimes H$ be the
semidirect product of $N$ by $H$. That is, $N\rtimes H$ is the
group whose underlying set is $N\times H$ with the following group
operation,
\[
(n_{1},h_{1})\cdot(n_{2},h_{2})=(n_{1}n_{2}^{h_{1}},h_{1}h_{2})\text{ for }(n_{1},h_{1}),(n_{2},h_{2})\in N\rtimes H\:.
\]
We equip $N$ and $H$ with the subspace topologies inherited from
$G$, and $N\rtimes H$ with the product topology. It is easy to verify
that this makes $N\rtimes H$ into a locally compact group. Let $F:N\rtimes H\rightarrow G$
be with $F(n,h)=nh$ for $(n,h)\in N\rtimes H$.
\begin{lem}
\label{lem:G splits and F iso}$G$ is a split extension of $N$ by
$H$, that is $HN=G$ and $H\cap N=\{1_{G}\}$. Consequently, the
map $F$ is an isomorphism of topological groups.
\end{lem}

\begin{proof}
For $g\in G$ we have $\psi(\gamma(\psi g)^{-1}g)=0$. Hence,
\[
g=\gamma(\psi g)\cdot\gamma(\psi g)^{-1}g\in HN,
\]
which shows that $HN=G$. Next let $g\in H\cap N$, then $\psi g=0$
and there exists $t\in\mathbb{R}$ with $\gamma(t)=g$. Thus,
\[
t=\psi(\gamma(t))=\psi g=0,
\]
and so $1_{G}=\gamma(0)=\gamma(t)=g$, which shows that $H\cap N=\{1_{G}\}$.

It is easy to verify that $F$ is a homomorphism. From $HN=G$ and
$H\cap N=\{1_{G}\}$ it follows that $F$ is a group isomorphism.
It is obvious that $F$ is continuous. It is also easy to see that
$F$ is a proper map, and so it is a closed map. This shows that $F$
is an isomorphism of topological groups, and completes the proof of
the lemma.
\end{proof}
Since $N$ is a closed subgroup of $\{0\}\times O(d)$ it is compact.
Let $\mathbf{m}_{N}$ be the Haar measure of $N$, normalized so that
$\mathbf{m}_{N}(N)=1$. By Lemma \ref{lem:proper homo from R} the
map $\gamma:\psi(G)\rightarrow H$ is an isomorphism of topological
groups. Write $\mathbf{m}_{H}$ for $\gamma\mathbf{m}_{\psi(G)}$,
so that $\mathbf{m}_{H}$ is a Haar measure for $H$.
\begin{lem}
\label{lem:haar for semi prod}$\mathbf{m}_{N}\times\mathbf{m}_{H}$
is a left and right Haar measure for $N\rtimes H$.
\end{lem}

\begin{proof}
Since $N$ is compact and $H$ is abelian, $\mathbf{m}_{N}$ and $\mathbf{m}_{H}$
are both left and right Haar measures. Let $(n_{0},h_{0})\in N\rtimes H$
and $f\in C_{c}(N\rtimes H)$ be given. Since $N$ is compact, the
automorphism $n\rightarrow n^{h_{0}}$ preserves $\mathbf{m}_{N}$
(see e.g. \cite[Section 1.1]{Wa}). Thus,
\begin{eqnarray*}
\int f((n_{0},h_{0})\cdot(n,h))\:d\mathbf{m}_{N}\times\mathbf{m}_{H}(n,h) & = & \int\int f(n_{0}n^{h_{0}},h_{0}h)\:d\mathbf{m}_{N}(n)\:d\mathbf{m}_{H}(h)\\
 & = & \int\int f(n_{0}n,h_{0}h)\:d\mathbf{m}_{N}(n)\:d\mathbf{m}_{H}(h)\\
 & = & \int f(n,h)\:d\mathbf{m}_{N}\times\mathbf{m}_{H}(n,h),
\end{eqnarray*}
which shows that $\mathbf{m}_{N}\times\mathbf{m}_{H}$ is a left Haar
measure for $N\rtimes H$. The proof that it is also a right Haar
measure is even simpler, and is therefore omitted.
\end{proof}
\begin{cor}
\label{cor:G unimodular and m_G=00003D}$G$ is unimodular, and $\mathbf{m}_{G}=F(\mathbf{m}_{N}\times\mathbf{m}_{H})$.
\end{cor}

\begin{proof}
From Lemmata \ref{lem:G splits and F iso} and \ref{lem:haar for semi prod}
it follows that $G$ is unimodular and that $F(\mathbf{m}_{N}\times\mathbf{m}_{H})$
is a Haar measure for $G$. For every $(n,h)\in N\rtimes H$ we have
$\psi F(n,h)=\psi h$. Hence,
\[
\psi F(\mathbf{m}_{N}\times\mathbf{m}_{H})=\psi\mathbf{m}_{H}=\psi\gamma\mathbf{m}_{\psi(G)}=\mathbf{m}_{\psi(G)}\:.
\]
Since $\mathbf{m}_{G}$ is the unique Haar measure for $G$ whose
image under $\psi$ is equal to $\mathbf{m}_{\psi(G)}$ the corollary
follows.
\end{proof}

\subsection{Self-similar measures vanish on proper affine subspaces}

The following lemma is a consequence of the affine irreducibility
of $\Phi$. It is well known, but since we could not find a proof
in the existing literature we provide one for completeness. The lemma
will be used in Section \ref{sec:The-nondiscrete-case} when we consider
the case in which the group $G$ is nondiscrete.
\begin{lem}
\label{lem:0 mass aff sub}Let $p=(p_{i})_{i=1}^{\ell}>0$ be a probability
vector, and let $\mu$ be the self-similar measure corresponding to
$\Phi$ and $p$. Then $\mu(\mathbb{V})=0$ for every proper affine
subspace $\mathbb{V}$ of $\mathbb{R}^{d}$.
\end{lem}

\begin{proof}
For $0\le k\le d$ denote by $\mathbf{A}_{k}$ the collection of all
$k$-dimensional affine subspaces of $\mathbb{R}^{d}$. Let $m$ be
the smallest nonnegative integer for which there exists $\mathbb{V}\in\mathbf{A}_{m}$
with $\mu(\mathbb{V})>0$. Assume by contradiction that $m<d$.

Set
\[
\kappa:=\sup\{\mu(\mathbb{V})\::\:\mathbb{V}\in\mathbf{A}_{m}\},
\]
and
\[
\mathbf{M}=\{\mathbb{V}\in\mathbf{A}_{m}\::\:\mu(\mathbb{V})=\kappa\}\:.
\]
By the definition of $m$ we have $\kappa>0$. From this, since $\mu$
is a finite measure, and since $\mu(\mathbb{V})=0$ for all $\mathbb{V}\in\cup_{k=0}^{m-1}\mathbf{A}_{k}$,
it follows that $\mathbf{M}$ is nonempty and finite. Let $\{\mathbb{V}_{j}\}_{j=1}^{s}$
be an enumeration of the elements in $\mathbf{M}$.

Set $\mathbb{Y}:=\cup_{j=1}^{s}\mathbb{V}_{j}$. For $1\le j\le s$,
\[
\kappa=\mu(\mathbb{V}_{j})=\sum_{i=1}^{\ell}p_{i}\cdot\mu(\varphi_{i}^{-1}(\mathbb{V}_{j}))\:.
\]
By the definition of $\kappa$ we have $\mu(\varphi_{i}^{-1}(\mathbb{V}_{j}))\le\kappa$
for $1\le i\le\ell$, and so $\varphi_{i}^{-1}(\mathbb{V}_{j})\in\mathbf{M}$
for $1\le i\le\ell$. This implies that $\mathbb{Y}$ is invariant
with respect to the maps in $\Phi$. From this and since $\mathbb{Y}$
is closed, it follows that the attractor $K$ is contained in $\mathbb{Y}$.

For every $1\le j_{1}<j_{2}\le s$ it holds that $\mathbb{V}_{j_{1}}\cap\mathbb{V}_{j_{2}}$
is either empty, or it is an affine subspace of dimension strictly
less than $m$. Thus, by the definition of $m$ we have $\mu(\mathbb{V}_{j_{1}}\cap\mathbb{V}_{j_{2}})=0$
for $1\le j_{1}<j_{2}\le s$. Since $\mu$ is supported on $K$, this
implies that $K$ is not contained in 
\[
\cup_{1\le j_{1}<j_{2}\le s}\mathbb{V}_{j_{1}}\cap\mathbb{V}_{j_{2}}\:.
\]
Hence, by reordering $\{\mathbb{V}_{j}\}_{j=1}^{s}$ if necessary,
we may assume that there exists $x\in K\cap\mathbb{V}_{1}$ with $x\notin\mathbb{V}_{j}$
for $2\le j\le s$. From
\[
\inf\{\mathrm{dist}(x,\mathbb{V}_{j})\::\:2\le j\le s\}>0
\]
and 
\[
\inf\{\mathrm{diam}(\varphi_{w}(K))\::\:w\in\Lambda^{*}\text{ and }x\in\varphi_{w}(K)\}=0,
\]
it follows that there exists $w\in\Lambda^{*}$ such that $\varphi_{w}(K)\cap\mathbb{V}_{j}=\emptyset$
for $2\le j\le s$. From this and $\varphi_{w}(K)\subset K\subset\mathbb{Y}$,
we obtain that $K\subset\varphi_{w}^{-1}(\mathbb{V}_{1})$. Since
$\dim(\varphi_{w}^{-1}(\mathbb{V}_{1}))=m<d$, this contradicts the
affine irreducibility of $\Phi$ and completes the proof of the lemma.
\end{proof}

\section{\label{sec:Renewal-theory-and first hit}Renewal theory and first
hitting distribution}
\begin{defn}
Given $q\in\mathcal{M}(G)$ we say that $q$ is adapted if the subgroup
generated by the support of $q$ is dense in $G$.
\end{defn}

Throughout this section we fix an adapted finitely supported probability
measure $q$ on $G$ with $q(G_{>0})=1$, where recall that $G_{>0}:=\psi^{-1}(0,\infty)$.
We write $\lambda$ in place of $\int\psi\:dq$, so that $\lambda>0$.

\subsection{\label{subsec:A-version-of ren thm G}A version of the renewal theorem
for $G$}

Set $Q:=\sum_{n\ge0}q^{n}$, where $q^{n}$ is the $n$-fold convolution
of $q$ with itself for $n\ge1$ and $q^{0}$ is the Dirac mass as
$1_{G}$. Since $q(G_{>0})=1$, it is obvious that $Q$ is a Radon
measure on $G$. For $g\in G$ let $L_{g}:G\rightarrow G$ be with
$L_{g}g'=gg'$ for $g'\in G$, and note that $L_{g}Q:=Q\circ L_{g}^{-1}$
is also a Radon measure on $G$.

The following theorem follows directly from \cite[Theorem A.1]{BDGHU}.
It extends the classical renewal theorem for closed subgroups of $\mathbb{R}$
(see e.g. \cite[Chapter 5]{Re}) to the group $G$.
\begin{thm}
\label{thm:renewal for G}Let $h_{1},h_{2},...\in G$ be with $\psi h_{n}\overset{n}{\rightarrow}-\infty$.
Then for every $f:G\rightarrow\mathbb{R}$ which is Borel measurable,
bounded, compactly supported and satisfies
\[
\mathbf{m}_{G}\{g\in G\::\:f\text{ is not continuous at }g\}=0,
\]
we have
\[
\underset{n\rightarrow\infty}{\lim}\:\int f\:dL_{h_{n}}Q=\lambda^{-1}\int f\:d\mathbf{m}_{G}\:.
\]
\end{thm}

\subsection{\label{subsec:Limit-distribution-of}First hitting distribution}

Let $X_{1},X_{2},...$ be i.i.d. $G$-valued random elements with
distribution $q$. Set $Y_{0}:=1_{G}$, and for $n\ge1$ let $Y_{n}:=X_{1}\cdot...\cdot X_{n}$.
For $t>0$ write,
\[
\tau_{t}:=\inf\{n\ge1\::\:\psi Y_{n}\ge t\}\:.
\]
For $g\in G$ set,
\[
\rho(g):=\lambda^{-1}\mathbb{P}\{\psi X_{1}>\psi g\ge0\}\:.
\]
Note that since $\lambda=\mathbb{E}[\psi X_{1}]$, and by our choice
of $\mathbf{m}_{G}$, it follows that $\int\rho\:d\mathbf{m}_{G}=1$.
Write $\nu$ in place of $\rho\:d\mathbf{m}_{G}$, so that $\nu\in\mathcal{M}(G)$. 

Recall the homomorphism $\gamma:\psi(G)\rightarrow G$ from Lemma
\ref{lem:proper homo from R}, and that for $t\in\psi(G)$ we often
write $\gamma_{t}$ in place of $\gamma(t)$. The idea of the proof
of the following proposition is based on \cite[Section 4, Proof of Theorem 3]{St}.
\begin{prop}
\label{prop:conv in dist}The random elements $\{\gamma_{-t}Y_{\tau_{t}}\}_{t\in\psi(G_{>0})}$
converge in distribution to $\nu$ as $t\rightarrow\infty$. That
is, for every continuous and bounded $f:G\rightarrow\mathbb{C}$,
\[
\underset{t\rightarrow\infty}{\lim}\:\mathbb{E}\left[f(\gamma_{-t}Y_{\tau_{t}})\right]=\int f\:d\nu\:.
\]
\end{prop}

For the proof we need the following lemma.
\begin{lem}
\label{lem:conv in dist first step}For $f\in C_{c}(G)$ we have,
\[
\underset{t\rightarrow\infty}{\lim}\:\int1_{G_{<0}}(g)\int f(gh)\:dq(h)\:dL_{\gamma_{-t}}Q(g)=\frac{1}{\lambda}\int1_{G_{<0}}(g)\int f(gh)\:dq(h)\:d\mathbf{m}_{G}(g)\:.
\]
\end{lem}

\begin{proof}
Let $f\in C_{c}(G)$, and for $g\in G$ set
\[
\tilde{f}(g)=1_{G_{<0}}(g)\int f(gh)\:dq(h)\:.
\]
It is clear that $\tilde{f}$ is Borel measurable and bounded. Since
$q$ and $f$ are compactly supported so does $\tilde{f}$. The set
of points at which $\tilde{f}$ is discontinuous is contained in the
boundary of $G_{<0}$, which we denoted by $\partial G_{<0}$. If
$\psi(G)=\mathbb{R}$ then,
\[
\mathbf{m}_{G}(\partial G_{<0})=\mathbf{m}_{G}(\psi^{-1}\{0\})=\mathbf{m}_{\mathbb{R}}\{0\}=0\:.
\]
If $\psi(G)\ne\mathbb{R}$ then $\partial G_{<0}=\emptyset$. From
Theorem \ref{thm:renewal for G} and since $\psi\circ\gamma=Id$ we
now get,
\[
\underset{t\rightarrow\infty}{\lim}\:\int\tilde{f}\:dL_{\gamma_{-t}}Q=\lambda^{-1}\int\tilde{f}\:d\mathbf{m}_{G},
\]
which completes the proof of the lemma.
\end{proof}
\begin{proof}[Proof of Proposition \ref{prop:conv in dist}]
Let $f\in C_{c}(G)$ be nonnegative and with $f(g)=0$ for $g\in G_{<0}$.
Since $\nu(G_{<0})=0$, it suffices to show
\begin{equation}
\underset{t\rightarrow\infty}{\lim}\:\mathbb{E}\left[f(\gamma_{-t}Y_{\tau_{t}})\right]=\int f\:d\nu\:.\label{eq:suf to show conv in dist}
\end{equation}
Note that for $n\ge1$ the distribution of $Y_{n}$ is equal to $q^{n}$.
Hence, for $t\in\psi(G_{>0})$
\begin{eqnarray*}
\int1_{G_{<0}}(g)\int f(gh)\:dq(h)\:dL_{\gamma_{-t}}Q(g) & = & \sum_{n\ge0}\int1_{G_{<0}}(g)\int1_{G_{\ge0}}(gh)f(gh)\:dq(h)\:dL_{\gamma_{-t}}q^{n}(g)\\
 & = & \sum_{n\ge0}\mathbb{E}\left[1_{G_{<0}}(\gamma_{-t}Y_{n})1_{G_{\ge0}}(\gamma_{-t}Y_{n+1})f(\gamma_{-t}Y_{n+1})\right]\\
 & = & \sum_{n\ge0}\mathbb{E}\left[1_{\{\psi Y_{n}<t\}}1_{\{\psi Y_{n+1}\ge t\}}f(\gamma_{-t}Y_{n+1})\right]\\
 & = & \sum_{n\ge0}\mathbb{E}\left[1_{\{\tau_{t}=n+1\}}f(\gamma_{-t}Y_{n+1})\right]=\mathbb{E}\left[f(\gamma_{-t}Y_{\tau_{t}})\right]\:.
\end{eqnarray*}
Moreover, by the right-invariance of $\mathbf{m}_{G}$
\begin{eqnarray*}
\lambda^{-1}\int1_{G_{<0}}(g)\int f(gh)\:dq(h)\:d\mathbf{m}_{G}(g) & = & \lambda^{-1}\int f(g)\int1_{G_{\ge0}}(g)1_{G_{<0}}(gh^{-1})\:dq(h)\:d\mathbf{m}_{G}(g)\\
 & = & \lambda^{-1}\int f(g)\mathbb{P}\{\psi X_{1}>\psi g\ge0\}\:d\mathbf{m}_{G}(g)\\
 & = & \int f(g)\rho(g)\:d\mathbf{m}_{G}(g)=\int f\:d\nu\:.
\end{eqnarray*}
Thus, the equality (\ref{eq:suf to show conv in dist}) follows from
Lemma \ref{lem:conv in dist first step}, which completes the proof
of the proposition.
\end{proof}
We shall need a uniform version of Proposition \ref{prop:conv in dist}.
Define a metric $d_{op}$ on $G$ by setting
\begin{equation}
d_{op}((r,U),(s,V))=\Vert2^{-r}U^{-1}-2^{-s}V^{-1}\Vert\text{ for }(r,U),(s,V)\in G,\label{eq:def of d_op}
\end{equation}
where $\Vert\cdot\Vert$ is the operator norm. It is clear that the
topology induced by $d_{op}$ is equal to the subspace topology inherited
from $\mathbb{R}\times O(d)$. Given $C>0$, we say that $f:G\rightarrow\mathbb{C}$
is $C$-Lipschitz with respect to $d_{op}$ if
\[
|f(g)-f(g')|\le Cd_{op}(g,g')\text{ for }g,g'\in G\:.
\]

Observe that there exists a compact subset $B$ of $G$ so that $\nu(B)=1$
and $\mathbb{P}\{\gamma_{-t}Y_{\tau_{t}}\in B\}=1$ for all $t\in\psi(G_{>0})$.
The following corollary follows directly from this, from Proposition
\ref{prop:conv in dist}, and from \cite[Lemma A.3.3]{BP}.
\begin{cor}
\label{cor:conv in dist lip}For every $\epsilon>0$ there exists
$T=T(q,\epsilon)>1$ so that the following holds. Let $f:G\rightarrow\mathbb{C}$
be $\epsilon^{-1}$-Lipschitz with respect to $d_{op}$, then
\[
\left|\mathbb{E}\left[f(\gamma_{-t}Y_{\tau_{t}})\right]-\int f\:d\nu\right|<\epsilon\text{ for all }t\in\psi(G_{\ge T})\:.
\]
\end{cor}

\section{\label{sec:The-nondiscrete-case}The nondiscrete case}

In this section we consider the situation in which the group $G$
is nondiscrete. In Section \ref{subsec:The-case psi(G)=00003DR} we
show that self-similar measures corresponding to positive probability
vectors are always Rajchman whenever $\psi(G)=\mathbb{R}$. In Section
\ref{subsec:Reduction-to-the disc case} we assume $\psi(G)\ne\mathbb{R}$,
and prove a result which will enable us to make a reduction to the
case in which $G$ is discrete

Recall that $\Phi=\{\varphi_{i}(x)=r_{i}U_{i}x+a_{i}\}_{i=1}^{\ell}$
is an affinely irreducible self-similar IFS on $\mathbb{R}^{d}$,
and that $g_{i}=(\log r_{i}^{-1},U_{i})$ for $1\le i\le\ell$. Throughout
this section let $p=(p_{i})_{i=1}^{\ell}$ be a fixed positive probability
vector. Since we consider $p$ as fixed for this section, usually
the dependence of various parameters on $p$ will not be indicated.
Let $\mu$ be the self-similar measure corresponding to $\Phi$ and
$p$. Let $\sigma\in\mathcal{M}(\mathbb{R}^{d})$ be defined by,
\[
\sigma(f)=\int\int f(x-y)\:d\mu(x)\:d\mu(y)\text{ for }f\in C_{c}(\mathbb{R}^{d})\:.
\]

Set
\[
q:=\sum_{i=1}^{\ell}p_{i}\delta_{g_{i}},
\]
where $\delta_{g_{i}}$ is the Dirac mass at $g_{i}$. By definition
$G$ is the closed subgroup generated be the support of $q$, and
so $q$ is adapted. As before, we write $\lambda$ in place of $\int\psi\:dq$.
Let $\{X_{n}\}_{n\ge1}$, $\{Y_{n}\}_{n\ge0}$, $\{\tau_{t}\}_{t>0}$,
$\rho:G\rightarrow[0,\infty)$ and $\nu\in\mathcal{M}(G)$ be as defined
in Section \ref{subsec:Limit-distribution-of}, where we assume that
these objects are defined with respect to the present choice of $q$.

\subsection{An initial upper bound}

The purpose of this subsection is to prove Lemma \ref{lem:initial upper bd}.
Recall that for $(t,U)=g\in G$ and $x\in\mathbb{R}^{d}$ we write
$x.g:=2^{-t}U^{-1}x$.
\begin{lem}
\label{lem:initial upper bd}For every $\epsilon>0$ there exists
$T>1$ such that the following holds. Let $t\in\psi(G_{\ge T})$ and
let $\xi\in\mathbb{R}^{d}$ be with $|\xi|\le\epsilon^{-1}$, then
\[
|\widehat{\mu}(\xi.\gamma_{-t})|^{2}\le\epsilon+\int\left|\int e^{i\left\langle \xi.g,x\right\rangle }d\nu(g)\right|\:d\sigma(x)\:.
\]
\end{lem}

We first need the following lemma, whose proof is similar to the proof
of \cite[Lemma 3.1]{LS}.
\begin{lem}
\label{lem:ub on fur by double integral}For every $\xi\in\mathbb{R}^{d}$
and $t\in\psi(G_{>0})$,
\[
|\widehat{\mu}(\xi)|^{2}\le\int\mathbb{E}\left[e^{i\left\langle \xi.Y_{\tau_{t}},x\right\rangle }\right]\:d\sigma(x)\:.
\]
\end{lem}

\begin{proof}
Recall that $\Lambda:=\{1,...,\ell\}$, and that for a group $Z$,
elements $\{z_{i}\}_{i=1}^{\ell}\subset Z$ and a word $i_{1}...i_{n}=w\in\Lambda^{*}$,
we write $z_{w}$ in place of $z_{i_{1}}\cdot...\cdot z_{i_{n}}$.
Let,
\[
\mathcal{W}=\{i_{1}...i_{n}\in\Lambda^{*}\::\:\psi(g_{i_{1}...i_{n}})\ge t>\psi(g_{i_{1}...i_{n-1}})\}\:.
\]
Since $\mathcal{W}$ is a minimal cut-set (see Section \ref{subsec:General-notations}),
\[
\mu=\sum_{w\in\mathcal{W}}p_{w}\cdot\varphi_{w}\mu\:.
\]
This implies,
\[
\widehat{\mu}(\xi)=\sum_{w\in\mathcal{W}}p_{w}\int e^{i\left\langle \xi,\varphi_{w}(x)\right\rangle }\:d\mu(x).
\]
Thus be Jensen's inequality,
\begin{eqnarray*}
|\widehat{\mu}(\xi)|^{2} & \le & \sum_{w\in\mathcal{W}}p_{w}\left|\int e^{i\left\langle \xi,\varphi_{w}(x)\right\rangle }\:d\mu(x)\right|^{2}\\
 & = & \sum_{w\in\mathcal{W}}p_{w}\int e^{i\left\langle \xi,\varphi_{w}(x)\right\rangle }\:d\mu(x)\cdot\int e^{-i\left\langle \xi,\varphi_{w}(y)\right\rangle }\:d\mu(y)\\
 & = & \int\int\sum_{w\in\mathcal{W}}p_{w}e^{i\left\langle \xi,r_{w}U_{w}(x-y)\right\rangle }\:d\mu(x)\:d\mu(y)\\
 & = & \int\sum_{w\in\mathcal{W}}p_{w}e^{i\left\langle \xi.g_{w},x\right\rangle }\:d\sigma(x)\:.
\end{eqnarray*}
The lemma now follows since the distribution of $Y_{\tau_{t}}$ is
equal to $\sum_{w\in\mathcal{W}}p_{w}\delta_{g_{w}}$.
\end{proof}
\begin{proof}[Proof of Lemma \ref{lem:initial upper bd}]
Let $\epsilon>0$ and let $t\in\psi(G_{>0})$ be large with respect
to $\epsilon$, $p$ and $\Phi$. Fix $\xi\in\mathbb{R}^{d}$ with
$|\xi|\le\epsilon^{-1}$. By Lemma \ref{lem:ub on fur by double integral},
\[
|\widehat{\mu}(\xi.\gamma_{-t})|^{2}\le\int\mathbb{E}\left[e^{i\left\langle \xi.(\gamma_{-t}Y_{\tau_{t}}),x\right\rangle }\right]\:d\sigma(x)\:.
\]

Recall the metic $d_{op}$ on $G$ defined in (\ref{eq:def of d_op}).
Observe that $\sigma$ is supported on the compact set $K-K$, where
$K$ is the attractor of $\Phi$. From this and since $|\xi|\le\epsilon^{-1}$,
there exists a constant $C>1$, which depends only on $\epsilon$
and $\Phi$, so that for every $x\in\mathrm{supp}(\sigma)$ the map
which takes $g\in G$ to $e^{i\left\langle \xi.g,x\right\rangle }$
is $C$-Lipschitz with respect to $d_{op}$. Thus, by assuming that
$t$ is large enough and by Corollary \ref{cor:conv in dist lip},
\[
|\widehat{\mu}(\xi.\gamma_{-t})|^{2}\le\epsilon+\int\left|\int e^{i\left\langle \xi.g,x\right\rangle }\:d\nu(g)\right|\:d\sigma(x),
\]
which completes the proof of the lemma.
\end{proof}

\subsection{Average Fourier decay of measures on curves}

The purpose of this subsection is to prove Lemma \ref{lem:ub int sig of fur of curve mass}.
It will enable us to make use of the upper bound obtained in the previous
section.

For $y\in\mathbb{R}^{d}$ and $\delta>0$ we write $B(y,\delta)$
for the closed ball in $\mathbb{R}^{d}$ with centre $y$ and radius
$\delta$. Let $\mathbb{RP}^{d-1}$ be the projective space of $\mathbb{R}^{d}$,
and for $0\ne x\in\mathbb{R}^{d}$ write $\overline{x}\in\mathbb{RP}^{d-1}$
for the line spanned by $x$. Recall that for a linear subspace $\mathbb{V}\subset\mathbb{R}^{d}$
we denote its orthogonal projection by $\pi_{\mathbb{V}}$.
\begin{lem}
\label{lem:sig mass of strips}For every $\epsilon>0$ there exists
$\delta>0$ such that,
\[
\pi_{\overline{x}}\sigma(B(y,\delta))<\epsilon\text{ for all }\overline{x}\in\mathbb{RP}^{d-1}\text{ and }y\in\overline{x}\:.
\]
\end{lem}

\begin{proof}
Assume by contradiction that the lemma fails for some $\epsilon>0$.
Then for every $n\ge1$ there exist $\overline{x_{n}}\in\mathbb{RP}^{d-1}$
and $y_{n}\in\overline{x_{n}}$ so that $\pi_{\overline{x_{n}}}\sigma(B(y_{n},\frac{1}{n}))\ge\epsilon$.
For every $n\ge1$ we have
\[
B(y_{n},\frac{1}{n})\cap\pi_{\overline{x_{n}}}(\mathrm{supp}(\sigma))\ne\emptyset\:.
\]
From this, and since $\sigma$ is compactly supported, it follows
that there exists $M>1$ so that $y_{1},y_{2},...\in B(0,M)$. Thus,
there exist $\overline{x}\in\mathbb{RP}^{d-1}$, $y\in\overline{x}$
and an increasing sequence $\{n_{k}\}_{k\ge1}\subset\mathbb{Z}_{\ge1}$,
so that $\overline{x_{n_{k}}}\overset{k}{\rightarrow}\overline{x}$
and $y_{n_{k}}\overset{k}{\rightarrow}y$.

For $\eta>0$ and sufficiently large $k\ge1$,
\[
\pi_{\overline{x_{n_{k}}}}^{-1}(B(y_{n_{k}},\frac{1}{n_{k}}))\cap\mathrm{supp}(\sigma)\subset\pi_{\overline{x}}^{-1}(B(y,\eta))\:.
\]
Since $\pi_{\overline{x_{n_{k}}}}\sigma(B(y_{n_{k}},\frac{1}{n_{k}}))\ge\epsilon$,
this implies $\pi_{\overline{x}}\sigma(B(y,\eta))\ge\epsilon$. Since
this holds for all $\eta>0$ we have $\pi_{\overline{x}}\sigma\{y\}\ge\epsilon$.
Hence by the definition of $\sigma$,
\[
\epsilon\le\int\int1_{\pi_{\overline{x}}^{-1}\{y\}}(z-\xi)\:\mu(z)\:d\mu(\xi)=\int\mu(\pi_{\overline{x}}^{-1}\{y+\pi_{\overline{x}}\xi\})\:d\mu(\xi)\:.
\]
Thus, there exists a proper affine subspace $\mathbb{V}$ of $\mathbb{R}^{d}$
so that $\mu(\mathbb{V})>0$. This contradicts Lemma \ref{lem:0 mass aff sub},
which completes the proof.
\end{proof}
\begin{lem}
\label{lem:fur of line curves}For every $\epsilon>0$ there exists
$S>1$ so that the following holds. Let $s\ge S$ and let $u\in\mathbb{R}^{d}$
be with $|u|\ge1$, then
\[
\int\left|\int_{0}^{1}e^{i\left\langle tu,sx\right\rangle }\:dt\right|\:d\sigma(x)<\epsilon\:.
\]
\end{lem}

\begin{proof}
Let $\epsilon>0$, let $\delta>0$ be small with respect to $\epsilon$,
and let $s>4/(\delta\epsilon)$. Fix $u\in\mathbb{R}^{d}$ with $|u|\ge1$.
By Lemma \ref{lem:sig mass of strips} we may assume that $\pi_{\overline{u}}\sigma(B(0,\delta))<\epsilon/2$.
For every $x\in\mathbb{R}^{d}$ with $|\pi_{\overline{u}}x|\ge\delta$,
\[
\left|\int_{0}^{1}e^{i\left\langle tu,sx\right\rangle }\:dt\right|=\left|\frac{1}{s\left\langle u,x\right\rangle }(e^{is\left\langle u,x\right\rangle }-1)\right|\le2s^{-1}\delta^{-1}<\epsilon/2\:.
\]
Hence,
\[
\int\left|\int_{0}^{1}e^{i\left\langle tu,sx\right\rangle }\:dt\right|\:d\sigma(x)<\int1_{\{|\pi_{\overline{u}}x|\ge\delta\}}\left|\int_{0}^{1}e^{i\left\langle tu,sx\right\rangle }\:dt\right|\:d\sigma(x)+\frac{\epsilon}{2}<\epsilon,
\]
which completes the proof of the lemma.
\end{proof}
\begin{lem}
\label{lem:ub int sig of fur of curve mass}For every $\epsilon>0$
there exists $S>1$ so that the following holds. Let $s\ge S$ and
let $c:[0,1]\rightarrow\mathbb{R}^{d}$ be a smooth curve with $|c'(t)|\ge\epsilon$
and $|c''(t)|\le\epsilon^{-1}$ for all $0\le t\le1$, then
\[
\int\left|\int_{0}^{1}e^{i\left\langle c(t),sx\right\rangle }\:dt\right|\:d\sigma(x)<\epsilon\:.
\]
\end{lem}

\begin{proof}
Let $\epsilon>0$ and let $s>1$ be large with respect to $\epsilon$
and $\mathrm{supp}(\sigma)$. Fix a smooth curve $c:[0,1]\rightarrow\mathbb{R}^{d}$
with $|c'(t)|\ge\epsilon$ and $|c''(t)|\le\epsilon^{-1}$ for all
$0\le t\le1$. Let $n\ge1$ be such that $(n-1)^{3/2}<s\le n^{3/2}$.
By assuming that $s$ is large enough,
\begin{equation}
\frac{s}{n}>\frac{1}{2}\frac{s}{n-1}>\frac{1}{2}(n-1)^{1/2}>\frac{1}{4}n^{1/2}\ge\frac{1}{4}s^{1/3}\:.\label{eq:s/n >}
\end{equation}

Firstly, we have
\begin{equation}
\int\left|\int_{0}^{1}e^{i\left\langle c(t),sx\right\rangle }\:dt\right|\:d\sigma(x)\le\sum_{k=0}^{n-1}\int\left|\int_{0}^{1/n}e^{i\left\langle c(t+\frac{k}{n}),sx\right\rangle }\:dt\right|\:d\sigma(x)\:.\label{eq:<=00003D sum int}
\end{equation}
For $0\le k<n$ set $v_{k}=c(\frac{k}{n})$ and $u_{k}=c'(\frac{k}{n})$,
and note that $|u_{k}|\ge\epsilon$. By Taylor's theorem it follows
that for $0\le t\le1/n$,
\[
|c(t+\frac{k}{n})-v_{k}-tu_{k}|\le d\frac{\Vert c''\Vert_{\infty}}{2}n^{-2}\le\frac{d}{2\epsilon n^{2}}\:.
\]
Hence for $x\in\mathrm{supp}(\sigma)$,
\begin{eqnarray*}
\left|e^{i\left\langle c(t+\frac{k}{n}),sx\right\rangle }-e^{i\left\langle v_{k}+tu_{k},sx\right\rangle }\right| & \le & \left|\left\langle c(t+\frac{k}{n})-v_{k}-tu_{k},sx\right\rangle \right|\\
 & \le & n^{3/2}|x|\cdot|c(t+\frac{k}{n})-v_{k}-tu_{k}|\\
 & \le & \frac{d|x|}{2\epsilon n^{1/2}}\le\frac{d|x|}{2\epsilon s^{1/3}}\:.
\end{eqnarray*}
By taking $s$ to be large enough, we may assume that the last expression
is at most $\epsilon$. Thus,
\begin{eqnarray}
\sum_{k=0}^{n-1}\int\left|\int_{0}^{1/n}e^{i\left\langle c(t+\frac{k}{n}),sx\right\rangle }\:dt\right|\:d\sigma(x) & \le & \epsilon+\sum_{k=0}^{n-1}\int\left|\int_{0}^{1/n}e^{i\left\langle v_{k}+tu_{k},sx\right\rangle }\:dt\right|\:d\sigma(x)\nonumber \\
 & = & \epsilon+\sum_{k=0}^{n-1}\int\left|\int_{0}^{1/n}e^{i\left\langle tu_{k},sx\right\rangle }\:dt\right|\:d\sigma(x)\:.\label{eq:sum in <=00003D}
\end{eqnarray}
Additionally, for every $0\le k<n$,
\[
\int\left|\int_{0}^{1/n}e^{i\left\langle tu_{k},sx\right\rangle }\:dt\right|\:d\sigma(x)=\frac{1}{n}\int\left|\int_{0}^{1}e^{i\left\langle t\epsilon^{-1}u_{k},\epsilon sn^{-1}x\right\rangle }\:dt\right|\:d\sigma(x)\:.
\]
From this, since $|\epsilon^{-1}u_{k}|\ge1$, by Lemma \ref{lem:fur of line curves},
by (\ref{eq:s/n >}), and by assuming that $s$ is large enough,
\[
\int\left|\int_{0}^{1/n}e^{i\left\langle tu_{k},sx\right\rangle }\:dt\right|\:d\sigma(x)\le\epsilon/n\:.
\]
From this, (\ref{eq:<=00003D sum int}) and (\ref{eq:sum in <=00003D}),
\[
\int\left|\int_{0}^{1}e^{i\left\langle c(t),sx\right\rangle }\:dt\right|\:d\sigma(x)\le2\epsilon,
\]
which completes the proof of the lemma.
\end{proof}

\subsection{\label{subsec:The-case psi(G)=00003DR}The case $\psi(G)=\mathbb{R}$}

Recall that $\psi(t,U)=t$ for $(t,U)\in G$, that $N$ is the kernel
of $\psi$, and that $\mathbf{m}_{N}$ is the Haar measure of $N$
normalized so that $\mathbf{m}_{N}(N)=1$. By reordering the maps
$\{\varphi_{i}\}_{i=1}^{\ell}$ if necessary, we may assume that $\log r_{i}^{-1}\le\log r_{j}^{-1}$
for $1\le i<j\le\ell$. For $1\le i\le\ell$ set $b_{i}=\log r_{i}^{-1}$
and $\alpha_{i}=\sum_{k=i}^{\ell}p_{k}$, and write $b_{0}=0$. Let
$\rho_{0}:\mathbb{R}\rightarrow[0,\infty)$ be such that,
\[
\rho_{0}(t)=\begin{cases}
0 & \text{ for }t<0\text{ and }t\ge b_{\ell}\\
\alpha_{i}/\lambda & \text{ for }1\le i\le\ell\text{ and }b_{i-1}\le t<b_{i}
\end{cases}\:.
\]

\begin{lem}
\label{lem:nu =00003D double int}Suppose that $\psi(G)=\mathbb{R}$.
Then $\int\rho_{0}(t)\:dt=1$, and for every bounded and continuous
$f:G\rightarrow\mathbb{C}$
\[
\nu(f)=\int\int f(\gamma_{t}n)\rho_{0}(t)\:dt\:d\mathbf{m}_{N}(n)\:.
\]
\end{lem}

\begin{proof}
By the definition of $\rho$ (see Section \ref{subsec:Limit-distribution-of})
we have $\rho(g)=\rho_{0}(\psi g)$ for $g\in G$. Since $\psi(G)=\mathbb{R}$,
and by the way we defined $\mathbf{m}_{G}$ (see Section \ref{subsec:General-notations}),
it follows that $\psi\mathbf{m}_{G}$ is equal to the Lebesgue measure
$\mathbf{m}_{\mathbb{R}}$. Thus,
\[
1=\int\rho\:d\mathbf{m}_{G}=\int\rho_{0}\:d\psi\mathbf{m}_{G}=\int\rho_{0}(t)\:dt\:.
\]
Let $f:G\rightarrow\mathbb{C}$ be bounded and continuous. By Corollary
\ref{cor:G unimodular and m_G=00003D} and since $\mathbf{m}_{H}=\gamma\mathbf{m}_{\mathbb{R}}$,
\begin{eqnarray*}
\nu(f)=\int f\rho\:d\mathbf{m}_{G} & = & \int\int f(nh)\rho(nh)\:d\mathbf{m}_{N}(n)\:d\gamma\mathbf{m}_{\mathbb{R}}(h)\\
 & = & \int\int f(n\gamma_{t})\rho_{0}(\psi(n\gamma_{t}))\:d\mathbf{m}_{N}(n)\:dt\:.
\end{eqnarray*}
Since $\psi\circ\gamma=Id$ (see Lemma \ref{lem:proper homo from R})
we have $\psi(n\gamma_{t})=t$ for $n\in N$ and $t\in\mathbb{R}$.
Thus,
\[
\nu(f)=\int\int f(\gamma_{t}(\gamma_{t}^{-1}n\gamma_{t}))\rho_{0}(t)\:d\mathbf{m}_{N}(n)\:dt\:.
\]
Fort every $t\in\mathbb{R}$ the map $n\rightarrow\gamma_{t}^{-1}n\gamma_{t}$
is a continuous automorphism of $N$. Since $N$ is compact this automorphism
preserves $\mathbf{m}_{N}$. The lemma now follows from the last equality.
\end{proof}
\begin{prop}
\label{prop:case Psi(G)=00003DR}Recall that $\mu$ is the self-similar
measure corresponding to $\Phi$ and the positive probability vector
$p$. Suppose that $\psi(G)=\mathbb{R}$, then $\mu$ is a Rajchman
measure. That is, $\widehat{\mu}(\xi)\rightarrow0$ as $\xi\rightarrow\infty$.
\end{prop}

\begin{proof}
Let $\epsilon>0$, let $r>1$ be large with respect to $\epsilon$,
$\Phi$, $p$ and $\gamma$, and let $T>1$ be large with respect
to $r$. Fix $\xi\in\mathbb{R}^{d}$ with $|\xi|\ge2^{T}r$. We prove
the proposition by showing that $|\widehat{\mu}(\xi)|^{2}\le\epsilon$.

Write $\mathbb{S}^{d-1}$ for the unit sphere in $\mathbb{R}^{d}$.
Let $u\in\mathbb{S}^{d-1}$ and $t\ge T$ be such that $\xi=2^{t}ru$.
Note that since $\psi(G)=\mathbb{R}$, the domain of $\gamma$ is
$\mathbb{R}$. Let $U\in O(d)$ be such that $\gamma_{-t}=(-t,U)$.
By Lemma \ref{lem:initial upper bd}, and by assuming that $T$ is
large enough with respect to $r$ and $\epsilon$,
\[
|\widehat{\mu}(\xi)|^{2}=|\widehat{\mu}((rUu).\gamma_{-t})|^{2}\le\epsilon/2+\int\left|\int e^{i\left\langle (rUu).g,x\right\rangle }\:d\nu(g)\right|\:d\sigma(x)\:.
\]
From this, Lemma \ref{lem:nu =00003D double int} and the definition
of $\rho_{0}$,
\begin{eqnarray}
|\widehat{\mu}(\xi)|^{2} & \le & \epsilon/2+\int\int\left|\int e^{i\left\langle (Uu).(\gamma_{s}n),rx\right\rangle }\rho_{0}(s)\:ds\right|\:d\mathbf{m}_{N}(n)\:d\sigma(x)\nonumber \\
 & \le & \epsilon/2+\int\sum_{j=1}^{\ell}\frac{\alpha_{j}}{\lambda}\int\left|\int_{b_{j-1}}^{b_{j}}e^{i\left\langle (Uu).(\gamma_{s}n),rx\right\rangle }\:ds\right|\:d\sigma(x)\:d\mathbf{m}_{N}(n)\:.\label{eq:triple int initial bd}
\end{eqnarray}
For $1\le j\le\ell$, $n\in N$ and $v\in\mathbb{S}^{d-1}$, let $c_{j,n}^{v}:[0,1]\rightarrow\mathbb{R}^{d}$
be such that
\[
c_{j,n}^{v}(s):=v.(\gamma_{s(b_{j}-b_{j-1})+b_{j-1}}n)\:\text{ for }s\in[0,1]\:.
\]
From (\ref{eq:triple int initial bd}) we get,
\begin{equation}
|\widehat{\mu}(\xi)|^{2}\le\epsilon/2+\int\sum_{j=1}^{\ell}\frac{\alpha_{j}(b_{j}-b_{j-1})}{\lambda}\int\left|\int_{0}^{1}e^{i\left\langle c_{j,n}^{Uu}(s),rx\right\rangle }\:ds\right|\:d\sigma(x)\:d\mathbf{m}_{N}(n)\:.\label{eq:triple int bd with c}
\end{equation}

By Lemma \ref{lem:proper homo from R} and since $\psi(G)=\mathbb{R}$,
it follows that $\gamma$ is smooth. Hence, the curves $c_{j,n}^{v}$
are also smooth. Let $C>1$ be large with respect to $\{b_{j}\}_{j=0}^{\ell}$
and the curve $\gamma$. For $x\in\mathbb{R}^{d}$ set $f(x)=|x|^{2}$.
From $\psi\circ\gamma=Id$ and $N\subset\{0\}\times O(d)$, we get
that for every $1\le j\le\ell$, $n\in N$, $v\in\mathbb{S}^{d-1}$
and $s\in[0,1]$,
\[
f(c_{j,n}^{v}(s))=2^{-2(s(b_{j}-b_{j-1})+b_{j-1})}\:.
\]
By differentiating the last equality with respect to $s$ and by assuming
that $C$ is sufficiently large, it follows that for every $1\le j\le\ell$
with $b_{j}>b_{j-1}$,
\[
|\frac{d}{ds}c_{j,n}^{v}(s)|\ge C^{-1}\text{ for }n\in N,\:v\in\mathbb{S}^{d-1}\text{ and }s\in[0,1]\:.
\]
Additionally, by assuming that $C$ is sufficiently large,
\[
|\frac{d^{2}}{ds^{2}}c_{j,n}^{v}(s)|\le C\text{ for }1\le j\le\ell,\:n\in N,\:v\in\mathbb{S}^{d-1}\text{ and }s\in[0,1]\:.
\]
Hence, from Lemma \ref{lem:ub int sig of fur of curve mass}, from
(\ref{eq:triple int bd with c}), and by assuming as we may that $r$
is large enough with respect to $\epsilon$, $C$, $\{\alpha_{j}/\lambda\}_{j=1}^{\ell}$
and $\{b_{j}\}_{j=0}^{\ell}$, we get $|\widehat{\mu}(\xi)|^{2}\le\epsilon$.
This completes the proof of the proposition.
\end{proof}

\subsection{\label{subsec:Reduction-to-the disc case}Reduction to the discrete
case}

Throughout this subsection we assume that $\psi(G)\ne\mathbb{R}$.
Recall that we write $G_{0}$ for the connected component of $G$
containing the identity element. Since $G$ is a Lie group it is locally
path connected, and so $G_{0}$ is an open and closed normal subgroup
of $G$. From $\psi(G)\ne\mathbb{R}$ it follows that $\psi(G)$ is
a discrete subgroup of $\mathbb{R}$. This implies that $N$ is also
open and close in $G$, and so $G_{0}\subset N$. In particular $G_{0}$
is compact.

Let $\mathbb{V}$ be the linear subspace of $\mathbb{R}^{d}$ consisting
of all $x\in\mathbb{R}^{d}$ so that $x.g=x$ for all $g\in G_{0}$.
Recall that $\mathbb{V}^{\perp}$ denotes the orthogonal complement
of $\mathbb{V}$.
\begin{lem}
\label{lem:G inv subspaces}The subspaces $\mathbb{V}$ and $\mathbb{V}^{\perp}$
are $G$-invariant. That is, $v.g\in\mathbb{V}$ and $w.g\in\mathbb{V}^{\perp}$
for all $v\in\mathbb{V}$, $w\in\mathbb{V}^{\perp}$ and $g\in G$.
\end{lem}

\begin{proof}
Let $g_{0}\in G_{0}$, $g\in G$ and $v\in\mathbb{V}$ be given. Since
$G_{0}\triangleleft G$, there exists $g_{0}'\in G_{0}$ with $gg_{0}=g_{0}'g$.
Thus,
\[
(v.g).g_{0}=v.(gg_{0})=v.(g_{0}'g)=(v.g_{0}').g=v.g\:.
\]
This shows that $v.g\in\mathbb{V}$, and so $\mathbb{V}$ is $G$-invariant.
It is now obvious that $\mathbb{V}^{\perp}$ is also $G$-invariant.
\end{proof}
The purpose of this subsection is to prove the following proposition.
In Section \ref{sec:Proof-of-the main}, when we complete the proof
of our main result, it will enable us to make a reduction to the case
in which $G$ is discrete.
\begin{prop}
\label{prop:cont case psi(G) not R}Recall that $\mu$ is the self-similar
measure corresponding to $\Phi$ and the positive probability vector
$p$. Suppose that $\psi(G)\ne\mathbb{R}$. Then for every $\epsilon>0$
there exists $R=R(\epsilon,p)>1$ so that $|\widehat{\mu}(\xi)|<\epsilon$
for every $\xi\in\mathbb{R}^{d}$ with $|\pi_{\mathbb{V}^{\perp}}\xi|\ge\max\{R,\epsilon|\pi_{\mathbb{V}}\xi|\}$.
\end{prop}

We start working towards the proof of the proposition. If $\mathbb{V}=\mathbb{R}^{d}$
then the proposition holds trivially, and so we may assume that $\dim\mathbb{V}^{\perp}>0$.
Write $\mathbb{S}_{\mathbb{V}^{\perp}}$ for the unit sphere of $\mathbb{V}^{\perp}$.
That is $\mathbb{S}_{\mathbb{V}^{\perp}}:=\mathbb{V}^{\perp}\cap\mathbb{S}^{d-1}$,
where $\mathbb{S}^{d-1}$ is the unit sphere of $\mathbb{R}^{d}$.

Let $\mathfrak{g}_{0}$ be the Lie algebra of $G_{0}$. For the rest
of this section, fix some compact neighbourhood $B_{\mathfrak{g}_{0}}$
of $0$ in $\mathfrak{g}_{0}$. That is, $B_{\mathfrak{g}_{0}}$ is
a compact subset of $\mathfrak{g}_{0}$ and $0\in\mathrm{Int}(B_{\mathfrak{g}_{0}})$.
Since $B_{\mathfrak{g}_{0}}$ is fixed, usually the dependence of
various parameters on $B_{\mathfrak{g}_{0}}$ will not be indicated.
For $X\in\mathfrak{g}_{0}$ and $y\in\mathbb{R}^{d}$ let $c_{X,y}:\mathbb{R}\rightarrow\mathbb{R}^{d}$
be such that
\[
c_{X,y}(t)=y.\exp(tX)\text{ for }t\in\mathbb{R},
\]
where $\exp:\mathfrak{g}_{0}\rightarrow G_{0}$ is the exponential
map of $G_{0}$. 
\begin{lem}
\label{lem:lb on der of cur}There exists $\delta>0$ so that for
every $y\in\mathbb{S}_{\mathbb{V}^{\perp}}$ there exists $X\in B_{\mathfrak{g}_{0}}$,
such that $|c_{X,y}'(t)|\ge\delta$ for all $t\in\mathbb{R}$.
\end{lem}

\begin{proof}
For $y\in\mathbb{S}_{\mathbb{V}^{\perp}}$, $X\in B_{\mathfrak{g}_{0}}$
and $t,s\in\mathbb{R}$,
\[
c_{X,y}(t+s)=(y.\exp(sX)).\exp(tX)=c_{X,y}(s).\exp(tX)\:.
\]
Differentiating with respect to $s$ at $s=0$ we get
\begin{equation}
c_{X,y}'(t)=c_{X,y}'(0).\exp(tX)\:.\label{eq:inv of der}
\end{equation}
Since $\exp(tX)\in G_{0}\subset N$ and $N\subset\{0\}\times O(d)$,
it follows that $|c_{X,y}'(t)|=|c_{X,y}'(0)|$. Thus, it suffices
to show that there exists $\delta>0$ so that
\begin{equation}
\text{for every }y\in\mathbb{S}_{\mathbb{V}^{\perp}}\text{ there exists }X\in B_{\mathfrak{g}_{0}}\text{ with }|c_{X,y}'(0)|\ge\delta.\label{eq:cont assum}
\end{equation}

Assume by contradiction that such a $\delta$ does not exist. Let
$M:\mathbb{S}_{\mathbb{V}^{\perp}}\rightarrow[0,\infty)$ be with
\[
M(y)=\max_{X\in B_{\mathfrak{g}_{0}}}\:|c_{X,y}'(0)|\text{ for }y\in\mathbb{S}_{\mathbb{V}^{\perp}}\:.
\]
Since there does not exist $\delta>0$ which satisfies (\ref{eq:cont assum}),
we have
\begin{equation}
\inf_{y\in\mathbb{S}_{\mathbb{V}^{\perp}}}\:M(y)=0\:.\label{eq:inf Q =00003D 0}
\end{equation}
Since $B_{\mathfrak{g}_{0}}$ is compact and since the map which takes
$(X,y)\in\mathfrak{g}_{0}\times\mathbb{S}_{\mathbb{V}^{\perp}}$ to
$|c_{X,y}'(0)|$ is continuous, it is easy to check that $M$ is also
continuous. From this, from (\ref{eq:inf Q =00003D 0}) and since
$\mathbb{S}_{\mathbb{V}^{\perp}}$ is compact, it follows that there
exists $y_{0}\in\mathbb{S}_{\mathbb{V}^{\perp}}$ so that $M(y_{0})=0$.
That is, $c_{X,y_{0}}'(0)=0$ for all $X\in B_{\mathfrak{g}_{0}}$.
Hence, by (\ref{eq:inv of der}) we have for all $t\in\mathbb{R}$
\[
c_{X,y_{0}}'(t)=c_{X,y_{0}}'(0).\exp(tX)=0\:.
\]
This shows that for all $X\in B_{\mathfrak{g}_{0}}$ and $t\in\mathbb{R}$,
\[
y_{0}=c_{X,y_{0}}(0)=c_{X,y_{0}}(t)=y_{0}.\exp(tX)\:.
\]
Since $0\in\mathrm{Int}(B_{\mathfrak{g}_{0}})$ we have $\mathfrak{g}_{0}=\cup_{t>0}tB_{\mathfrak{g}_{0}}$,
and so $y_{0}=y_{0}.\exp(X)$ for all $X\in\mathfrak{g}_{0}$. Additionally,
since $G_{0}$ is a compact and connected Lie group its exponential
map is surjective (see e.g. \cite[Corollary 11.10]{Ha}), that is
$\exp(\mathfrak{g}_{0})=G_{0}$. Thus $y_{0}=y_{0}.g$ for all $g\in G_{0}$,
which gives $y_{0}\in\mathbb{V}$. This contradicts $y_{0}\in\mathbb{S}_{\mathbb{V}^{\perp}}$
and completes the proof of the lemma.
\end{proof}
Given a compact Hausdorff topological group $Y$, let $\mathbf{m}_{Y}$
be its Haar measure normalized so that $\mathbf{m}_{Y}(Y)=1$. The
following equidistribution lemma is standard. We provide the simple
proof for completeness.
\begin{lem}
\label{lem:equi dist of flow}Let $X\in\mathfrak{g}_{0}$, and let
$G_{1}$ be the smallest closed subgroup of $G_{0}$ containing $\{\exp(tX)\}_{t\in\mathbb{R}}$.
Then for every $f\in C_{c}(G_{1})$,
\[
\mathbf{m}_{G_{1}}(f)=\underset{T\rightarrow\infty}{\lim}\frac{1}{T}\int_{0}^{T}f(\exp(tX))\:dt\:.
\]
\end{lem}

\begin{proof}
Since $G_{1}$ is compact and separable, $\mathcal{M}(G_{1})$ is
compact and metrizable with respect to the weak-{*} topology. For
$T>0$ let $\alpha_{T}\in\mathcal{M}(G_{1})$ be such that,
\[
\alpha_{T}(f)=\frac{1}{T}\int_{0}^{T}f(\exp(tX))\:dt\text{ for }f\in C_{c}(G_{1})\:.
\]
Let $\{T_{k}\}_{k\ge1}\subset(0,\infty)$ and $\alpha\in\mathcal{M}(G_{1})$
be with $T_{k}\overset{k}{\rightarrow}\infty$ and $\alpha_{T_{k}}\overset{k}{\rightarrow}\alpha$
in the weak-{*} topology. In order to prove the lemma it suffices
to show that $\alpha=\mathbf{m}_{G_{1}}$.

Let $s\in\mathbb{R}$, and write $g=\exp(sX)$. Recall that we write
$L_{g}h=gh$ for $h\in G$. For $f\in C_{c}(G_{1})$,
\begin{eqnarray*}
|\alpha(f)-L_{g}\alpha(f)| & = & \underset{k\rightarrow\infty}{\lim}\:|\alpha_{T_{k}}(f)-\alpha_{T_{k}}(f\circ L_{g})|\\
 & = & \underset{k\rightarrow\infty}{\lim}\:\frac{1}{T_{k}}\left|\int_{0}^{T_{k}}f(\exp(tX))\:dt-\int_{s}^{s+T_{k}}f(\exp(tX))\:dt\right|=0\:.
\end{eqnarray*}
Thus $\alpha=L_{g}\alpha$ for all $g\in\{\exp(tX)\}_{t\in\mathbb{R}}$.
Since $G_{1}$ is the closure of $\{\exp(tX)\}_{t\in\mathbb{R}}$,
it follows that $\alpha=L_{g}\alpha$ for all $g\in G_{1}$. Hence
$\alpha$ is a Haar measure for $G_{1}$. Since it is also a probability
measure we get $\alpha=\mathbf{m}_{G_{1}}$, which completes the proof
of the lemma.
\end{proof}
\begin{lem}
\label{lem:exists G_1}For every $\epsilon>0$ there exists $R>1$
so that the following holds. Let $y\in\mathbb{S}_{\mathbb{V}^{\perp}}$,
then there exists a closed subgroup $G_{1}$ of $G_{0}$ so that for
all $r\ge R$ and $g_{0}\in G_{0}$,
\[
\int\left|\int e^{i\left\langle y.gg_{0},rx\right\rangle }\:d\mathbf{m}_{G_{1}}(g)\right|\:d\sigma(x)<\epsilon\:.
\]
\end{lem}

\begin{proof}
It follows from (\ref{eq:inv of der}) that for $X\in B_{\mathfrak{g}_{0}}$,
$y\in\mathbb{S}_{\mathbb{V}^{\perp}}$ and $s,t\in\mathbb{R}$,
\[
c_{X,y}'(t+s)=(c_{X,y}'(0).\exp(sX)).\exp(tX)=c_{X,y}'(s).\exp(tX)\:.
\]
Differentiating with respect to $s$ at $s=0$ and using $\exp(tX)\in G_{0}\subset\{0\}\times O(d)$,
we get $|c_{X,y}''(t)|=|c_{X,y}''(0)|$. From this and since $B_{\mathfrak{g}_{0}}$
and $\mathbb{S}_{\mathbb{V}^{\perp}}$ are compact, it follows that
there exists a constant $C>1$, depending only on $B_{\mathfrak{g}_{0}}$,
so that $|c_{X,y}''(t)|\le C$ for all $X\in B_{\mathfrak{g}_{0}}$,
$y\in\mathbb{S}_{\mathbb{V}^{\perp}}$ and $t\in\mathbb{R}$.

Let $\delta>0$ be as obtained in Lemma \ref{lem:lb on der of cur}.
Let $\epsilon>0$, and let $r>1$ be large with respect to $\epsilon$,
$\delta$ and $C$. Let $y\in\mathbb{S}_{\mathbb{V}^{\perp}}$, then
there exists $X\in B_{\mathfrak{g}_{0}}$ so that $|c_{X,y}'(t)|\ge\delta$
for all $t\in\mathbb{R}$. Let $G_{1}$ be the smallest closed subgroup
of $G_{0}$ containing $\{\exp(tX)\}_{t\in\mathbb{R}}$. For $g_{0}\in G_{0}$
and $k\in\mathbb{Z}_{\ge0}$, let $c_{X,y}^{g_{0},k}:[0,1]\rightarrow\mathbb{R}^{d}$
be with
\[
c_{X,y}^{g_{0},k}(t)=c_{X,y}(k+t).g_{0}\text{ for }t\in[0,1]\:.
\]
Let $g_{0}\in G_{0}$, then by Lemma \ref{lem:equi dist of flow},
\begin{eqnarray}
\int\left|\int e^{i\left\langle y.gg_{0},rx\right\rangle }\:d\mathbf{m}_{G_{1}}(g)\right|\:d\sigma(x) & = & \underset{T\rightarrow\infty}{\lim}\frac{1}{T}\int\left|\int_{0}^{T}e^{i\left\langle c_{X,y}(t).g_{0},rx\right\rangle }\:dt\right|\:d\sigma(x)\nonumber \\
 & \le & \underset{M\rightarrow\infty}{\limsup}\frac{1}{M}\sum_{k=0}^{M-1}\int\left|\int_{0}^{1}e^{i\left\langle c_{X,y}^{g_{0},k}(t),rx\right\rangle }\:dt\right|\:d\sigma(x)\:.\label{eq:ub by lim of avg}
\end{eqnarray}
Since $g_{0}\in\{0\}\times O(d)$, we have for $k\in\mathbb{Z}_{\ge0}$
and $0\le t\le1$
\[
|\frac{d}{dt}c_{X,y}^{g_{0},k}(t)|=|c_{X,y}'(k+t)|\ge\delta\:\text{ and }\:|\frac{d^{2}}{dt^{2}}c_{X,y}^{g_{0},k}(t)|=|c_{X,y}''(k+t)|\le C\:.
\]
From this, from Lemma \ref{lem:ub int sig of fur of curve mass},
by assuming that $r$ is large enough, and by (\ref{eq:ub by lim of avg})
\[
\int\left|\int e^{i\left\langle y.gg_{0},rx\right\rangle }\:d\mathbf{m}_{G_{1}}(g)\right|\:d\sigma(x)<\epsilon,
\]
which completes the proof of the lemma.
\end{proof}
\begin{lem}
\label{lem:ub int m_G_0 =000026 sig}For every $\epsilon>0$ there
exists $R>1$ so that for all $y\in\mathbb{S}_{\mathbb{V}^{\perp}}$
and $r\ge R$,
\[
\int\left|\int e^{i\left\langle y.g_{0},rx\right\rangle }\:d\mathbf{m}_{G_{0}}(g_{0})\right|\:d\sigma(x)<\epsilon\:.
\]
\end{lem}

\begin{proof}
Let $\epsilon>0$, let $r>1$ be large with respect to $\epsilon$,
and let $y\in\mathbb{S}_{\mathbb{V}^{\perp}}$. By Lemma \ref{lem:exists G_1}
and by assuming that $r$ is large enough, it follows that there exists
a closed subgroup $G_{1}$ of $G_{0}$ so that for all $g_{0}\in G_{0}$
\begin{equation}
\int\left|\int e^{i\left\langle y.gg_{0},rx\right\rangle }\:d\mathbf{m}_{G_{1}}(g)\right|\:d\sigma(x)<\epsilon\:.\label{eq:from prev lem}
\end{equation}
Additionally, for all $f\in C_{c}(G_{0})$
\[
\int f\:d\mathbf{m}_{G_{0}}=\int\int f(gg_{0})\:d\mathbf{m}_{G_{1}}(g)\:d\mathbf{m}_{G_{0}}(g_{0})\:.
\]
Hence,
\begin{eqnarray*}
\int\left|\int e^{i\left\langle y.g_{0},rx\right\rangle }\:d\mathbf{m}_{G_{0}}(g_{0})\right|d\sigma(x) & = & \int\left|\int\int e^{i\left\langle y.gg_{0},rx\right\rangle }\:d\mathbf{m}_{G_{1}}(g)\:d\mathbf{m}_{G_{0}}(g_{0})\right|d\sigma(x)\\
 & \le & \int\int\left|\int e^{i\left\langle y.gg_{0},rx\right\rangle }\:d\mathbf{m}_{G_{1}}(g)\right|d\sigma(x)\:d\mathbf{m}_{G_{0}}(g_{0})\:.
\end{eqnarray*}
This together with (\ref{eq:from prev lem}) completes the proof of
the lemma.
\end{proof}
\begin{proof}[Proof of Proposition \ref{prop:cont case psi(G) not R}]
Let $\epsilon>0$, let $r>1$ be large with respect to $\epsilon$,
$p$ and $\Phi$, and let $T>1$ be large with respect to $r$. Fix
$\xi\in\mathbb{R}^{d}$ with
\[
|\pi_{\mathbb{V}^{\perp}}\xi|\ge\max\{2^{T}r,\epsilon|\pi_{\mathbb{V}}\xi|\}\:.
\]
Since $\psi(G)\ne\mathbb{R}$, there exists $\beta>0$ so that $\psi(G)=\beta\mathbb{Z}$.
Let $t\in\psi(G_{\ge T})$ be such that
\begin{equation}
2^{t-\beta}r<|\pi_{\mathbb{V}^{\perp}}\xi|\le2^{t}r\:.\label{eq:choise of t}
\end{equation}
Since $\psi\gamma_{t}=t$,
\[
|\xi.\gamma_{t}|=|2^{-t}\xi|\le2^{-t}(|\pi_{\mathbb{V}^{\perp}}\xi|+|\pi_{\mathbb{V}}\xi|)\le2^{-t}|\pi_{\mathbb{V}^{\perp}}\xi|(1+\epsilon^{-1})\le r(1+\epsilon^{-1})\:.
\]
By Lemma \ref{lem:initial upper bd}, from $|\xi.\gamma_{t}|\le r(1+\epsilon^{-1})$,
and by assuming that $T$ is large enough with respect to $r$ and
$\epsilon$,
\[
|\widehat{\mu}(\xi)|^{2}=|\widehat{\mu}((\xi.\gamma_{t}).\gamma_{-t})|^{2}\le\epsilon+\int\left|\int e^{i\left\langle \xi.(\gamma_{t}g),x\right\rangle }\:d\nu(g)\right|\:d\sigma(x)\:.
\]
From this, since $\nu=\rho\:d\mathbf{m}_{G}$, and since
\[
\mathbf{m}_{G}(f)=\int\int f(gg_{0})\:d\mathbf{m}_{G_{0}}(g_{0})\:d\mathbf{m}_{G}(g)\text{ for }f\in C_{c}(G),
\]
we get,
\begin{eqnarray}
|\widehat{\mu}(\xi)|^{2} & \le & \epsilon+\int\left|\int e^{i\left\langle \xi.(\gamma_{t}g),x\right\rangle }\rho(g)\:d\mathbf{m}_{G}(g)\right|\:d\sigma(x)\nonumber \\
 & \le & \epsilon+\int\rho(g)\int\left|\int e^{i\left\langle \xi.(\gamma_{t}gg_{0}),x\right\rangle }\:d\mathbf{m}_{G_{0}}(g_{0})\right|\:d\sigma(x)\:d\mathbf{m}_{G}(g)\;.\label{eq:ub 3 int G_0 inside}
\end{eqnarray}
We have also used here the fact that $\rho(gg_{0})=\rho(g)$ for $g\in G$
and $g_{0}\in G_{0}$, which holds since $G_{0}\subset N$.

Let $g\in G$ and write $w_{g}=(\pi_{\mathbb{V}^{\perp}}\xi).\gamma_{t}g$
and $y_{g}=w_{g}/|w_{g}|$. It follows from Lemma \ref{lem:G inv subspaces}
that the subspaces $\mathbb{V}$ and $\mathbb{V}^{\perp}$ are $\gamma_{t}g$-invariant.
Hence,
\[
\pi_{\mathbb{V}^{\perp}}(\xi.\gamma_{t}g)=\pi_{\mathbb{V}^{\perp}}((\pi_{\mathbb{V}}\xi).\gamma_{t}g+(\pi_{\mathbb{V}^{\perp}}\xi).\gamma_{t}g)=w_{g}\:.
\]
Additionally, given $g_{0}\in G_{0}$ it follows by the definition
of $\mathbb{V}$ that $v.g_{0}=v$ for $v\in\mathbb{V}$. Thus,
\[
\xi.(\gamma_{t}gg_{0})=(\pi_{\mathbb{V}}(\xi.\gamma_{t}g)+\pi_{\mathbb{V}^{\perp}}(\xi.\gamma_{t}g)).g_{0}=\pi_{\mathbb{V}}(\xi.\gamma_{t}g)+w_{g}.g_{0}\:.
\]
From this and (\ref{eq:ub 3 int G_0 inside}),
\begin{equation}
|\widehat{\mu}(\xi)|^{2}\le\epsilon+\int\rho(g)\int\left|\int e^{i\left\langle y_{g}.g_{0},|w_{g}|x\right\rangle }\:d\mathbf{m}_{G_{0}}(g_{0})\right|\:d\sigma(x)\:d\mathbf{m}_{G}(g)\:.\label{eq:ub with y_g =000026 w_g}
\end{equation}
By the definition of $\rho$ there exists a constant $C>1$, which
depends only on $\Phi$, so that $\psi g\le C$ for all $g\in G$
with $\rho(g)\ne0$. For such a $g$ it follows by (\ref{eq:choise of t})
that,
\[
|w_{g}|=|(\pi_{\mathbb{V}^{\perp}}\xi).\gamma_{t}g|=2^{-\psi g}2^{-t}|\pi_{\mathbb{V}^{\perp}}\xi|>2^{-C-\beta}r\:.
\]
Thus, from Lemma \ref{lem:ub int m_G_0 =000026 sig}, since $y_{g}\in\mathbb{S}_{\mathbb{V}^{\perp}}$,
and by assuming as we may that $r$ is sufficiently large with respect
to $\epsilon$, $\beta$ and $C$, we get
\[
\int\left|\int e^{i\left\langle y_{g}.g_{0},|w_{g}|x\right\rangle }\:d\mathbf{m}_{G_{0}}(g_{0})\right|\:d\sigma(x)\le\epsilon\:.
\]
This together with (\ref{eq:ub with y_g =000026 w_g}) shows that
$|\widehat{\mu}(\xi)|^{2}\le2\epsilon$, which completes the proof
of the Proposition.
\end{proof}

\section{\label{sec:The-discrete-case}The discrete case}

Throughout this section we always assume that the group $G$ is discrete.
That is, we assume that the subspace topology on $G$, inherited from
$\mathbb{R}\times O(d)$, is equal to the discrete topology. Since
$N$ is a compact subset of a discrete space, it holds that $N$ is
finite. From the discreteness of $G$ it also follows that there exists
$\beta>0$ with $\psi(G)=\beta\mathbb{Z}$. Hence, $l_{i}:=\psi(g_{i})/\beta$
is a positive integer for all $1\le i\le\ell$.

Fix some $h\in G$ with $\psi(h)=\beta$, and recall that $H:=\gamma\circ\psi(G)$.
By Lemma \ref{lem:proper homo from R} we may assume that $\gamma_{\beta}=h$,
which gives $H=\{h^{j}\}_{j\in\mathbb{Z}}$. By Lemma \ref{lem:G splits and F iso}
and since $N\triangleleft G$, it follows that for every $g\in G$
\begin{equation}
g=nh^{l}=h^{l}n'\text{ for some }l\in\mathbb{Z}\text{ and }n,n'\in N\:.\label{eq:rep of each g}
\end{equation}
Let $U\in O(d)$ be such that $h=(\beta,U)$. Set $A:=2^{-\beta}U$
and $B:=A^{*}$, where $A^{*}=2^{-\beta}U^{-1}$ is the transpose
of $A$. Write, 
\[
N_{0}:=\{V\in O(d)\::\:(0,V)\in N\},
\]
so that $N_{0}$ is a finite subgroup of $O(d)$. Since $N\triangleleft G$,
we also have $A^{-1}N_{0}A=N_{0}$. From (\ref{eq:rep of each g})
it follows that for each $1\le i\le\ell$,
\begin{equation}
r_{i}U_{i}=VA^{l_{i}}=A^{l_{i}}V'\text{ for some }V,V'\in N_{0}\:.\label{eq:rep of riUi}
\end{equation}

The following proposition is the main result of this section. Its
proof is a nontrivial extension of an argument used in \cite{VY}
to prove one of the directions of Theorem \ref{thm:bermont} stated
in the introduction. Recall that $\Phi=\{\varphi_{i}(x)=r_{i}U_{i}x+a_{i}\}_{i=1}^{\ell}$
is an affinely irreducible self-similar IFS on $\mathbb{R}^{d}$,
and recall from Section \ref{subsec:The-main-result.} the definition
of a P.V. $k$-tuple. We shall consider $A$ as a linear operator
on $\mathbb{C}^{d}$ in the natural way, that is by setting $A(x+iy):=Ax+iAy$
for $x,y\in\mathbb{R}^{d}$.
\begin{prop}
\label{prop:main disc case}Suppose that $G$ is discrete and that
$a_{1}=0$. Moreover, assume that there exists a probability vector
$p=(p_{i})_{i=1}^{\ell}>0$ so that the self-similar measure corresponding
to $\Phi$ and $p$ is non-Rajchman. Let $A$ and $N_{0}$ be as defined
above. Then there exist $k\ge1$, $\theta_{1},...,\theta_{k}\in\mathbb{C}$
and $\zeta_{1},...,\zeta_{k}\in\mathbb{C}^{d}\setminus\{0\}$, so
that
\begin{enumerate}
\item \label{enu:dic P.V. tup}$\{\theta_{1},...,\theta_{k}\}$ is a P.V.
$k$-tuple;
\item \label{enu:disc eigen vec}$A^{-1}\zeta_{j}=\theta_{j}\zeta_{j}$
for $1\le j\le k$;
\item \label{enu:poly disc case}for every $1\le i\le\ell$ and $V\in N_{0}$
there exists $P_{i,V}\in\mathbb{Q}[X]$ so that $\left\langle Va_{i},\zeta_{j}\right\rangle =P_{i,V}(\theta_{j})$
for all $1\le j\le k$;
\end{enumerate}
\end{prop}

The assumption $a_{1}=0$ might seem somewhat arbitrary. It simplifies
the statement of condition (\ref{enu:poly disc case}), and some of
the arguments that follow.

The proof of the proposition is carried out in Sections \ref{subsec:a pre prop}
and \ref{subsec:Proof-of-Proposition disc case}. In Section \ref{subsec:Construction-of-non-Rajchman}
we state and prove a converse to it.

\subsection{\label{subsec:a pre prop}A preliminary proposition}

Throughout this subsection let $p=(p_{i})_{i=1}^{\ell}$ be a fixed
positive probability vector. Let $\mu$ be the self-similar measure
corresponding to $\Phi$ and $p$. Recall that $G$ is assumed to
be discrete, and that we write $B$ in place of $A^{*}$. For a real
number $x$ let $\Vert x\Vert$ be the distance from $x$ to its nearest
integer, that is
\[
\Vert x\Vert:=\inf\{|x-k|\::\:k\in\mathbb{Z}\}\:.
\]
Recall that $\Lambda:=\{1,...,\ell\}$, and that a finite set of words
$\mathcal{W}\subset\Lambda^{*}$ is said to be a minimal cut-set for
$\Lambda^{*}$ if every infinite sequence in $\Lambda^{\mathbb{N}}$
has a unique prefix in $\mathcal{W}$. The purpose of this subsection
is to prove the following proposition.
\begin{prop}
\label{prop:ub on sum of dist to int}Let $\mathcal{W}$ be a minimal
cut-set for $\Lambda^{*}$, and let $u,u'\in\mathcal{W}$. Suppose
that $G$ is generated by $\{g_{w}\}_{w\in\mathcal{W}}$, and that
$g_{u}=g_{u'}$. Then for every $\epsilon>0$ there exists $C=C(\epsilon,\mathcal{W},p)>1$
so that for all $V\in N_{0}$,
\[
\sum_{j\ge0}\left\Vert \left\langle V(\varphi_{u}(0)-\varphi_{u'}(0)),B^{j}\xi\right\rangle \right\Vert ^{2}\le C\text{ for }\xi\in\mathbb{R}^{d}\text{ with }|\widehat{\mu}(2\pi\xi)|\ge\epsilon\:.
\]
\end{prop}

For the rest of this subsection fix $\mathcal{W}\subset\Lambda^{*}$
and $u,u'\in\mathcal{W}$ as in the statement of the proposition.
Note that since $\mathcal{W}$ is a minimal cut-set, $(p_{w})_{w\in\mathcal{W}}$
is a probability vector. Let $I_{1},I_{2},...$ be i.i.d. $\mathcal{W}$-valued
random words with $\mathbb{P}\{I_{1}=w\}=p_{w}$ for $w\in\mathcal{W}$.
Set $Y_{0}=1_{G}$, and for $k\ge1$ let $X_{k}:=g_{I_{k}}$, $Y_{k}:=X_{1}\cdot...\cdot X_{k}$,
and
\[
\tau_{\beta}(k):=\inf\{m\ge1\::\:\psi Y_{m}\ge k\beta\}\:.
\]
For $\xi\in\mathbb{R}^{d}$ and $w\in\{u,u'\}$ set,
\[
\alpha_{w}(\xi):=\frac{1}{p_{u}+p_{u'}}\left|p_{u}e^{i\left\langle \xi,\varphi_{u}(0)\right\rangle }+p_{u'}e^{i\left\langle \xi,\varphi_{u'}(0)\right\rangle }\right|,
\]
and for $w\in\mathcal{W}\setminus\{u,u'\}$ write $\alpha_{w}(\xi):=1$.
Set $Z_{\xi,0}:=1$, and for $n\ge1$ let
\[
Z_{\xi,n}:=\prod_{k=1}^{n}\alpha_{I_{k}}(\xi.Y_{k-1})\:.
\]

\begin{lem}
\label{lem:from doob's}For $k\ge1$ and $\xi\in\mathbb{R}^{d}$ we
have $|\widehat{\mu}(\xi)|\le\mathbb{E}[Z_{\xi,\tau_{\beta}(k)}]$.
\end{lem}

\begin{proof}
Since $\mathcal{W}$ is a minimal cut-set and since $g_{u}=g_{u'}$,
it follows that for $y\in\mathbb{R}^{d}$
\begin{eqnarray*}
|\widehat{\mu}(y)| & = & \left|\sum_{w\in\mathcal{W}}p_{w}\int e^{i\left\langle y,\varphi_{w}(x)\right\rangle }\:d\mu(x)\right|\\
 & = & \left|\sum_{w\in\mathcal{W}}p_{w}e^{i\left\langle y,\varphi_{w}(0)\right\rangle }\widehat{\mu}(y.g_{w})\right|\le\sum_{w\in\mathcal{W}}p_{w}|\widehat{\mu}(y.g_{w})|\cdot\alpha_{w}(y)\:.
\end{eqnarray*}
Let $n\ge0$, then by applying the last inequality with $y=\xi.Y_{n}$,
\begin{eqnarray*}
Z_{\xi,n}|\widehat{\mu}(\xi.Y_{n})| & \le & Z_{\xi,n}\sum_{w\in\mathcal{W}}p_{w}|\widehat{\mu}((\xi.Y_{n}).g_{w})|\cdot\alpha_{w}(\xi.Y_{n})\\
 & = & \mathbb{E}\left[Z_{\xi,n}|\widehat{\mu}((\xi.Y_{n}).g_{I_{n+1}})|\cdot\alpha_{I_{n+1}}(\xi.Y_{n})\:\Bigl|\:I_{1},...,I_{n}\right]\\
 & = & \mathbb{E}\left[Z_{\xi,n+1}|\widehat{\mu}(\xi.Y_{n+1})|\:\Bigl|\:I_{1},...,I_{n}\right]\:.
\end{eqnarray*}
This shows that $\{Z_{\xi,n}|\widehat{\mu}(\xi.Y_{n})|\}_{n\ge0}$
is a submartingale with respect to the filtration $\{\mathcal{F}_{n}\}_{n\ge0}$,
where $\mathcal{F}_{n}$ is the $\sigma$-algebra generated by $I_{1},...,I_{n}$.
Thus, since $\tau_{\beta}(k)$ is a bounded stopping time with respect
to the filtration $\{\mathcal{F}_{n}\}_{n\ge0}$, and by Doob's optional
stopping theorem, we get
\[
|\widehat{\mu}(\xi)|=\mathbb{E}\left[Z_{\xi,0}|\widehat{\mu}(\xi.Y_{0})|\right]\le\mathbb{E}\left[Z_{\xi,\tau_{\beta}(k)}|\widehat{\mu}(\xi.Y_{\tau_{\beta}(k)})|\right]\le\mathbb{E}\left[Z_{\xi,\tau_{\beta}(k)}\right],
\]
which completes the proof of the lemma.
\end{proof}
\begin{lem}
\label{lem:lb on prob =00003D g}There exists a constant $C=C(\mathcal{W},p)>1$
so that,
\[
\mathbb{P}\left\{ \gamma_{-\beta k}Y_{\tau_{\beta}(k)}=g\right\} >C^{-1}\text{ for every integer }k\ge C\text{ and }g\in N\:.
\]
\end{lem}

\begin{proof}
Set $q:=\sum_{w\in\mathcal{W}}p_{w}\delta_{g_{w}}$ and $\lambda:=\int\psi\:dq$.
For $g\in G$ let,
\[
\rho(g):=\lambda^{-1}\mathbb{P}\{\psi X_{1}>\psi g\ge0\},
\]
and write $\nu$ in place of $\rho\:d\mathbf{m}_{G}$. Since $\psi(G)=\beta\mathbb{Z}$
and $|N|<\infty$, it follows by our choice of $\mathbf{m}_{G}$ (see
Section \ref{subsec:General-notations}) that $\mathbf{m}_{G}\{g\}=\beta/|N|$
for $g\in G$. For $g\in N$ we have $\psi(g)=0$, and so $\rho(g)=\lambda^{-1}$.
Since $G$ is generated by $\{g_{w}\}_{w\in\mathcal{W}}$ it holds
that $q$ is adapted, and so we can apply Proposition \ref{prop:conv in dist}.
It follows that for $g\in N$,
\[
\underset{k\rightarrow\infty}{\lim}\mathbb{P}\left\{ \gamma_{-\beta k}Y_{\tau_{\beta}(k)}=g\right\} =\nu\{g\}=\rho(g)\mathbf{m}_{G}\{g\}=\beta/(\lambda|N|)\:.
\]
Since $N$ is finite this completes the proof of the lemma.
\end{proof}
Recall that $\psi(g_{i})/\beta=:l_{i}\in\mathbb{Z}_{>0}$ for $1\le i\le\ell$.
Given a word $i_{1}...i_{n}=w\in\Lambda^{*}$ we write $l_{w}$ in
place of $l_{i_{1}}+...+l_{i_{n}}$.
\begin{lem}
\label{lem:ub on cond exp}There exists an integer $C=C(\mathcal{W},p)>1$
so that for every $k\in\mathbb{Z}_{>C}$, $V\in N_{0}$ and $\xi\in\mathbb{R}^{d}$,
\[
\mathbb{E}\left[Z_{\xi,\tau_{\beta}(k)}\right]\le\mathbb{E}\left[Z_{\xi,\tau_{\beta}(k-C)}\right]\left(1-C^{-1}(1-\alpha_{u}(VB^{k-l_{u}}\xi))\right)\:.
\]
\end{lem}

\begin{proof}
Let $C\in\mathbb{Z}_{>1}$ be large with respect to $\Phi$, $p$
and $\mathcal{W}$. Set $l_{\mathrm{max}}=\max_{w\in\mathcal{W}}l_{w}$,
and suppose that $C>2l_{\mathrm{max}}$. Fix $k\in\mathbb{Z}_{>C}$,
$V\in N_{0}$ and $\xi\in\mathbb{R}^{d}$. Let $n_{V}\in N$ be with
$n_{V}=(0,V^{-1})$. Denote by $\mathcal{W}^{*}$ the set of finite
words over $\mathcal{W}$. For $w_{1}...w_{m}=\mathbf{w}\in\mathcal{W}^{*}$
we write,
\[
\:g_{\mathbf{w}}:=g_{w_{1}}\cdot...\cdot g_{w_{m}}\text{ and }l_{\mathbf{w}}:=l_{w_{1}}+...+l_{w_{m}}\:.
\]
Let,
\[
\mathcal{Y}:=\{w_{1}...w_{m}\in\mathcal{W}^{*}\::\:\psi(g_{w_{1}...w_{m}})\ge\beta(k-C)>\psi(g_{w_{1}...w_{m-1}})\}\:.
\]
For $\mathbf{y}\in\mathcal{Y}$ set,
\[
\eta_{\mathbf{y}}:=\mathbb{P}\left\{ Y_{\tau_{\beta}(k-l_{u})}=h^{k-l_{u}}n_{V}\mid I_{1}...I_{\tau_{\beta}(k-C)}=\mathbf{y}\right\} \:.
\]
For $m\in\mathbb{Z}_{\ge0}$ and $b\in\mathbb{Z}_{\ge1}$ write,
\[
\tau_{\beta,m}(b):=\inf\{j>m\::\:\psi(X_{m+1}\cdot...\cdot X_{j})\ge b\beta\}\:.
\]

Fix $w_{1}...w_{m}=\mathbf{y}\in\mathcal{Y}$ for the moment. From
(\ref{eq:rep of each g}) and since $\psi(h)=\beta$, it follows that
there exists $n_{\mathbf{y}}\in N$ with $g_{\mathbf{y}}=h^{l_{\mathbf{y}}}n_{\mathbf{y}}$.
Additionally, by the definition of $\mathcal{Y}$,
\begin{equation}
k-C\le l_{\mathbf{y}}<k-C+l_{\mathrm{max}}<k-l_{\mathrm{max}}\:.\label{eq:est on l_y}
\end{equation}
Note that,
\[
\mathbb{P}\left\{ Y_{m}=g_{\mathbf{y}}\mid I_{1}...I_{\tau_{\beta}(k-C)}=\mathbf{y}\right\} =1\:.
\]
From this and since $\psi(g_{\mathbf{y}})=\beta l_{\mathbf{y}}<\beta(k-l_{u})$,
\[
\mathbb{P}\left\{ \tau_{\beta}(k-l_{u})=\tau_{\beta,m}(k-l_{u}-l_{\mathbf{y}})\mid I_{1}...I_{\tau_{\beta}(k-C)}=\mathbf{y}\right\} =1\:.
\]
Hence, by multiplying from the left both sides of the equation $Y_{\tau_{\beta}(k-l_{u})}=h^{k-l_{u}}n_{V}$
by $g_{\mathbf{y}}^{-1}=n_{\mathbf{y}}^{-1}h^{-l_{\mathbf{y}}}$,
we get
\begin{eqnarray*}
\eta_{\mathbf{y}} & = & \mathbb{P}\left\{ X_{m+1}\cdot...\cdot X_{\tau_{\beta,m}(k-l_{u}-l_{\mathbf{y}})}=n_{\mathbf{y}}^{-1}h^{k-l_{u}-l_{\mathbf{y}}}n_{V}\mid I_{1}...I_{\tau_{\beta}(k-C)}=\mathbf{y}\right\} \\
 & = & \mathbb{P}\left\{ Y_{\tau_{\beta}(k-l_{u}-l_{\mathbf{y}})}=n_{\mathbf{y}}^{-1}h^{k-l_{u}-l_{\mathbf{y}}}n_{V}\right\} ,
\end{eqnarray*}
where in last equality we have used the stationarity of the process
$\{X_{j}\}_{j\ge1}$. Set,
\[
z_{\mathbf{y}}:=\gamma_{-\beta(k-l_{u}-l_{\mathbf{y}})}n_{\mathbf{y}}^{-1}h^{k-l_{u}-l_{\mathbf{y}}}n_{V},
\]
then
\begin{equation}
\eta_{\mathbf{y}}=\mathbb{P}\left\{ \gamma_{-\beta(k-l_{u}-l_{\mathbf{y}})}Y_{\tau_{\beta}(k-l_{u}-l_{\mathbf{y}})}=z_{\mathbf{y}}\right\} \:.\label{eq:alpha_u =00003D}
\end{equation}
From $\psi\circ\gamma=Id$, $\psi(h)=\beta$ and $n_{\mathbf{y}},n_{V}\in N$
it follows that $z_{\mathbf{y}}\in N$. Also, by (\ref{eq:est on l_y})
we have $k-l_{u}-l_{\mathbf{y}}>C-2l_{\mathrm{max}}$. Hence, by Lemma
\ref{lem:lb on prob =00003D g} and by assuming that $C$ is sufficiently
large, it follows that $\eta_{\mathbf{y}}>p_{u}^{-1}C^{-1}$. This
holds for all $\mathbf{y}\in\mathcal{Y}$, which implies that almost
surely
\[
\mathbb{P}\left\{ Y_{\tau_{\beta}(k-l_{u})}=h^{k-l_{u}}n_{V}\mid I_{1}...I_{\tau_{\beta}(k-C)}\right\} \ge p_{u}^{-1}C^{-1}\:.
\]

From the last inequality we get,
\begin{multline}
\mathbb{P}\left\{ Y_{\tau_{\beta}(k)-1}=h^{k-l_{u}}n_{V}\text{ and }I_{\tau_{\beta}(k)}=u\mid I_{1}...I_{\tau_{\beta}(k-C)}\right\} \\
=\mathbb{P}\left\{ Y_{\tau_{\beta}(k-l_{u})}=h^{k-l_{u}}n_{V}\text{ and }I_{\tau_{\beta}(k-l_{u})+1}=u\mid I_{1}...I_{\tau_{\beta}(k-C)}\right\} \\
=\mathbb{P}\left\{ Y_{\tau_{\beta}(k-l_{u})}=h^{k-l_{u}}n_{V}\mid I_{1}...I_{\tau_{\beta}(k-C)}\right\} \mathbb{P}\left\{ I_{1}=u\right\} \ge C^{-1}\:.\label{eq:>C^-1}
\end{multline}
Since $C>l_{\mathrm{max}}$ we have $\tau_{\beta}(k-C)\le\tau_{\beta}(k)-1$.
Thus, since $\alpha_{w}(x)\le1$ for all $w\in\mathcal{W}$ and $x\in\mathbb{R}^{d}$,
it follows that $Z_{\xi,\tau_{\beta}(k-C)}\ge Z_{\xi,\tau_{\beta}(k)-1}$.
Hence,
\[
Z_{\xi,\tau_{\beta}(k)}=Z_{\xi,\tau_{\beta}(k)-1}\cdot\alpha_{I_{\tau_{\beta}(k)}}(\xi.Y_{\tau_{\beta}(k)-1})\le Z_{\xi,\tau_{\beta}(k-C)}\cdot\alpha_{I_{\tau_{\beta}(k)}}(\xi.Y_{\tau_{\beta}(k)-1})\:.
\]
From this and (\ref{eq:>C^-1}) we get,
\begin{eqnarray*}
\mathbb{E}\left[Z_{\xi,\tau_{\beta}(k)}\mid I_{1},...,I_{\tau_{\beta}(k-C)}\right] & \le & Z_{\xi,\tau_{\beta}(k-C)}\mathbb{E}\left[\alpha_{I_{\tau_{\beta}(k)}}(\xi.Y_{\tau_{\beta}(k)-1})\mid I_{1},...,I_{\tau_{\beta}(k-C)}\right]\\
 & \le & Z_{\xi,\tau_{\beta}(k-C)}\left(1-C^{-1}+C^{-1}\alpha_{u}(\xi.h^{k-l_{u}}n_{V})\right)\\
 & = & Z_{\xi,\tau_{\beta}(k-C)}\left(1-C^{-1}+C^{-1}\alpha_{u}(VB^{k-l_{u}}\xi)\right)\:.
\end{eqnarray*}
This gives,
\begin{eqnarray*}
\mathbb{E}\left[Z_{\xi,\tau_{\beta}(k)}\right] & = & \mathbb{E}\left[\mathbb{E}\left[Z_{\xi,\tau_{\beta}(k)}\mid I_{1},...,I_{\tau_{\beta}(k-C)}\right]\right]\\
 & \le & \mathbb{E}\left[Z_{\xi,\tau_{\beta}(k-C)}\right]\left(1-C^{-1}+C^{-1}\alpha_{u}(VB^{k-l_{u}}\xi)\right),
\end{eqnarray*}
which completes the proof of the lemma.
\end{proof}
\begin{proof}[Proof of Proposition \ref{prop:ub on sum of dist to int} ]
Let $0<\epsilon<1$, let $C\in\mathbb{Z}_{>1}$ be large with respect
to $\epsilon$, $\mathcal{W}$, $p$ and $\Phi$, let $V\in N_{0}$,
let $\xi_{0}\in\mathbb{R}^{d}$ be with $|\widehat{\mu}(2\pi\xi_{0})|\ge\epsilon$,
and write $\xi=2\pi\xi_{0}$. For $y\in\mathbb{R}^{d}$ set $\Psi(y)=1-\alpha_{u}(y)$,
and note that $0\le\Psi(y)\le1$. By Lemma \ref{lem:ub on cond exp}
it follows that for $k\in\mathbb{Z}_{>C}$,
\[
\mathbb{E}\left[Z_{\xi,\tau_{\beta}(k)}\right]\le\mathbb{E}\left[Z_{\xi,\tau_{\beta}(k-C)}\right]\left(1-C^{-1}\Psi(VB^{k-l_{u}}\xi)\right)\:.
\]
Iterating this and using the fact that $0\le Z_{\xi,n}\le1$ for all
$n\in\mathbb{Z}_{\ge1}$, we get
\[
\mathbb{E}\left[Z_{\xi,\tau_{\beta}(k)}\right]\le\prod_{j=0}^{\left\lceil k/C\right\rceil -2}(1-C^{-1}\Psi(VB^{k-jC-l_{u}}\xi)),
\]
where $\left\lceil k/C\right\rceil $ is the smallest integer which
is at least as large as $k/C$. Let $n\in\mathbb{Z}_{\ge1}$, then
by applying the last inequality for $nC<k\le nC+C$ we get,
\begin{eqnarray*}
\prod_{k=nC+1}^{nC+C}\mathbb{E}\left[Z_{\xi,\tau_{\beta}(k)}\right] & \le & \prod_{k=nC+1}^{nC+C}\prod_{j=0}^{\left\lceil k/C\right\rceil -2}(1-C^{-1}\Psi(VB^{k-jC-l_{u}}\xi))\\
 & = & \prod_{k=1}^{C}\prod_{j=0}^{n-1}(1-C^{-1}\Psi(VB^{nC+k-jC-l_{u}}\xi))\\
 & = & \prod_{j=C+1}^{nC+C}(1-C^{-1}\Psi(VB^{j-l_{u}}\xi))\:.
\end{eqnarray*}
Hence, by Lemma \ref{lem:from doob's}
\[
\epsilon^{C}\le|\widehat{\mu}(\xi)|^{C}\le\prod_{j=C+1}^{nC+C}(1-C^{-1}\Psi(VB^{j-l_{u}}\xi))\:.
\]
From this and the inequality $1+t\le e^{t}$,
\[
\epsilon^{C}\le\exp\left(-C^{-1}\sum_{j=C+1}^{nC+C}\Psi(VB^{j-l_{u}}\xi)\right)\:.
\]
Since this holds for all $n\in\mathbb{Z}_{\ge1}$,
\begin{equation}
C^{2}\ln\epsilon^{-1}\ge\sum_{j=C+1}^{\infty}\Psi(VB^{j}\xi)\:.\label{eq:ub sum u}
\end{equation}

Set $\delta:=p_{u}/(p_{u}+p_{u'})$ and $\delta':=p_{u'}/(p_{u}+p_{u'})$.
By Taylor's theorem, given $0\le s\le1/8$ there exists $0\le t\le2\pi s$
so that
\[
\cos(2\pi s)-1=-\frac{\cos(t)}{2}(2\pi s)^{2}\le-\frac{\cos(\pi/4)}{2}(2\pi s)^{2}\le-s^{2}\:.
\]
Hence,
\begin{eqnarray*}
|\delta e^{2\pi is}+\delta'|^{2} & = & (\delta\cos(2\pi s)+\delta')^{2}+\delta^{2}\sin^{2}(2\pi s)\\
 & = & 1+2\delta\delta'(\cos(2\pi s)-1)\le1-2\delta\delta's^{2},
\end{eqnarray*}
and so,
\[
1-|\delta e^{2\pi is}+\delta'|\ge1-(1-2\delta\delta's^{2})^{1/2}\ge\delta\delta's^{2}\:.
\]
It follows that if $y\in\mathbb{R}^{d}$ satisfies $\left\Vert \left\langle y,\varphi_{u}(0)-\varphi_{u'}(0)\right\rangle \right\Vert \le1/8$,
then
\begin{eqnarray}
\Psi(2\pi y) & = & 1-\left|\delta\exp\left(2\pi i\left\Vert \left\langle y,\varphi_{u}(0)-\varphi_{u'}(0)\right\rangle \right\Vert \right)+\delta'\right|\nonumber \\
 & \ge & \delta\delta'\left\Vert \left\langle y,\varphi_{u}(0)-\varphi_{u'}(0)\right\rangle \right\Vert ^{2}\:.\label{eq:if <=00003D1/8}
\end{eqnarray}
Additionally, for $1/8<s\le1/2$ we have
\begin{eqnarray*}
|\delta e^{2\pi is}+\delta'|^{2} & = & 1+2\delta\delta'(\cos(2\pi s)-1)\\
 & \le & 1+2\delta\delta'(\cos(\pi/4)-1)\le1-\delta\delta'/2,
\end{eqnarray*}
and so,
\[
1-|\delta e^{2\pi is}+\delta'|\ge1-(1-\delta\delta'/2)^{1/2}\ge\delta\delta'/4\:.
\]
It follows that if $y\in\mathbb{R}^{d}$ satisfies $\left\Vert \left\langle y,\varphi_{u}(0)-\varphi_{u'}(0)\right\rangle \right\Vert >1/8$,
then
\begin{eqnarray}
\Psi(2\pi y) & = & 1-\left|\delta\exp\left(2\pi i\left\Vert \left\langle y,\varphi_{u}(0)-\varphi_{u'}(0)\right\rangle \right\Vert \right)+\delta'\right|\nonumber \\
 & \ge & \delta\delta'/4>\frac{1}{4}\delta\delta'\left\Vert \left\langle y,\varphi_{u}(0)-\varphi_{u'}(0)\right\rangle \right\Vert ^{2}\:.\label{eq:if >1/8}
\end{eqnarray}
Now recall that $\xi=2\pi\xi_{0}$, then from (\ref{eq:ub sum u}),
(\ref{eq:if <=00003D1/8}) and (\ref{eq:if >1/8})
\[
C^{2}\ln\epsilon^{-1}\ge\sum_{j=C+1}^{\infty}\Psi(2\pi VB^{j}\xi_{0})\ge\frac{1}{4}\delta\delta'\sum_{j=C+1}^{\infty}\left\Vert \left\langle VB^{j}\xi_{0},\varphi_{u}(0)-\varphi_{u'}(0)\right\rangle \right\Vert ^{2},
\]
which completes the proof of the proposition.
\end{proof}

\subsection{\label{subsec:Proof-of-Proposition disc case}Proof of Proposition
\ref{prop:main disc case}}

We continue to assume that $G$ is discrete. In order to apply Proposition
\ref{prop:ub on sum of dist to int} we need the following lemma.
\begin{lem}
\label{lem:good fam of words}Suppose that $a_{1}=0$. Then there
exists $\mathcal{W}\subset\Lambda^{*}$, $L\in\mathbb{Z}_{\ge1}$
and $\{u_{j}\}_{i=1}^{\ell},\{u_{j}'\}_{i=1}^{\ell}\subset\mathcal{W}$
so that,
\begin{enumerate}
\item $\mathcal{W}$ is a minimal cut-set for $\Lambda^{*}$;
\item $G$ is generated by $\{g_{w}\}_{w\in\mathcal{W}}$;
\item $g_{u_{j}}=g_{u_{j}'}$ for $1\le j\le\ell$;
\item $\varphi_{u_{j}}(0)-\varphi_{u_{j}'}(0)=a_{j}-A^{L}a_{j}$ for $1\le j\le\ell$;
\item $VA^{L}=A^{L}V$ for $V\in N_{0}$.
\end{enumerate}
\end{lem}

\begin{proof}
For every $k\ge1$ we have $h^{kl_{1}}g_{1}^{-k}\in N$. Since $N$
is finite there exist $k_{1}>k_{2}\ge1$ with $h^{k_{1}l_{1}}g_{1}^{-k_{1}}=h^{k_{2}l_{1}}g_{1}^{-k_{2}}$,
and so $g_{1}^{k_{1}-k_{2}}=h^{(k_{1}-k_{2})l_{1}}$. For every $g,g'\in G$
it holds that $[g,g']\in N$, where $[g,g']$ is the commutator of
$g$ and $g'$. Since $N$ is finite there exist $m_{1}>m_{2}\ge1$
so that $[g^{m_{1}},g']=[g^{m_{2}},g']$, and so $g^{m_{1}-m_{2}}g'=g'g^{m_{1}-m_{2}}$.
It follows that there exists $b\in\mathbb{Z}_{>1}$ so that $g_{1}^{b}=h^{bl_{1}}$,
$g_{1}^{b}g_{j}=g_{j}g_{1}^{b}$ for $1\le j\le\ell$, $g_{\ell}^{b}g_{1}=g_{1}g_{\ell}^{b}$,
and $h^{b}n=nh^{b}$ for $n\in N$. We set $L:=bl_{1}$.

Recall that $h=(\beta,U)$. For $V\in N_{0}$ we have $(0,V)\in N$,
thus
\[
(b\beta,VU^{b})=(0,V)h^{b}=h^{b}(0,V)=(b\beta,U^{b}V),
\]
and so $VU^{b}=U^{b}V$. Since $A=2^{-\beta}U$ this implies that
$VA^{L}=A^{L}V$, and so the fifth condition in the statement of the
lemma is satisfied.

For $m\ge1$ denote the set of $m$-words over $\Lambda$ by $\Lambda^{m}$.
For $1\le j\le\ell$ we write $j^{m}$ for the word $i_{1}...i_{m}\in\Lambda^{m}$
with $i_{k}=j$ for $1\le k\le m$. Given $m_{1},m_{2}\ge1$, $w_{1}\in\Lambda^{m_{1}}$
and $w_{2}\in\Lambda^{m_{2}}$, we write $w_{1}w_{2}\in\Lambda^{m_{1}+m_{2}}$
for the concatenation of $w_{1}$ with $w_{2}$.

Set
\[
\mathcal{W}:=(\Lambda^{b+1}\setminus\{\ell^{b}1\})\cup\{\ell^{b}1i\::\:1\le i\le\ell\}\:.
\]
It is clear that $\mathcal{W}$ is a minimal cut-set for $\Lambda^{*}$.
For $1\le j\le\ell$ set $u_{j}:=j1^{b}$ and $u_{j}':=1^{b}j$. Note
that since $\Phi$ is affinely irreducible we must have $\ell>1$.
From this and $b>1$, it follows that $u_{j},u_{j}'\in\mathcal{W}$.

From $g_{1}^{b}g_{j}=g_{j}g_{1}^{b}$ it follows that the third condition
is satisfied. From $g_{1}=(\log r_{1}^{-1},U_{1})$, $h=(\beta,U)$
and $g_{1}^{b}=h^{bl_{1}}$ it follows that, 
\[
r_{1}^{b}U_{1}^{b}=2^{-\beta bl_{1}}U^{bl_{1}}=A^{L}\:.
\]
Thus, since $a_{1}=0$
\[
\varphi_{u_{j}}(0)-\varphi_{u_{j}'}(0)=a_{j}-r_{1}^{b}U_{1}^{b}a_{j}=a_{j}-A^{L}a_{j},
\]
which shows that the fourth condition is satisfied.

It remains to show that $G$ is generated by $\{g_{w}\}_{w\in\mathcal{W}}$.
By definition $G$ is the closed subgroup of $\mathbb{R}\times O(d)$
generated by $\{g_{i}\}_{i=1}^{\ell}$. From this and since $G$ is
discrete, it follows that $G$ is generated by $\{g_{i}\}_{i=1}^{\ell}$.
Write $G_{1}$ for the group generated by $\{g_{w}\}_{w\in\mathcal{W}}$.
For every $1\le i\le\ell$ we have $1\ell^{b},\ell^{b}1i\in\mathcal{W}$.
Hence from $g_{\ell}^{b}g_{1}=g_{1}g_{\ell}^{b}$,
\[
g_{i}=(g_{1\ell^{b}})^{-1}g_{\ell^{b}1i}\in G_{1}\:.
\]
 This shows that $G_{1}=G$, which completes the proof of the lemma.
\end{proof}
The treatment of the $1$-dimensional case, carried out in \cite{Br}
and \cite{VY}, relies on a classical theorem of Pisot (see \cite[Theorem 2.1]{Bu}).
In the proof of Proposition \ref{prop:main disc case} we shall need
the following extension of this result. It follows directly from \cite[Chapter III, Theorem III]{Pi}
together with \cite[Theorem 1]{Ko}. A result similar to \cite[Theorem 1]{Ko}
was obtained in \cite[Lemma 2]{Ma}.
\begin{thm}
\label{thm:gen of pisot}Let $k\ge1$ and $\theta_{1},...,\theta_{k},\lambda_{1},...,\lambda_{k}\in\mathbb{C}$
be with $|\theta_{j}|>1$ and $\lambda_{j}\ne0$ for $1\le j\le k$,
and $\theta_{j}\ne\theta_{i}$ for $1\le j<i\le k$. For $n\ge0$
set $\eta_{n}=\sum_{j=1}^{k}\lambda_{j}\theta_{j}^{n}$, and suppose
that $\eta_{n}\in\mathbb{R}$ for all $n\ge0$. Moreover assume that
$\sum_{n\ge0}\Vert\eta_{n}\Vert^{2}<\infty$. Then,
\begin{enumerate}
\item $\{\theta_{1},...,\theta_{k}\}$ is a P.V. $k$-tuple;
\item $\lambda_{j}\in\mathbb{Q}(\theta_{j})$ for each $1\le j\le k$;
\item if $1\le j,i\le k$ are such that $\theta_{j}$ and $\theta_{i}$
are conjugates over $\mathbb{Q}$ and $\sigma:\mathbb{Q}(\theta_{j})\rightarrow\mathbb{Q}(\theta_{i})$
is an isomorphism with $\sigma(\theta_{j})=\theta_{i}$, then $\sigma(\lambda_{j})=\lambda_{i}$.
\end{enumerate}
\end{thm}

\begin{proof}[Proof of Proposition \ref{prop:main disc case}]
Recall that $A=2^{-\beta}U$ and $B=A^{*}$, where $U\in O(d)$.
Let $\theta_{1},...,\theta_{s}\in\mathbb{C}$ be the distinct eigenvalues
of $A^{-1}$. For $1\le j\le s$ let $\mathbb{V}_{j}\subset\mathbb{C}^{d}$
be the eigenspace of $A^{-1}$ corresponding to $\theta_{j}$. Since
$B^{-1}=(A^{-1})^{*}$, the numbers $\theta_{1},...,\theta_{s}$ are
also the distinct eigenvalues of $B^{-1}$, and $\mathbb{V}_{j}$
is the eigenspace of $B^{-1}$ corresponding to $\overline{\theta_{j}}$
for each $1\le j\le s$.

Assume that there exists a probability vector $p=(p_{i})_{i=1}^{\ell}>0$
so that the self-similar measure $\mu$ corresponding to $\Phi$ and
$p$ is non-Rajchman. There exist $\epsilon>0$ and $\xi_{1},\xi_{2},...\in\mathbb{R}^{d}$
so that $|\xi_{k}|\ge1$ and $|\widehat{\mu}(2\pi\xi_{k})|\ge\epsilon$
for all $k\ge1$, and also $|\xi_{k}|\overset{k}{\rightarrow}\infty$.
For $k\ge1$ set
\[
n_{k}:=\min\{n\ge1\::\:|B^{n}\xi_{k}|\le1\},
\]
then $2^{-\beta}\le|B^{n_{k}}\xi_{k}|\le1$. Thus, by moving to a
subsequence without changing the notation, we may assume that there
exists $0\ne\xi\in\mathbb{R}^{d}$ so that $B^{n_{k}}\xi_{k}\overset{k}{\rightarrow}\xi$.

Recall that we assume $a_{1}=0$, and let $\mathcal{W}$, $L$, $\{u_{i}\}_{i=1}^{\ell}$
and $\{u_{i}'\}_{i=1}^{\ell}$ be as obtained in Lemma \ref{lem:good fam of words}.
Let $C>1$ be large with respect to $\epsilon$, $\mathcal{W}$, $\Phi$
and $p$. For $1\le i\le\ell$ and $V\in N_{0}$ we have $g_{u_{i}}=g_{u_{i}'}$
and,
\[
V(\varphi_{u_{i}}(0)-\varphi_{u_{i}'}(0))=V(a_{i}-A^{L}a_{i})=(I-A^{L})Va_{i},
\]
where $I$ is the identity operator here. Set $b_{i,V}:=(I-A^{L})Va_{i}$,
then by Proposition \ref{prop:ub on sum of dist to int} it follows
that for all $k\ge1$,
\[
C\ge\sum_{n\ge0}\Vert\left\langle V(\varphi_{u_{i}}(0)-\varphi_{u_{i}'}(0)),B^{n}\xi_{k}\right\rangle \Vert^{2}=\sum_{n\ge-n_{k}}\Vert\left\langle b_{i,V},B^{n}B^{n_{k}}\xi_{k}\right\rangle \Vert^{2}\:.
\]
From $|\xi_{k}|\overset{k}{\rightarrow}\infty$ it follows that $n_{k}\overset{k}{\rightarrow}\infty$.
Thus, for every fixed $T\ge1$ and $k\ge1$ large enough with respect
to $T$,
\[
\sum_{n=0}^{T}\Vert\left\langle b_{i,V},B^{-n}B^{n_{k}}\xi_{k}\right\rangle \Vert^{2}\le C\:.
\]
From this and since $B^{n_{k}}\xi_{k}\overset{k}{\rightarrow}\xi$,
\[
\sum_{n=0}^{T}\Vert\left\langle b_{i,V},B^{-n}\xi\right\rangle \Vert^{2}\le C\:.
\]
Hence, since this holds for every $T\ge1$,
\begin{equation}
\sum_{n=0}^{\infty}\Vert\left\langle b_{i,V},B^{-n}\xi\right\rangle \Vert^{2}<\infty\text{ for all }1\le i\le\ell\text{ and }V\in N_{0}\:.\label{eq:< infinity}
\end{equation}

Recall that for a linear subspace $\mathbb{V}$ of $\mathbb{C}^{d}$
we denote by $\pi_{\mathbb{V}}$ the orthogonal projection onto $\mathbb{V}$.
For $1\le j\le s$ set
\[
\zeta_{j}:=(1-\overline{\theta_{j}^{-L}})\pi_{\mathbb{V}_{j}}\xi,
\]
where we consider $\xi$ as a vector in $\mathbb{C}^{d}$ here. Let
$\lambda_{j}:\mathbb{R}^{d}\rightarrow\mathbb{C}$ be with $\lambda_{j}(x)=\left\langle x,\zeta_{j}\right\rangle $
for $x\in\mathbb{R}^{d}$. Regarding $\mathbb{C}$ as a $2$-dimensional
vector space over $\mathbb{R}$, the maps $\lambda_{1},...,\lambda_{s}$
are $\mathbb{R}$-linear. Additionally, for $1\le i\le\ell$, $V\in N_{0}$
and $n\ge0$
\begin{eqnarray}
\left\langle b_{i,V},B^{-n}\xi\right\rangle  & = & \left\langle (I-A^{L})Va_{i},\sum_{j=1}^{s}\overline{\theta_{j}^{n}}\pi_{\mathbb{V}_{j}}\xi\right\rangle \nonumber \\
 & = & \sum_{j=1}^{s}\theta_{j}^{n}\left\langle Va_{i},(I-B^{L})\pi_{\mathbb{V}_{j}}\xi\right\rangle =\sum_{j=1}^{s}\theta_{j}^{n}\lambda_{j}(Va_{i}),\label{eq:fin exp dir prod}
\end{eqnarray}
which in particular implies that $\sum_{j=1}^{s}\theta_{j}^{n}\lambda_{j}(Va_{i})\in\mathbb{R}$.
From (\ref{eq:< infinity}) and (\ref{eq:fin exp dir prod}),
\begin{equation}
\sum_{n=0}^{\infty}\:\Bigl\Vert\sum_{j=1}^{s}\theta_{j}^{n}\lambda_{j}(Va_{i})\Bigr\Vert^{2}<\infty\text{ for all }1\le i\le\ell\text{ and }V\in N_{0}\:.\label{eq:comb of facts}
\end{equation}

For every $1\le j\le s$ we have $|\theta_{j}|=2^{\beta}>1$. From
this and since $\xi\ne0$, we get that there exists $1\le j_{0}\le s$
so that $\zeta_{j_{0}}\ne0$. Hence $\lambda_{j_{0}}$ is not identically
$0$, and so $\ker\lambda_{j_{0}}$ is a proper subspace of $\mathbb{R}^{d}$.
Let us show that,
\begin{equation}
\lambda_{j_{0}}(Va_{i})\ne0\text{ for some }1\le i\le\ell\text{ and }V\in N_{0}\:.\label{eq:func not all 0}
\end{equation}
Assume by contradiction that this is false. Then,
\[
\{a_{i}\}_{i=1}^{\ell}\subset\cap_{V\in N_{0}}V(\ker\lambda_{j_{0}})=:\mathbb{W}\:.
\]
For $x\in\mathbb{R}^{d}$,
\[
\lambda_{j_{0}}(A^{-1}x)=\left\langle A^{-1}x,\zeta_{j_{0}}\right\rangle =\left\langle x,B^{-1}\zeta_{j_{0}}\right\rangle =\theta_{j_{0}}\lambda_{j_{0}}(x),
\]
from which it follows that $A(\ker\lambda_{j_{0}})=\ker\lambda_{j_{0}}$.
Moreover, from $N\triangleleft G$ it follows that $AN_{0}=N_{0}A$,
which implies
\[
A(\mathbb{W})=\cap_{V\in N_{0}}VA(\ker\lambda_{j_{0}})=\mathbb{W}\:.
\]
Since $N_{0}$ is a group, we also have $V(\mathbb{W})=\mathbb{W}$
for all $V\in N_{0}$. By (\ref{eq:rep of each g}), for every $1\le i\le\ell$
there exists $V_{i}\in N_{0}$ so that $\varphi_{i}(x)=V_{i}A^{l_{i}}x+a_{i}$.
From all of this it follows that $\varphi_{i}(\mathbb{W})=\mathbb{W}$
for all $1\le i\le\ell$. Since $\mathbb{W}\subset\ker\lambda_{j_{0}}$
and since $\ker\lambda_{j_{0}}$ is a proper subspace of $\mathbb{R}^{d}$,
this contradicts the affine irreducibility of $\Phi$, which shows
that (\ref{eq:func not all 0}) must hold. For $1\le i\le\ell$ and
$V\in N_{0}$ set,
\[
J_{i,V}:=\{1\le j\le s\::\:\lambda_{j}(Va_{i})\ne0\}\:.
\]
From (\ref{eq:func not all 0}) it follows that $J_{i,V}\ne\emptyset$
for some $1\le i\le\ell$ and $V\in N_{0}$.

For $1\le i\le\ell$ and $V\in N_{0}$ it follows from (\ref{eq:comb of facts})
and Theorem \ref{thm:gen of pisot} that,
\begin{enumerate}
\item $\{\theta_{j}\}_{j\in J_{i,V}}$ is a P.V. $|J_{i,V}|$-tuple or $J_{i,V}=\emptyset$;
\item $\lambda_{j}(Va_{i})\in\mathbb{Q}(\theta_{j})$ for $1\le j\le s$;
\item $\sigma(\lambda_{j_{1}}(Va_{i}))=\lambda_{j_{2}}(Va_{i})$ for every
$j_{1},j_{2}\in J_{i,V}$ and isomorphism $\sigma:\mathbb{Q}(\theta_{j_{1}})\rightarrow\mathbb{Q}(\theta_{j_{2}})$
with $\sigma(\theta_{j_{1}})=\theta_{j_{2}}$ (if such a $\sigma$
exists).
\end{enumerate}
Let $1\le i_{0}\le\ell$ and $V_{0}\in N_{0}$ be with $J_{i_{0},V_{0}}\ne\emptyset$,
so that $\{\theta_{j}\}_{j\in J_{i_{0},V_{0}}}$ is a P.V. $|J_{i_{0},V_{0}}|$-tuple.
By the definition of a P.V. tuple, there exists a nonempty subset
$J$ of $J_{i_{0},V_{0}}$ so that $\{\theta_{j}\}_{j\in J}$ is a
P.V. $|J|$-tuple and $\theta_{j_{1}},\theta_{j_{2}}$ are conjugates
over $\mathbb{Q}$ for all $j_{1},j_{2}\in J$. For $j\in J$ we have,
\[
\left\langle V_{0}a_{i_{0}},\zeta_{j}\right\rangle =\lambda_{j}(V_{0}a_{i_{0}})\ne0,
\]
 and so $\zeta_{j}\ne0$. Recall that $\zeta_{j}:=(1-\overline{\theta_{j}^{-L}})\pi_{\mathbb{V}_{j}}\xi$,
which implies $A^{-1}\zeta_{j}=\theta_{j}\zeta_{j}$ for $j\in J$.
It remains to construct to polynomials $P_{i,V}$.

Let $1\le i\le\ell$ and $V\in N_{0}$ be given. Since $\{\theta_{j}\}_{j\in J}$
are algebraic conjugates, since $|\theta_{j}|>1$ for $j\in J$, and
since $\{\theta_{j}\}_{j\in J_{i,V}}$ is either empty or a P.V. tuple,
it follows that $J\cap J_{i,V}=\emptyset$ or $J\subset J_{i,V}$.
If $J\cap J_{i,V}=\emptyset$ we set $P_{i,V}(X):=0$. For $j\in J$
we have $j\notin J_{i,V}$, and so
\[
\left\langle Va_{i},\zeta_{j}\right\rangle =\lambda_{j}(Va_{i})=0=P_{i,V}(\theta_{j})\:.
\]
Next suppose that $J\subset J_{i,V}$, and let $j_{1}\in J$. Since
$\theta_{j_{1}}$ is algebraic and from $\lambda_{j_{1}}(Va_{i})\in\mathbb{Q}(\theta_{j_{1}})$,
it follows that there exists $P_{i,V}(X)\in\mathbb{Q}[X]$ so that
$\lambda_{j_{1}}(Va_{i})=P_{i,V}(\theta_{j_{1}})$. Let $j\in J$,
then $\theta_{j_{1}}$ and $\theta_{j}$ are conjugates over $\mathbb{Q}$,
and so there exists an isomorphism $\sigma:\mathbb{Q}(\theta_{j_{1}})\rightarrow\mathbb{Q}(\theta_{j})$
with $\sigma(\theta_{j_{1}})=\theta_{j}$. From this and $j_{1},j\in J\subset J_{i,V}$
we get,
\[
\left\langle Va_{i},\zeta_{j}\right\rangle =\lambda_{j}(Va_{i})=\sigma(\lambda_{j_{1}}(Va_{i}))=\sigma(P_{i,V}(\theta_{j_{1}}))=P_{i,V}(\sigma(\theta_{j_{1}}))=P_{i,V}(\theta_{j}),
\]
and so $P_{i,V}$ satisfies the required property. This completes
the proof of the proposition.
\end{proof}

\subsection{\label{subsec:Construction-of-non-Rajchman}Construction of non-Rajchman
self-similar measures}

The purpose of this subsection is to prove the following converse
to Proposition \ref{prop:main disc case}.
\begin{prop}
\label{prop:conv disc case}Suppose that $G$ is discrete, and let
$A$ and $N_{0}$ be as defined before the statement of Proposition
\ref{prop:main disc case}. Assume that there exist $k\ge1$, $\theta_{1},...,\theta_{k}\in\mathbb{C}$
and $\zeta_{1},...,\zeta_{k}\in\mathbb{C}^{d}\setminus\{0\}$, so
that
\begin{enumerate}
\item $\{\theta_{1},...,\theta_{k}\}$ is a P.V. $k$-tuple;
\item $A^{-1}\zeta_{j}=\theta_{j}\zeta_{j}$ for $1\le j\le k$;
\item for every $1\le i\le\ell$ and $V\in N_{0}$ there exists $P_{i,V}\in\mathbb{Q}[X]$
so that $\left\langle Va_{i},\zeta_{j}\right\rangle =P_{i,V}(\theta_{j})$
for all $1\le j\le k$;
\end{enumerate}
Then there exists a positive probability vector $p=(p_{i})_{i=1}^{\ell}$
so that the self-similar measure corresponding to $\Phi$ and $p$
is non-Rajchman.
\end{prop}

The proof of the proposition relies on the following lemma. A version
of it can be found in \cite[Theorem 3.5]{Ca}, but we provide the
short proof for the reader's convenience.
\begin{lem}
\label{lem:exp decay}Let $\{\theta_{1},...,\theta_{k}\}$ be a P.V.
$k$-tuple and let $P\in\mathbb{Z}[X]$. Then there exist $C>1$ and
$0<\delta<1$ such that,
\[
\Vert P(\theta_{1})\theta_{1}^{n}+...+P(\theta_{k})\theta_{k}^{n}\Vert\le C\delta^{n}\text{ for all }n\ge0\;.
\]
\end{lem}

\begin{proof}
Let $Q\in\mathbb{Z}[X]$ be the monic polynomial of smallest degree
with $Q(\theta_{j})=0$ for $1\le j\le k$. Let $\theta_{k+1},...,\theta_{s}$
be the remaining roots of $Q$. Set
\[
\delta:=\underset{k<j\le s}{\max}\:|\theta_{j}|\text{ and }C:=\sum_{j=k+1}^{s}|P(\theta_{j})|,
\]
then $0<\delta<1$ since $\{\theta_{1},...,\theta_{k}\}$ is a P.V.
$k$-tuple. Since $\theta_{1},...,\theta_{s}$ are all the roots of
$Q$, and by the fundamental theorem of symmetric polynomials, it
follows that for all $n\ge0$
\[
P(\theta_{1})\theta_{1}^{n}+...+P(\theta_{s})\theta_{s}^{n}\in\mathbb{Z}\:.
\]
Hence,
\[
\Vert\sum_{j=1}^{k}P(\theta_{j})\theta_{j}^{n}\Vert\le\sum_{j=k+1}^{s}|P(\theta_{j})\theta_{j}^{n}|\le C\delta^{n},
\]
which completes the proof of the lemma.
\end{proof}
The following lemma is a consequence of the affine irreducibility
of $\Phi$. For $(z_{1},...,z_{d})=z\in\mathbb{C}^{d}$ we write $\overline{z}$
in place of $(\overline{z_{1}},...,\overline{z_{d}})$.
\begin{lem}
\label{lem:zeta equal}Assume the conditions of Proposition \ref{prop:conv disc case}
are satisfied. Let $1\le j_{1},j_{2}\le k$ be with $\theta_{j_{2}}=\overline{\theta_{j_{1}}}$,
then $\zeta_{j_{2}}=\overline{\zeta_{j_{1}}}$.
\end{lem}

\begin{proof}
The proof is similar to the argument used in the proof of Proposition
\ref{prop:main disc case} to establish (\ref{eq:func not all 0}).
Set,
\[
\mathbb{V}:=\left\{ x\in\mathbb{R}^{d}\::\:\left\langle x,\zeta_{j_{2}}-\overline{\zeta_{j_{1}}}\right\rangle =0\right\} \text{ and }\mathbb{W}:=\cap_{V\in N_{0}}V(\mathbb{V})\:.
\]
For $1\le i\le\ell$ and $V\in N_{0}$,
\[
\left\langle Va_{i},\zeta_{j_{2}}-\overline{\zeta_{j_{1}}}\right\rangle =P_{i,V}(\theta_{j_{2}})-\overline{P_{i,V}(\theta_{j_{1}})}=0,
\]
and so $a_{i}\in\mathbb{W}$. For $x\in\mathbb{V}$,
\[
\left\langle A^{-1}x,\zeta_{j_{2}}-\overline{\zeta_{j_{1}}}\right\rangle =\left\langle x,B^{-1}\zeta_{j_{2}}-\overline{B^{-1}\zeta_{j_{1}}}\right\rangle =\theta_{j_{2}}\left\langle x,\zeta_{j_{2}}-\overline{\zeta_{j_{1}}}\right\rangle =0,
\]
and so $A(\mathbb{V})=\mathbb{V}$. Moreover, since $AN_{0}=N_{0}A$,
\[
A(\mathbb{W})=\cap_{V\in N_{0}}VA(\mathbb{V})=\mathbb{W}\:.
\]
Since $N_{0}$ is a group, we also have $V(\mathbb{W})=\mathbb{W}$
for all $V\in N_{0}$. By (\ref{eq:rep of each g}), for every $1\le i\le\ell$
there exists $V_{i}\in N_{0}$ so that $\varphi_{i}(x)=V_{i}A^{l_{i}}x+a_{i}$.
From all of this it follows that $\varphi_{i}(\mathbb{W})=\mathbb{W}$
for all $1\le i\le\ell$. Since $\Phi$ is affinely irreducible and
$\mathbb{W}\subset\mathbb{V}$, we must have $\mathbb{V}=\mathbb{R}^{d}$.
This implies that $\zeta_{j_{2}}=\overline{\zeta_{j_{1}}}$, which
completes the proof of the lemma.
\end{proof}
The following lemma will enable us to assume that $a_{1}=0$, which
will be useful in the proof of Proposition \ref{prop:conv disc case}.
\begin{lem}
\label{lem:red to a1=00003D0}Assume the conditions of Proposition
\ref{prop:conv disc case} are satisfied. Suppose also that $\theta_{1},...,\theta_{k}$
are all conjugates over $\mathbb{Q}$. For $x\in\mathbb{R}^{d}$ set
$Tx=x-(I-r_{1}U_{1})^{-1}a_{1}$, where $I$ is the identity operator.
Then $T\circ\varphi_{1}\circ T^{-1}(0)=0$, and for every $1\le i\le\ell$
and $V\in N_{0}$ there exists $Q_{i,V}\in\mathbb{Q}[X]$ so that
\[
\left\langle VT\circ\varphi_{i}\circ T^{-1}(0),\zeta_{j}\right\rangle =Q_{i,V}(\theta_{j})\text{ for all }1\le j\le k\:.
\]
\end{lem}

\begin{proof}
For $1\le i\le\ell$ we have,
\begin{equation}
T\circ\varphi_{i}\circ T^{-1}(0)=a_{i}-(I-r_{i}U_{i})(I-r_{1}U_{1})^{-1}a_{1},\label{eq:conj by tran}
\end{equation}
which shows that $T\circ\varphi_{1}\circ T^{-1}(0)=0$.

Let $V\in N_{0}$. By (\ref{eq:rep of riUi}), since $BN_{0}B^{-1}=N_{0}$
and since $N_{0}$ is finite, there exists $m\ge1$ so that $r_{1}^{m}U_{1}^{-m}=B^{ml_{1}}$
and $B^{m}V=VB^{m}$. Additionally, for every $b\in\mathbb{Z}_{\ge0}$
there exists $V_{b}\in N_{0}$ so that $r_{1}^{b}U_{1}^{-b}V=V_{b}B^{bl_{1}}$.
Set $S:=\sum_{b=0}^{m-1}r_{1}^{b}U_{1}^{-b}$, then for $1\le j\le k$,
\begin{equation}
\left\langle a_{1},SV\zeta_{j}\right\rangle =\sum_{b=0}^{m-1}\left\langle a_{1},V_{b}B^{bl_{1}}\zeta_{j}\right\rangle =\sum_{b=0}^{m-1}\theta_{j}^{-bl_{1}}P_{1,V_{b}^{-1}}(\theta_{j})\:.\label{eq:first dev}
\end{equation}
On the other hand, since
\[
(I-r_{1}U_{1}^{-1})SV\zeta_{j}=(I-r_{1}^{m}U_{1}^{-m})V\zeta_{j}=V(I-B^{ml_{1}})\zeta_{j}=(1-\overline{\theta_{j}^{-ml_{1}}})V\zeta_{j},
\]
we have
\begin{eqnarray*}
\left\langle a_{1},SV\zeta_{j}\right\rangle  & = & \left\langle (I-r_{1}U_{1})^{-1}a_{1},(I-r_{1}U_{1}^{-1})SV\zeta_{j}\right\rangle \\
 & = & (1-\theta_{j}^{-ml_{1}})\left\langle V^{-1}(I-r_{1}U_{1})^{-1}a_{1},\zeta_{j}\right\rangle \:.
\end{eqnarray*}
From this and (\ref{eq:first dev}) we get,
\[
\left\langle V^{-1}(I-r_{1}U_{1})^{-1}a_{1},\zeta_{j}\right\rangle =(1-\theta_{j}^{-ml_{1}})^{-1}\sum_{b=0}^{m-1}\theta_{j}^{-bl_{1}}P_{1,V_{b}^{-1}}(\theta_{j})\in\mathbb{Q}(\theta_{j})\:.
\]
Since $\theta_{1},...,\theta_{k}$ are algebraic conjugates, it follows
that for every $V\in N_{0}$ there exists $Q_{V}\in\mathbb{Q}[X]$
so that
\begin{equation}
\left\langle V(I-r_{1}U_{1})^{-1}a_{1},\zeta_{j}\right\rangle =Q_{V}(\theta_{j})\text{ for }1\le j\le k\:.\label{eq:exi Q}
\end{equation}

Fix $1\le i\le\ell$ and $V\in N_{0}$. There exists $V'\in N_{0}$
so that $r_{i}VU_{i}=A^{l_{i}}V'$. Hence for $1\le j\le k$,
\[
\left\langle r_{i}VU_{i}(I-r_{1}U_{1})^{-1}a_{1},\zeta_{j}\right\rangle =\left\langle V'(I-r_{1}U_{1})^{-1}a_{1},B^{l_{i}}\zeta_{j}\right\rangle =\theta_{j}^{-l_{i}}Q_{V'}(\theta_{j})\:.
\]
It follows that there exists $R_{i,V}\in\mathbb{Q}[X]$ so that,
\[
\left\langle r_{i}VU_{i}(I-r_{1}U_{1})^{-1}a_{1},\zeta_{j}\right\rangle =R_{i,V}(\theta_{j})\text{ for }1\le j\le k\:.
\]
From this, (\ref{eq:exi Q}) and (\ref{eq:conj by tran}), we get
that for $1\le j\le k$ 
\[
\left\langle VT\circ\varphi_{i}\circ T^{-1}(0),\zeta_{j}\right\rangle =P_{i,V}(\theta_{j})-Q_{V}(\theta_{j})+R_{i,V}(\theta_{j}),
\]
which completes the proof of the lemma.
\end{proof}
\begin{proof}[Proof of Proposition \ref{prop:conv disc case}]
There exists $\emptyset\ne J\subset\{1,...,k\}$ so that $\{\theta_{j}\}_{j\in J}$
is a P.V. $|J|$-tuple, and such that $\theta_{j_{1}}$ and $\theta_{j_{2}}$
are conjugates over $\mathbb{Q}$ for all $j_{1},j_{2}\in J$. Thus,
be replacing $\{\theta_{j}\}_{j=1}^{k}$ with $\{\theta_{j}\}_{j\in J}$
and $\{\zeta_{j}\}_{j=1}^{k}$ with $\{\zeta_{j}\}_{j\in J}$, without
changing the notation, we may assume that $\theta_{1},...,\theta_{k}$
are all conjugates over $\mathbb{Q}$.

Let $T:\mathbb{R}^{d}\rightarrow\mathbb{R}^{d}$ be as in Lemma \ref{lem:red to a1=00003D0}.
By that lemma $T\circ\varphi_{1}\circ T^{-1}(0)=0$, and there exists
$M\in\mathbb{Z}_{\ge1}$ so that for every $1\le i\le\ell$ and $V\in N_{0}$
there exists $Q_{i,V}\in\mathbb{Z}[X]$ such that,
\[
\left\langle VT\circ\varphi_{i}\circ T^{-1}(0),M\zeta_{j}\right\rangle =Q_{i,V}(\theta_{j})\text{ for all }1\le j\le k\:.
\]
Set $\Phi'=\{T\circ\varphi_{i}\circ T^{-1}\}_{i=1}^{\ell}$, and note
that $\Phi'$ is affinely irreducible (since $\Phi$ is), and that
the linear parts of the maps in $\Phi'$ are equal to the linear parts
of the maps is $\Phi$. Additionally, observe that if $p=(p_{i})_{i=1}^{\ell}$
is a probability vector and $\mu$ is the self-similar measure corresponding
to $\Phi$ and $p$, then $T\mu$ is the self-similar measure corresponding
to $\Phi'$ and $p$. Moreover, it is clear that $\mu$ is Rajchman
if and only if $T\mu$ is Rajchman. From all of this it follows that
by replacing $\Phi$ with $\Phi'$, $\{\zeta_{j}\}_{j=1}^{k}$ with
$\{M\zeta_{j}\}_{j=1}^{k}$ and $\{P_{i,V}\}$ with $\{Q_{i,V}\}$,
without changing the notation, we may assume that $a_{1}=0$ and $P_{i,V}\in\mathbb{Z}[X]$
for all $1\le i\le\ell$ and $V\in N_{0}$.

By Lemma \ref{lem:exp decay}, since $\{P_{i,V}\}\subset\mathbb{Z}[X]$
and since $N_{0}$ is finite, there exists $C>1$ and $0<\delta<1$
so that for all $1\le i\le\ell$, $V\in N_{0}$ and $b\in\mathbb{Z}_{\ge0}$,
\begin{equation}
\Vert\sum_{j=1}^{k}\theta_{j}^{b}\left\langle Va_{i},\zeta_{j}\right\rangle \Vert=\Vert\sum_{j=1}^{k}\theta_{j}^{b}P_{i,V}(\theta_{j})\Vert\le C\delta^{b}\:.\label{eq:exp decay}
\end{equation}

Set $\xi=\sum_{j=1}^{k}\zeta_{j}$. Since $\zeta_{1},...,\zeta_{k}$
are eigenvectors of $A^{-1}$ corresponding to distinct eigenvalues,
they are independent. In particular $\xi\ne0$, and $\zeta_{j_{1}}\ne\zeta_{j_{2}}$
for $1\le j_{1}<j_{2}\le k$. Since $\{\theta_{1},...,\theta_{k}\}$
is a P.V. $k$-tuple, for every $1\le j_{1}\le k$ there exists $1\le j_{2}\le k$
with $\theta_{j_{2}}=\overline{\theta_{j_{1}}}$. By Lemma \ref{lem:zeta equal}
this implies $\zeta_{j_{2}}=\overline{\zeta_{j_{1}}}$, which shows
that $\xi\in\mathbb{R}^{d}$. From (\ref{eq:exp decay}) it follows
that for all $1\le i\le\ell$, $V\in N_{0}$ and $b\in\mathbb{Z}_{\ge0}$,
\begin{equation}
\Vert\left\langle Va_{i},B^{-b}\xi\right\rangle \Vert=\Vert\sum_{j=1}^{k}\theta_{j}^{b}\left\langle Va_{i},\zeta_{j}\right\rangle \Vert\le C\delta^{b}\:.\label{eq:exp decay in prod}
\end{equation}

Set,
\[
\Delta:=\{(p_{1},...,p_{\ell})\in[0,1]^{\ell}\::\:p_{1}+...+p_{\ell}=1\}\:.
\]
For $(p_{i})_{i=1}^{\ell}=p\in\Delta$ let $\mu_{p}$ be the self-similar
measure corresponding to $\Phi$ and $p$. Additionally, set
\[
q_{p}:=\sum_{i=1}^{\ell}p_{i}\delta_{g_{i}}\in\mathcal{M}(G),
\]
let $X_{p,1},X_{p,2},...$ be i.i.d. $G$-valued random elements with
distribution $q_{p}$, and write $\lambda_{p}:=\mathbb{E}[\psi X_{p,1}]$.
For $g\in G$ set
\[
\rho_{p}(g)=\lambda_{p}^{-1}\mathbb{P}\{\psi X_{p,1}>\psi g\ge0\},
\]
and write $\nu_{p}$ in place of $\rho_{p}\:d\mathbf{m}_{G}$. Note
that $\nu_{p}\in\mathcal{M}(G)$.

Let $m\ge1$ be large with respect to $\delta$ and $C$. Let $f:\Delta\rightarrow\mathbb{C}$
be such that,
\[
f(p)=\int\widehat{\mu_{p}}(2\pi(B^{-m}\xi).g)\:d\nu_{p}(g)\;\text{ for }p\in\Delta\:.
\]
Let $(1,0,...,0)=:e_{1}\in\Delta$, then $\mu_{e_{1}}$ is unique
member of $\mathcal{M}(\mathbb{R}^{d})$ which satisfies $\mu_{e_{1}}=\varphi_{1}\mu_{e_{1}}$.
Since $a_{1}=0$, this relation is also satisfied by $\delta_{0}$,
where $\delta_{0}$ is the Dirac mass centred at $0$. This implies
that $\mu_{e_{1}}=\delta_{0}$, and so $f(e_{1})=1$. It is easy to
see that $f$ is continuous, and so there exists $(p_{1},...,p_{\ell})=p\in\Delta$
with $|f(p)|\ge1/2$ and $p_{i}>0$ for $1\le i\le\ell$. Fix this
$p$ until the end of the proof. We shall show that $\mu_{p}$ is
non-Rajchman. Since $p$ is positive, this will complete the proof
of the proposition.

Let $n\ge1$ be large with respect to $m$ and $p$. Set,
\[
\mathcal{W}_{n}:=\{i_{1},...,i_{s}\in\Lambda^{*}\::\:\psi(g_{i_{1}...i_{s}})\ge\beta n>\psi(g_{i_{1}...i_{s-1}})\}\:.
\]
As noted in the beginning of the present section, by Lemma \ref{lem:proper homo from R}
we may assume that $\gamma_{\beta}=h$. Thus, for $y\in\mathbb{R}^{d}$
\[
B^{-1}y=2^{\beta}Uy=y.h^{-1}=y.\gamma_{-\beta}\:.
\]
Additionally,
\[
r_{w}U_{w}^{-1}y=y.g_{w}\text{ for }w\in\Lambda^{*}\text{ and }y\in\mathbb{R}^{d}\:.
\]
Hence, since $\mathcal{W}_{n}$ is a minimal cut-set for $\Lambda^{*}$,
\begin{eqnarray}
\widehat{\mu_{p}}(2\pi B^{-m-n}\xi) & = & \sum_{w\in\mathcal{W}_{n}}p_{w}\int e^{2\pi i\left\langle B^{-m-n}\xi,\varphi_{w}(x)\right\rangle }\:d\mu_{p}(x)\nonumber \\
 & = & \sum_{w\in\mathcal{W}_{n}}p_{w}e^{2\pi i\left\langle B^{-m-n}\xi,\varphi_{w}(0)\right\rangle }\widehat{\mu_{p}}(2\pi(B^{-m}\xi).(\gamma_{-n\beta}g_{w}))\:.\label{eq:dev of fur}
\end{eqnarray}

Let,
\[
f_{n}(p):=\sum_{w\in\mathcal{W}_{n}}p_{w}\widehat{\mu_{p}}(2\pi(B^{-m}\xi).(\gamma_{-n\beta}g_{w}))\:.
\]
For $s\in\mathbb{Z}_{\ge1}$ set $Y_{s}:=X_{p,1}\cdot...\cdot X_{p,s}$,
and let
\[
\tau_{\beta}(n):=\inf\{s\in\mathbb{Z}_{\ge1}\::\:\psi Y_{s}\ge\beta n\}\:.
\]
Observe that,
\[
f_{n}(p)=\mathbb{E}\left[\widehat{\mu_{p}}(2\pi(B^{-m}\xi).(\gamma_{-n\beta}Y_{\tau_{\beta}(n)}))\right]\:.
\]
Additionally, since $\mathrm{supp}(q_{p})=\{g_{i}\}_{i=1}^{\ell}$
and since $G$ is generated by $\{g_{i}\}_{i=1}^{\ell}$, the measure
$q_{p}$ is adapted. Thus, by Proposition \ref{prop:conv in dist}
and by taking $n$ to be large enough with respect to $m$ and $p$,
we may assume that $|f_{n}(p)|\ge|f(p)|-\frac{1}{4}\ge\frac{1}{4}$.

By (\ref{eq:rep of riUi}) it follows that for every $1\le i\le\ell$
there exists $V_{i}\in N_{0}$ so that $r_{i}U_{i}=A^{l_{i}}V_{i}$.
Hence for $i_{1}...i_{s}=w\in\mathcal{W}_{n}$,
\[
\varphi_{w}(0)=\sum_{j=1}^{s}r_{i_{1}...i_{j-1}}U_{i_{1}...i_{j-1}}a_{i_{j}}=\sum_{j=1}^{s}A^{l_{i_{1}}}V_{i_{1}}...A^{l_{i_{j-1}}}V_{i_{j-1}}a_{i_{j}}\:.
\]
From $N_{0}A=N_{0}A$ it follows that there exist $V_{w,1},...,V_{w,s}\in N_{0}$
so that
\[
\varphi_{w}(0)=\sum_{j=1}^{s}A^{\sigma_{w,j}}V_{w,j}a_{i_{j}},
\]
where $\sigma_{w,j}:=l_{i_{1}}+...+l_{i_{j-1}}$ for $1\le j\le s$.
From (\ref{eq:exp decay in prod}) we now get that for all $b\in\mathbb{Z}_{\ge\sigma_{w,s}}$,
\begin{eqnarray*}
\Vert\left\langle B^{-b}\xi,\varphi_{w}(0)\right\rangle \Vert & \le & \sum_{j=1}^{s}\Vert\left\langle B^{\sigma_{w,j}-b}\xi,V_{w,j}a_{i_{j}}\right\rangle \Vert\\
 & \le & C\sum_{j=1}^{s}\delta^{b-\sigma_{w,j}}\le C\sum_{j=b-\sigma_{w,s}}^{\infty}\delta^{j}=\frac{C}{1-\delta}\delta^{b-\sigma_{w,s}}\:.
\end{eqnarray*}
Additionally, since $w\in\mathcal{W}_{n}$
\[
\beta n>\psi(g_{i_{1}...i_{s-1}})=\sum_{j=1}^{s-1}\psi(g_{i_{j}})=\beta\sigma_{w,s},
\]
which implies,
\[
\Vert\left\langle B^{-m-n}\xi,\varphi_{w}(0)\right\rangle \Vert\le\frac{C\delta^{m}}{1-\delta}\:.
\]
Hence,
\[
\left|1-e^{2\pi i\left\langle B^{-m-n}\xi,\varphi_{w}(0)\right\rangle }\right|\le\frac{2\pi C\delta^{m}}{1-\delta}\text{ for }w\in\mathcal{W}_{n}\:.
\]
Now from this, from (\ref{eq:dev of fur}) and by assuming that $m$
is large enough with respect to $\delta$ and $C$,
\[
|\widehat{\mu_{p}}(2\pi B^{-m-n}\xi)-f_{n}(p)|\le\sum_{w\in\mathcal{W}_{n}}p_{w}\left|1-e^{2\pi i\left\langle B^{-m-n}\xi,\varphi_{w}(0)\right\rangle }\right|\le\frac{1}{8}\:.
\]
Since $|f_{n}(p)|\ge1/4$, it follows that $|\widehat{\mu_{p}}(2\pi B^{-m-n}\xi)|\ge1/8$.
Note that this inequality holds for all sufficiently large $n\ge1$.
Since $\xi\ne0$, this shows that $\mu_{p}$ is not a Rajchman measure,
which completes the proof of the proposition.
\end{proof}

\section{\label{sec:Proof-of-the main}Proof of the main result}

In this section we prove Theorem \ref{thm:main}, which we now restate.
As always, recall that $\Phi=\{\varphi_{i}(x)=r_{i}U_{i}x+a_{i}\}_{i=1}^{\ell}$
is an affinely irreducible self-similar IFS on $\mathbb{R}^{d}$
\begin{thm*}
There exists a probability vector $p=(p_{i})_{i=1}^{\ell}>0$ such
that the self-similar measure corresponding to $\Phi$ and $p$ is
non-Rajchman if and only if there exists a linear subspace $\mathbb{V}\subset\mathbb{R}^{d}$,
with $d':=\dim\mathbb{V}>0$ and $U_{i}(\mathbb{V})=\mathbb{V}$ for
$1\le i\le\ell$, and an isometry $S:\mathbb{V}\rightarrow\mathbb{R}^{d'}$
so that the following conditions are satisfied.
\begin{enumerate}
\item \label{enu:first cond}For $1\le i\le\ell$ let $U_{i}'\in O(d')$
and $a_{i}'\in\mathbb{R}^{d'}$ be with $S\circ\pi_{\mathbb{V}}\circ\varphi_{i}\circ S^{-1}(x)=r_{i}U_{i}'x+a_{i}'$.
Let $\mathbf{H}\subset GL_{d'}(\mathbb{R})$ be the group generated
by $\{r_{i}U_{i}'\}_{i=1}^{\ell}$, and set $\mathbf{N}:=\mathbf{H}\cap O(d')$.
Then $\mathbf{N}$ is finite, $\mathbf{N}\triangleleft\mathbf{H}$
and $\mathbf{H}/\mathbf{N}$ is cyclic.
\item \label{enu:second cond}For every contracting $A\in\mathbf{H}$ with
$\{A^{n}\mathbf{N}\}_{n\in\mathbb{Z}}=\mathbf{H}/\mathbf{N}$, there
exist $k\ge1$, $\theta_{1},...,\theta_{k}\in\mathbb{C}$ and $\zeta_{1},...,\zeta_{k}\in\mathbb{C}^{d'}\setminus\{0\}$,
so that
\begin{enumerate}
\item $\{\theta_{1},...,\theta_{k}\}$ is a P.V. $k$-tuple;
\item $A^{-1}\zeta_{j}=\theta_{j}\zeta_{j}$ for $1\le j\le k$;
\item for every $1\le i\le\ell$ and $V\in\mathbf{N}$ there exists $P_{i,V}\in\mathbb{Q}[X]$
so that $\left\langle Va_{i}',\zeta_{j}\right\rangle =P_{i,V}(\theta_{j})$
for all $1\le j\le k$.
\end{enumerate}
\end{enumerate}
\end{thm*}
We shall need following lemma.
\begin{lem}
\label{lem:aff irr of tag}Let $\mathbb{V}\subset\mathbb{R}^{d}$
be a linear subspace with $d':=\dim\mathbb{V}>0$ and $U_{i}(\mathbb{V})=\mathbb{V}$
for $1\le i\le\ell$, and let $S:\mathbb{V}\rightarrow\mathbb{R}^{d'}$
be an isometry. For $1\le i\le\ell$ set $\varphi_{i}':=S\circ\pi_{\mathbb{V}}\circ\varphi_{i}\circ S^{-1}$,
and write $\Phi'$ for the self-similar IFS $\{\varphi_{i}'\}_{i=1}^{\ell}$.
Then $\Phi'$ is affinely irreducible.
\end{lem}

\begin{proof}
Let $\mathbb{W}$ be an affine subspace of $\mathbb{R}^{d'}$ so that
$\varphi_{i}'(\mathbb{W})=\mathbb{W}$ for $1\le i\le\ell$. Set $\mathbb{W}_{0}:=S^{-1}(\mathbb{W})$
and let $1\le i\le\ell$, then $\mathbb{W}_{0}\subset\mathbb{V}$
and $\pi_{\mathbb{V}}\circ\varphi_{i}(\mathbb{W}_{0})=\mathbb{W}_{0}$.
From this, $U_{i}(\mathbb{V})=\mathbb{V}$ and $U_{i}(\mathbb{V}^{\perp})=\mathbb{V}^{\perp}$,
it follows that for $x\in\mathbb{W}_{0}$ and $y\in\mathbb{V}^{\perp}$
\begin{eqnarray*}
\varphi_{i}(x+y) & = & r_{i}U_{i}x+\pi_{\mathbb{V}}a_{i}+r_{i}U_{i}y+\pi_{\mathbb{V}^{\perp}}a_{i}\\
 & = & \pi_{\mathbb{V}}\circ\varphi_{i}(x)+\pi_{\mathbb{V}^{\perp}}\circ\varphi_{i}(y)\in\mathbb{W}_{0}+\mathbb{V}^{\perp},
\end{eqnarray*}
and so $\varphi_{i}(\mathbb{W}_{0}+\mathbb{V}^{\perp})=\mathbb{W}_{0}+\mathbb{V}^{\perp}$.
Since this holds for every $1\le i\le\ell$ and $\Phi$ is affinely
irreducible, it follows that $\mathbb{W}_{0}+\mathbb{V}^{\perp}=\mathbb{R}^{d}$.
Since $\mathbb{W}_{0}\subset\mathbb{V}$, we must have $\mathbb{W}_{0}=\mathbb{V}$.
Hence $\mathbb{W}=\mathbb{R}^{d'}$, which shows that $\Phi'$ is
affinely irreducible.
\end{proof}
\begin{proof}[Proof of Theorem \ref{thm:main}]
Suppose first that there exist a linear subspace $\mathbb{V}$ and
an isometry $S$ as in the statement of the theorem. For $1\le i\le\ell$
set $\varphi_{i}':=S\circ\pi_{\mathbb{V}}\circ\varphi_{i}\circ S^{-1}$,
and let $\Phi':=\{\varphi_{i}'\}_{i=1}^{\ell}$. For every $1\le i\le\ell$
and $x\in\mathbb{R}^{d'}$ we have $\varphi_{i}'(x)=r_{i}U_{i}'x+a_{i}'$.
By Lemma \ref{lem:aff irr of tag} it follows that $\Phi'$ is affinely
irreducible.

Let $G'\subset\mathbb{R}\times O(d')$ be the group generated by $\{(\log r_{i}^{-1},U_{i}')\}_{i=1}^{\ell}$.
Let $A\in\mathbf{H}$ be contracting and with $\{A^{n}\mathbf{N}\}_{n\in\mathbb{Z}}=\mathbf{H}/\mathbf{N}$,
and let $\beta>0$ and $U\in O(d')$ be with $A=2^{-\beta}U$. From
$\{A^{n}\mathbf{N}\}_{n\in\mathbb{Z}}=\mathbf{H}/\mathbf{N}$ it follows
that,
\[
G'=\{(n\beta,U^{n}V)\::\:n\in\mathbb{Z}\text{ and }V\in\mathbf{N}\}\:.
\]
Since $\mathbf{N}$ is finite, $G'$ is easily seen to be discrete
and closed in $\mathbb{R}\times O(d')$. By Proposition \ref{prop:conv disc case}
and by condition (\ref{enu:second cond}) in the statement of the
theorem, it now follows that there exists a probability vector $p=(p_{i})_{i=1}^{\ell}>0$
so that the self-similar measure $\mu'\in\mathcal{M}(\mathbb{R}^{d'})$
corresponding to $\Phi'$ and $p$ is non-Rajchman.

Let $\mu\in\mathcal{M}(\mathbb{R}^{d})$ be the self-similar measure
corresponding to $\Phi$ and $p$. Since for $1\le i\le\ell$ we have
$U_{i}(\mathbb{V}^{\perp})=\mathbb{V}^{\perp}$, it follows that for
$x\in\mathbb{R}^{d}$
\[
\pi_{\mathbb{V}}\varphi_{i}(x)=\pi_{\mathbb{V}}(r_{i}U_{i}\pi_{\mathbb{V}}x+r_{i}U_{i}\pi_{\mathbb{V}^{\perp}}x)+\pi_{\mathbb{V}}a_{i}=\pi_{\mathbb{V}}\circ\varphi_{i}\circ\pi_{\mathbb{V}}(x)\:.
\]
From this and by the self-similarity of $\mu$,
\[
S\pi_{\mathbb{V}}\mu=\sum_{i=1}^{\ell}p_{i}\cdot S\circ\pi_{\mathbb{V}}\circ\varphi_{i}\circ\pi_{\mathbb{V}}\mu=\sum_{i=1}^{\ell}p_{i}\cdot\varphi_{i}'\circ S\circ\pi_{\mathbb{V}}\mu\:.
\]
Since $\mu'$ is the unique member of $\mathcal{M}(\mathbb{R}^{d'})$
which satisfies the relation
\[
\mu'=\sum_{i=1}^{\ell}p_{i}\cdot\varphi_{i}'\mu',
\]
it follows that $\mu'=S\pi_{\mathbb{V}}\mu$. From this and since
$\mu'$ is non-Rajchman, we get that there exist $\epsilon>0$ and
$\xi_{1},\xi_{2},...\in\mathbb{V}$ so that $|\xi_{n}|\overset{n}{\rightarrow}\infty$
and $|\widehat{\pi_{\mathbb{V}}\mu}(\xi_{n})|>\epsilon$. Since $\widehat{\pi_{\mathbb{V}}\mu}(\xi)=\widehat{\mu}(\xi)$
for $\xi\in\mathbb{V}$, this shows that $\mu$ is also non-Rajchman,
which completes the proof of the first direction of the theorem.

Suppose next that there exists a probability vector $p=(p_{i})_{i=1}^{\ell}>0$
so that the self-similar measure $\mu$ corresponding to $\Phi$ and
$p$ is non-Rajchman. By Proposition \ref{prop:case Psi(G)=00003DR}
it follows that $\psi(G)\ne\mathbb{R}$. Recall that $G_{0}$ denotes
the connected component of $G$ containing the identity. Let $\mathbb{V}$
be the linear subspace of $\mathbb{R}^{d}$ consisting of all $x\in\mathbb{R}^{d}$
so that $x.g=x$ for all $g\in G_{0}$. By Proposition \ref{prop:cont case psi(G) not R}
and since $\mu$ is non-Rajchman, we have $d':=\dim\mathbb{V}>0$.
By Lemma \ref{lem:G inv subspaces},
\begin{equation}
U_{i}(\mathbb{V})=\mathbb{V}\text{ and }U_{i}(\mathbb{V}^{\perp})=\mathbb{V}^{\perp}\text{ for all }1\le i\le\ell\:.\label{eq:inv subs}
\end{equation}

The map $\pi_{\mathbb{V}}\varphi_{1}|_{\mathbb{V}}$ is a strict contraction
of $\mathbb{V}$, and so there exists $y\in\mathbb{V}$ with $\pi_{\mathbb{V}}\varphi_{1}(y)=y$.
Let $S:\mathbb{V}\rightarrow\mathbb{R}^{d'}$ be an isometry with
$Sy=0$. For $1\le i\le\ell$ set $\varphi_{i}':=S\circ\pi_{\mathbb{V}}\circ\varphi_{i}\circ S^{-1}$,
and let $U_{i}'\in O(d')$ and $a_{i}'\in\mathbb{R}^{d'}$ be with
$\varphi_{i}'(x)=r_{i}U_{i}'x+a_{i}'$ for $x\in\mathbb{R}^{d'}$.
Let $\mathbf{H}$ be the smallest closed subgroup of $GL_{d'}(\mathbb{R})$
containing $\{r_{i}U_{i}'\}_{i=1}^{\ell}$.

Since $S$ is also an affine map, there exists a linear isometry $L:\mathbb{V}\rightarrow\mathbb{R}^{d'}$
so that $Sx=Lx-Ly$ for $x\in\mathbb{V}$. From (\ref{eq:inv subs})
it follows that $L\circ U\circ L^{-1}\in O(d')$ for every $(t,U)\in G$.
Note that $U_{i}'=L\circ U_{i}\circ L^{-1}$ for $1\le i\le\ell$.
For $(t,U)\in G$ set $F(t,U)=2^{-t}L\circ U\circ L^{-1}$, so that
$F:G\rightarrow GL_{d'}(\mathbb{R})$ is a continuous homomorphism.
It is easy to verify that $F$ is a proper map, which implies that
$F$ is a closed map. From this and since the group generated by $\{r_{i}U_{i}'\}_{i=1}^{\ell}$
is dense in $F(G)$, it follows that $\mathbf{H}=F(G)$ and that $F$
descends to an isomorphism of topological groups from $G/\ker F$
onto $\mathbf{H}$.

Since $\psi(G)\ne\mathbb{R}$, we have $G_{0}\subset\{0\}\times O(d)$.
From this and by the definition of $\mathbb{V}$, it follows that
$G_{0}\subset\ker F$. Since $G_{0}$ is an open subgroup of $G$,
it follows that $\ker F$ is also an open subgroup of $G$. This implies
that $G/\ker F$ is discrete, and so that $\mathbf{H}$ is also discrete.
From this and by the definition of $\mathbf{H}$, it follows that
$\mathbf{H}$ is equal to the group generated by $\{r_{i}U_{i}'\}_{i=1}^{\ell}$
(and not just to the closed subgroup generated by these elements).

Set $\mathbf{N}:=\mathbf{H}\cap O(d')$. Since $\mathbf{N}$ is the
kernel of the homomorphism taking $rU\in\mathbf{H}$ to $r$, where
$U\in O(d')$ and $r>0$, we have $\mathbf{N}\triangleleft\mathbf{H}$.
Since $\mathbf{H}$ is closed in $GL_{d'}(\mathbb{R})$ and $O(d')$
is compact, it follows that $\mathbf{N}$ is compact. From this and
since $\mathbf{H}$ is discrete, it follows that $\mathbf{N}$ is
finite. Since $\psi(G)\ne\mathbb{R}$ and $\mathbf{H}=F(G)$, there
exists $A\in\mathbf{H}$ so that $\Vert A\Vert<1$ and $\Vert A\Vert\ge\Vert B\Vert$
for all $B\in\mathbf{H}$ with $\Vert B\Vert<1$, where $\Vert\cdot\Vert$
is the operator norm here. It is now obvious that $\{A^{n}\mathbf{N}\}_{n\in\mathbb{Z}}=\mathbf{H}/\mathbf{N}$,
which shows that condition (\ref{enu:first cond}) in the statement
of the theorem is satisfied.

We turn to prove that condition (\ref{enu:second cond}) is also satisfied.
First we show that,
\begin{equation}
\underset{M\rightarrow\infty}{\lim}\sup\{|\widehat{\pi_{\mathbb{V}}\mu}(\xi)|\::\:\xi\in\mathbb{V}\text{ and }|\xi|\ge M\}>0\:.\label{eq:proj non-raj}
\end{equation}
Since $\mu$ is non-Rajchman, there exists $\epsilon_{0}>0$ so that
\[
\underset{|\xi|\rightarrow\infty}{\limsup}\:|\widehat{\mu}(\xi)|>\epsilon_{0}\:.
\]
Let $0<\epsilon<1$ be small with respect to $\Phi$, $p$ and $\epsilon_{0}$,
let $R>1$ be large with respect to $\epsilon$, and let $\xi\in\mathbb{R}^{d}$
be with $|\xi|>R$ and $|\widehat{\mu}(\xi)|>\epsilon_{0}$. By Proposition
\ref{prop:cont case psi(G) not R} we may assume that,
\[
|\pi_{\mathbb{V}^{\perp}}\xi|<\max\{\epsilon R/2,\epsilon|\pi_{\mathbb{V}}\xi|\}\:.
\]
If $|\pi_{\mathbb{V}^{\perp}}\xi|\ge\epsilon|\pi_{\mathbb{V}}\xi|$
then $|\pi_{\mathbb{V}^{\perp}}\xi|<\epsilon R/2$, and so
\[
R<|\pi_{\mathbb{V}}\xi|+|\pi_{\mathbb{V}^{\perp}}\xi|\le(\epsilon^{-1}+1)|\pi_{\mathbb{V}^{\perp}}\xi|<2\epsilon^{-1}(\epsilon R/2)=R,
\]
which is not possible. Hence we must have $|\pi_{\mathbb{V}^{\perp}}\xi|<\epsilon|\pi_{\mathbb{V}}\xi|\le\epsilon|\xi|$.

We may assume that $R>\epsilon^{-1/2}$, which gives $|\xi|^{-1}\epsilon^{-1/2}<1$.
Set,
\[
\mathcal{W}=\{i_{1}...i_{n}\in\Lambda^{*}\::\:r_{i_{1}...i_{n}}\le|\xi|^{-1}\epsilon^{-1/2}<r_{i_{1}...i_{n-1}}\}\:.
\]
Since $\mathcal{W}$ is a minimal cut-set,
\[
\epsilon_{0}<|\widehat{\mu}(\xi)|=\left|\sum_{w\in\mathcal{W}}p_{w}\int e^{i\left\langle \xi,\varphi_{w}(x)\right\rangle }\:d\mu(x)\right|\le\sum_{w\in\mathcal{W}}p_{w}\left|\int e^{i\left\langle r_{w}U_{w}^{-1}\xi,x\right\rangle }\:d\mu(x)\right|,
\]
and so there exists $w\in\mathcal{W}$ with $|\widehat{\mu}(r_{w}U_{w}^{-1}\xi)|>\epsilon_{0}$.
By (\ref{eq:inv subs}) and the definition of $\mathcal{W}$,
\[
|r_{w}U_{w}^{-1}\xi-\pi_{\mathbb{V}}(r_{w}U_{w}^{-1}\xi)|=|\pi_{\mathbb{V}^{\perp}}(r_{w}U_{w}^{-1}\xi)|=r_{w}|\pi_{\mathbb{V}^{\perp}}\xi|\le r_{w}\epsilon|\xi|\le\epsilon^{1/2}\:.
\]
Since $\mu$ is compactly supported, the map which takes $\eta\in\mathbb{R}^{d}$
to $\widehat{\mu}(\eta)$ is uniformly continuous. Hence, by assuming
that $\epsilon$ is sufficiently small with respect to $\Phi$, $p$
and $\epsilon_{0}$, we get
\begin{equation}
|\widehat{\pi_{\mathbb{V}}\mu}(r_{w}\pi_{\mathbb{V}}U_{w}^{-1}\xi)|=|\widehat{\mu}(r_{w}\pi_{\mathbb{V}}U_{w}^{-1}\xi)|\ge|\widehat{\mu}(r_{w}U_{w}^{-1}\xi)|-\epsilon_{0}/2>\epsilon_{0}/2\:.\label{eq:lb fur of proj}
\end{equation}
Set $r_{\mathrm{min}}:=\min_{1\le i\le\ell}\:r_{i}$, then by the
definition of $\mathcal{W}$ we have $r_{w}>r_{\mathrm{min}}|\xi|^{-1}\epsilon^{-1/2}$.
Thus,
\[
|r_{w}\pi_{\mathbb{V}}U_{w}^{-1}\xi|\ge r_{w}|\xi|-|\pi_{\mathbb{V}^{\perp}}(r_{w}U_{w}^{-1}\xi)|>r_{\mathrm{min}}\epsilon^{-1/2}-\epsilon^{1/2}\:.
\]
Since $\epsilon$ can be chosen to be arbitrarily small, the last
expression can be made arbitrarily large (while keeping $\epsilon_{0}$
fixed). This together with (\ref{eq:lb fur of proj}) gives (\ref{eq:proj non-raj}).

Set $\Phi':=\{\varphi_{i}'\}_{i=1}^{\ell}$, where recall that $\varphi_{i}':=S\circ\pi_{\mathbb{V}}\circ\varphi_{i}\circ S^{-1}$
for $1\le i\le\ell$. By Lemma \ref{lem:aff irr of tag} it follows
that $\Phi'$ is affinely irreducible. Since $\pi_{\mathbb{V}}\varphi_{1}(y)=y$
and $Sy=0$, we have $a_{1}'=\varphi_{1}'(0)=0$. As in the proof
of the first direction of the theorem, it holds that $S\pi_{\mathbb{V}}\mu$
is the self-similar measure corresponding to $\Phi'$ and $p$. From
(\ref{eq:proj non-raj}) it clearly follows that $S\pi_{\mathbb{V}}\mu$
is non-Rajchman. Since the closed subgroup generated by $\{r_{i}U_{i}'\}_{i=1}^{\ell}$
is discrete, it follows that all of the assumptions in Proposition
\ref{prop:main disc case} are satisfied for the IFS $\Phi'$. This
implies that condition (\ref{enu:second cond}) in the statement of
the theorem holds, which completes the proof.
\end{proof}

\section*{\textbf{Acknowledgment}}

This research was supported by the Herchel Smith Fund at the University
of Cambridge. I would like to thank Han Yu for helpful discussions
during the preparation of this work. I would also like to thank Amir
Algom for helpful remarks.

\bibliographystyle{plain}
\bibliography{bibfile}

\begin{thebibliography}{10}

\bibitem{AHW}
A.~Algom, F.~Rodriguez~Hertz, and Z.~Wang.
\newblock Pointwise normality and {F}ourier decay for self-conformal measures,
  2021.
\newblock arXiv:2012.06529.

\bibitem{BDGPS}
M.-J. Bertin, A.~Decomps-Guilloux, M.~Grandet-Hugot, M.~Pathiaux-Delefosse, and
  J.-P. Schreiber.
\newblock {\em Pisot and {S}alem numbers}.
\newblock Birkh\"{a}user Verlag, Basel, 1992.
\newblock With a preface by D. W. Boyd.

\bibitem{BP}
C.~J. Bishop and Y.~Peres.
\newblock {\em Fractals in probability and analysis}, volume 162.
\newblock Cambridge University Press, 2017.

\bibitem{Bo}
D.~W. Boyd.
\newblock Irreducible polynomials with many roots of maximal modulus.
\newblock {\em Acta Arith.}, 68(1):85--88, 1994.

\bibitem{Br}
J.~Br\'{e}mont.
\newblock Self-similar measures and the {R}ajchman property.
\newblock {\em To appear in Ann. H. Lebesgue}, 2020.
\newblock arXiv:1910.03463.

\bibitem{Bu}
Y.~Bugeaud.
\newblock {\em Distribution modulo one and Diophantine approximation}, volume
  193.
\newblock Cambridge University Press, 2012.

\bibitem{BDGHU}
D.~Buraczewski, E.~Damek, Y.~Guivarc'h, A.~Hulanicki, and R.~Urban.
\newblock Tail-homogeneity of stationary measures for some multidimensional
  stochastic recursions.
\newblock {\em Probab. Theory Related Fields}, 145(3-4):385, 2009.

\bibitem{Ca}
D.~G. Cantor.
\newblock On sets of algebraic integers whose remaining conjugates lie in the
  unit circle.
\newblock {\em Trans. Amer. Math. Soc.}, 105(3):391--406, 1962.

\bibitem{CM}
C.~Christopoulos and J.~McKee.
\newblock Galois theory of salem polynomials.
\newblock In {\em Math. Proc. Cambridge Philos. Soc.}, volume 148, page~47.
  Cambridge University Press, 2010.

\bibitem{Er1}
P.~Erd\H{o}s.
\newblock On a family of symmetric {B}ernoulli convolutions.
\newblock {\em Amer. J. Math.}, 61(4):974--976, 1939.

\bibitem{Er2}
P.~Erd\H{o}s.
\newblock On the smoothness properties of a family of bernoulli convolutions.
\newblock {\em Amer. J. Math.}, 62(1):180--186, 1940.

\bibitem{Fer}
R.~Ferguson.
\newblock Irreducible polynomials with many roots of equal modulus.
\newblock {\em Acta Arith.}, 78(3):221--225, 1997.

\bibitem{Ha}
B.~Hall.
\newblock {\em Lie groups, Lie algebras, and representations: an elementary
  introduction}, volume 222.
\newblock Springer, second edition, 2015.

\bibitem{Ho}
M.~Hochman.
\newblock On self-similar sets with overlaps and inverse theorems for entropy
  in {$\mathbb{R}^d$}.
\newblock {\em To appear in Mem. Amer. Math. Soc.}, 2015.
\newblock arXiv:1503.09043.

\bibitem{Hut}
J.~E. Hutchinson.
\newblock Fractals and self-similarity.
\newblock {\em Indiana Univ. Math. J.}, 30(5):713--747, 1981.

\bibitem{Kah}
J.-P. Kahane.
\newblock Sur la distribution de certaines s{\'e}ries al{\'e}atoires.
\newblock In {\em Colloque de Th{\'e}orie des Nombres (Univ. Bordeaux,
  Bordeaux, 1969)}, pages 119--122. Soc. Math. France, Paris, 1971.

\bibitem{Ko}
I.~K\"{o}rnyei.
\newblock On a theorem of {P}isot.
\newblock {\em Publ. Math. Debrecen}, 34(3-4):169--179, 1987.

\bibitem{Li}
J.~Li.
\newblock Decrease of fourier coefficients of stationary measures.
\newblock {\em Math. Ann.}, 372(3):1189--1238, 2018.

\bibitem{LS}
J.~Li and T.~Sahlsten.
\newblock Trigonometric series and self-similar sets.
\newblock {\em To appear in J. Eur. Math. Soc.}, 2019.
\newblock arXiv:1902.00426.

\bibitem{LS2}
J.~Li and T.~Sahlsten.
\newblock Fourier transform of self-affine measures.
\newblock {\em Adv. Math.}, 374:107349, 2020.

\bibitem{Ly}
R.~Lyons.
\newblock Seventy years of rajchman measures.
\newblock {\em J. Fourier Anal. Appl.}, 1:363--378, 1995.

\bibitem{Ma}
C.~Mauduit.
\newblock Caract{\'e}risation des ensembles normaux substitutifs.
\newblock {\em Invent. Math.}, 95(1):133--147, 1989.

\bibitem{Pi}
C.~Pisot.
\newblock La r{\'e}partition modulo 1 et les nombres alg{\'e}briques.
\newblock {\em Ann. Scuola Norm. Super. Pisa Cl. Sci.}, 7(3-4):205--248, 1938.

\bibitem{Re}
D.~Revuz.
\newblock {\em Markov chains}, volume~11 of {\em North-Holland Mathematical
  Library}.
\newblock North-Holland Publishing Co., Amsterdam, second edition, 1984.

\bibitem{SS}
T.~Sahlsten and C.~Stevens.
\newblock Fourier transform and expanding maps on {C}antor sets, 2020.
\newblock arXiv:2009.01703.

\bibitem{Sa}
R.~Salem.
\newblock Sets of uniqueness and sets of multiplicity.
\newblock {\em Trans. Amer. Math. Soc.}, 54(2):218--228, 1943.

\bibitem{So}
B.~Solomyak.
\newblock Fourier decay for self-similar measures, 2019.
\newblock arXiv:1906.12164.

\bibitem{St}
C.~J. Stone.
\newblock Infinite particle systems and multi-dimensional renewal theory.
\newblock {\em J. Math. Mech.}, 18(3):201--227, 1968.

\bibitem{VY}
P.~Varj\'{u} and H.~Yu.
\newblock Fourier decay of self-similar measures and self-similar sets of
  uniqueness.
\newblock {\em To appear in Anal. PDE}, 2020.
\newblock arXiv:2004.09358.

\bibitem{Wa}
P.~Walters.
\newblock {\em An introduction to ergodic theory}, volume~79.
\newblock Springer Science \& Business Media, 2000.

\end{thebibliography}

$\newline$$\newline$\textsc{Centre for Mathematical Sciences,\newline Wilberforce Road, Cambridge CB3 0WA, UK}$\newline$$\newline$\textit{E-mail: }
\texttt{ariel.rapaport2@gmail.com}
\end{document}